\documentclass[a4paper, abstract]{scrartcl}
\usepackage[scale=0.75]{geometry}

\usepackage{setspace,graphicx,tikz,tabularx} 

\usepackage{amsmath, amsthm, bm}

\usepackage{textcomp}
\usepackage[semibold]{libertine}
\usepackage[libertine, upint]{newtxmath}

\usepackage[scr=euler, scrscaled = 1.0, cal = stixplain, bb = dsfontserif]{mathalpha}

\usepackage[utf8]{inputenc} 
\usepackage[T1]{fontenc} 
\usepackage[ngerman,english]{babel}
\usepackage[draft=false,babel,tracking=true,kerning=true,spacing=true]{microtype} 

\usepackage{xparse,upgreek,xcolor,csquotes,float,array,subcaption,relsize, nicefrac, etoolbox,lastpage, accents,enumerate}
\usepackage[colorinlistoftodos]{todonotes}

\usepackage[unicode=true]{hyperref}					
\hypersetup{
     colorlinks   = true,
     linkcolor    = blue,
     citecolor	  = blue,
     urlcolor	  = blue,
     pdfauthor={Sascha Robert Gaudlitz},
}	

\usepackage[backend=biber, 
sorting=nyt,
bibstyle = authoryear,
citestyle=authoryear-comp,
maxbibnames = 99,
giveninits=true,
date = year,
doi = false,
isbn = false,
url = false,
]{biblatex}

\newcommand{\citep}{\autocite}
\newcommand{\citet}{\textcite}

\AtBeginRefsection{\GenRefcontextData{sorting=ynt}}
\AtEveryCite{\localrefcontext[sorting=ynt]}

\renewbibmacro{in:}{%
  \ifentrytype{article}
    {}
    {\bibstring{in}%
     \printunit{\intitlepunct}}}

\DeclareFieldFormat{pages}{#1}		

	\usepackage{scrlayer-scrpage}
	\automark{section}
	
	\setkomafont{pagehead}{\footnotesize\normalfont}
	
	\usepackage{float}
	\usepackage{subcaption}	
	
	\usepackage{xurl}
\usepackage{xparse}
\usepackage{amsmath, mathtools}
	
\allowdisplaybreaks

\DeclareMathOperator{\Exists}{\exists}
\DeclareMathOperator{\Forall}{\forall}

\newcommand{\MTkillspecial}[1]{
	\begingroup%
	\catcode`\&=9%
	\let\\\relax%
	\scantokens{#1}%
	\endgroup%
}
\newcommand{\MTemptyplaceholder}{\:\cdot\:}
\DeclarePairedDelimiter\abs\lvert\rvert
\reDeclarePairedDelimiterInnerWrapper\abs{star}{%
	\mathopen{#1\vphantom{\MTkillspecial{#2}}\kern-\nulldelimiterspace\right.}%
	\ifblank{#2}{\MTemptyplaceholder}{#2}%
	\mathclose{\left.\kern-\nulldelimiterspace\vphantom{\MTkillspecial{#2}}#3}%
}

\DeclarePairedDelimiterXPP\iprodWrapper[2]{}{\langle}{\rangle}{}{
	\ifblank{#1}{\MTemptyplaceholder}{#1},
	\ifblank{#2}{\MTemptyplaceholder}{#2}
}
\NewDocumentCommand\iprod{ s o m m }{
	\IfBooleanTF {#1}
	{ \iprodWrapper*{#3}{#4} }
	{ \IfNoValueTF{#2}
		{
			\iprodWrapper{#3}{#4} 
		}
		{
			\iprodWrapper[#2]{#3}{#4}
		}
	}
}

\DeclarePairedDelimiterXPP\normWrapper[1]{}\lVert\rVert{}{\ifblank{#1}{\MTemptyplaceholder}{#1}}

\NewDocumentCommand\norm{ s o m }{
	\IfBooleanTF {#1}
	{ \normWrapper*{#3}}
	{ \IfNoValueTF{#2}
		{
			\normWrapper{#3}
		}
		{
			\normWrapper[#2]{#3}
		}
	}
}

\newcommand{\tnorm}[1]{{\left\vert\kern-0.25ex\left\vert\kern-0.25ex\left\vert #1 
    \right\vert\kern-0.25ex\right\vert\kern-0.25ex\right\vert}}

\providecommand\given{}

\DeclarePairedDelimiterX\Set[1]\{\}{%
	\renewcommand\given{\SetSymbol[\delimsize]}
	#1
}

\DeclarePairedDelimiterXPP\ProbWrapper[2]{#1}(){}{
	\renewcommand\given{\nonscript\:\delimsize\vert\nonscript\:\mathopen{}}
	#2
}

\NewDocumentCommand\prob{ s O{} O{} O{} m }{
	\ifblank {#5}{\mathrm{P}_{#2}^{#3}}
		{\IfBooleanTF {#1}
			{ \ProbWrapper*{\mathrm{P}_{#2}^{#3}}{#5} }
			{ \IfNoValueTF{#4}
				{
					\ProbWrapper{\mathrm{P}_{#2}^{#3}}{#5}
				}
				{
					\ProbWrapper[#4]{\mathrm{P}_{#2}^{#3}}{#5}
				}
			}
		}
}
\NewDocumentCommand\qrob{ s O{} O{} O{} m }{
	\ifblank {#5}{\mathrm{Q}_{#2}^{#3}}
		{\IfBooleanTF {#1}
			{ \ProbWrapper*{\mathrm{Q}_{#2}^{#3}}{#5} }
			{ \IfNoValueTF{#4}
				{
					\ProbWrapper{\mathrm{Q}_{#2}^{#3}}{#5}
				}
				{
					\ProbWrapper[#4]{\mathrm{Q}_{#2}^{#3}}{#5}
				}
			}
		}
}
\DeclarePairedDelimiterXPP\EVWrapper[2]{#1}[]{}{
	\renewcommand\given{\mathrel{}\mathclose{}\delimsize\vert\mathopen{}\mathrel{}}
	#2
}
\NewDocumentCommand\EV{ s O{} O{} O{} m }{
	\ifblank {#5}{\mathrm{E}_{#2}^{#3}}
		{\IfBooleanTF {#1}
			{ \EVWrapper*{\mathrm{E}_{#2}^{#3}}{#5} }
			{ \IfNoValueTF{#4}
				{
					\EVWrapper{\mathrm{E}_{#2}^{#3}}{#5}
				}
				{
					\EVWrapper[#4]{\mathrm{E}_{#2}^{#3}}{#5}
				}
			}
		}
}

\NewDocumentCommand{\Lip}{O{} m}{
    \ifblank {#1} {\norm{#2}_{\operatorname{Lip}}}
    {\norm{#2}_{\operatorname{Lip(#1)}}}
    }

\DeclarePairedDelimiter{\ceil}{\lceil}{\rceil}
\DeclarePairedDelimiter{\floor}{\lfloor}{\rfloor}

\renewcommand{\underbar}[1]{\underaccent{\bar}{#1}}
\providecommand{\transpose}{^{\intercal}}
\newcommand{\inv}{^{{\mathsmaller{-1}}}}

\newcommand{\identity}{\mathbb{I}}		
\newcommand{\indicator}{\mathbb{1}}		

\makeatletter
\providecommand*{\bigcupdot}{%
	\mathop{%
		\vphantom{\bigcup}%
		\mathpalette\@bigcupdot{}%
	}%
}
\newcommand*{\@bigcupdot}[2]{%
	\ooalign{%
		$\m@th#1\bigcup$\cr
		\sbox0{$#1\bigcup$}%
		\dimen@=\ht0 %
		\advance\dimen@ by -\dp0 %
		\sbox0{\scalebox{2}{$\m@th#1\cdot$}}%
		\advance\dimen@ by -\ht0 %
		\dimen@=.5\dimen@
		\hidewidth\raise\dimen@\box0\hidewidth
	}%
}
\makeatother

\newcommand{\D}{\ensuremath{\mathrm{d}}}

\let\Pr=P

\makeatletter
\newcommand{\subalign}[1]{%
	\vcenter{%
		\Let@ \restore@math@cr \default@tag
		\baselineskip\fontdimen10 \scriptfont\tw@
		\advance\baselineskip\fontdimen12 \scriptfont\tw@
		\lineskip\thr@@\fontdimen8 \scriptfont\thr@@
		\lineskiplimit\lineskip
		\ialign{\hfil$\m@th\scriptstyle##$&$\m@th\scriptstyle{}##$\crcr
			#1\crcr
		}%
	}
}
\makeatother

	\renewcommand{\phi}{\varphi}			
	\renewcommand{\epsilon}{\varepsilon}	
	\renewcommand{\theta}{\vartheta}		
	\renewcommand{\Delta}{\varDelta}		
	
\newcommand{\smallo}{
	  \mathchoice
	    {{\scriptstyle\mathcal{O}}}
	    {{\scriptstyle\mathcal{O}}}
	    {{\scriptscriptstyle\mathcal{O}}}
	    {\scalebox{.7}{$\scriptscriptstyle\mathcal{O}$}}
	  }	
	  
	\newcounter{thc}[section]
	\numberwithin{thc}{section}
	\numberwithin{equation}{section}
	
	\theoremstyle{plain}
	
		\newtheorem{corollary}[thc]{Corollary}
		\newtheorem{proposition}[thc]{Proposition}
		\newtheorem*{proposition*}{Proposition}
		\newtheorem{theorem}[thc]{Theorem}
		\newtheorem{lemma}[thc]{Lemma}
		\newtheorem*{lemma*}{Lemma}
		\newtheorem{assumption}[thc]{Assumption}
		
	\theoremstyle{definition}

		\newtheorem{example}[thc]{Example}
		
		\newtheorem{remark}[thc]{Remark}

\setkomafont{sectioning}{\normalcolor\bfseries}

\title{Nonparametric Bayesian Inference for Stochastic Reaction-Diffusion Equations}
\author{Randolf Altmeyer\footnote{Imperial College London \\ Email: r.altmeyer<at>imperial.ac.uk}\,  \& Sascha Gaudlitz\footnote{Humboldt--Universität zu Berlin \\ Email: sascha.gaudlitz<at>hu-berlin.de}}

\date{}
\begin{document}

\maketitle
\begin{abstract}
	We consider the Bayesian nonparametric estimation of a nonlinear reaction function in a reaction-diffusion stochastic partial differential equation (SPDE). The likelihood is well-defined and tractable by the infinite-dimensional Girsanov theorem, and the posterior distribution is analysed in the growing domain asymptotic. Based on a Gaussian wavelet prior, the contraction of the posterior distribution around the truth at the minimax optimal rate is proved. The analysis of the posterior distribution is complemented by a semiparametric Bernstein--von Mises theorem. The proofs rely on the sub-Gaussian concentration of spatio-temporal averages of transformations of the SPDE, which is derived by combining the Clark-Ocone formula with bounds for the derivatives of the (marginal) densities of the SPDE.
\end{abstract}


\section{Introduction}

For a time horizon $T>0$ and a bounded spatial domain $\Lambda\subset\mathbb{R}$ consider the semi-linear SPDE
\begin{equation}
    \D X_t(y) = \Updelta X_t(y)\,\D t + f(X_t(y))\,\D t + \,\D W_t(y),\quad 0\le t\le T,y\in\Lambda,\label{eq:SPDE}
\end{equation}
with a continuous and deterministic initial value $X_0$, a Lipschitz function $f:\mathbb{R}\to\mathbb{R}$ and a cylindrical Wiener process $W$ defined on $L^2(\Lambda)$. We endow the SPDE with homogeneous Neumann boundary conditions, such that $\Delta =\partial_{y}^2$ is the Neumann-Laplacian. We are interested in the nonparametric Bayesian inference on the function $f$ from a single observation of $(X_t(y)\colon (t,y)\in [0,T]\times \Lambda)$ solving \eqref{eq:SPDE}. 

Semi-linear SPDEs of type \eqref{eq:SPDE} are known as \emph{stochastic reaction-diffusion equations}. They are of fundamental importance in various fields, such as cell biology \citep{Alonso2018, Altmeyer2020b}, biodiversity \citep{groseljHowTurbulenceRegulates2015}, climate modelling \citep{debusscheStochasticChafeeInfante2013, franzkeStochasticClimateTheory2015}, phase segregation and interface modelling \citep{spohnInterfaceMotionModels1993, funakiScalingLimitStochastic1995, kharchenkoNoiseInducedPatterning2009, bertiniStochasticAllenCahn2017}, optics \citep{caraballoEffectNoiseChafeeInfante2006, kamelStochasticAnalysisSoliton2024} and spatial statistics \citep{roquesSpatialStatisticsStochastic2022}. The Laplacian describes diffusion in a medium and the nonlinear \emph{reaction function} $f$ captures local state-dependent reactions. For example, if $X_t(y)$ models to heat diffusion at a time $t$ and location $y$, then the function $f$ corresponds to temperature-dependent reactions. The choice $f(x) = x(x-a)(x-b)$ yields the Allen-Cahn equation in phase field modelling with stable points $a,b\in\mathbb{R}$ \citep{berglundIntroductionSingularStochastic2022}. The \enquote{space-time white noise} $\D W_t(y)$ represents neglected effects arising from the absence of detailed knowledge about the physical setting \citep{benzi1989stochastically}.

Inference for semi-linear SPDEs is well-studied with respect to the linear part, in particular for the diffusivity in a stochastic heat equation \citep{Huebner1995, krizCentralLimitTheorems2019, cialencoDriftEstimationDiscretely2020, Chong2020, Altmeyer2021, Hildebrandt2021a, altmeyerOptimalParameterEstimation2024}. Much less is known about estimating the reaction function $f$. References for the classical noiseless inverse problem of inferring $f$ from boundary or interior measurements of $\Lambda$ (with $\D W_t(y)=0$) can be found in \citep{cannonStructuralIdentificationUnknown1998}. In our measurement model, maximum likelihood estimation for a finite-dimensional reaction parameter is studied by \citet{Ibragimov1999,Ibragimov2000,Ibragimov2001,Ibragimov2003} in the small noise limit, by \citet{gaudlitzEstimationReactionTerm2023} for vanishing diffusivity and by  \citet{goldysParameterEstimationControlled2002} as $T\to\infty$. Nonparametric procedures were considered by  \citet{hildebrandtNonparametricCalibrationStochastic2023}, who obtain $L^2$ convergence rates for least-squares estimators in the large time limit, and by \citet{gaudlitzNonparametricEstimationReaction2023} using kernel-type estimators when diffusivity and noise level are coupled and both tend to zero. Bayesian posteriors were only studied when $f$ does not depend on the state $X_t(y)$ and the solution to \eqref{eq:SPDE} is Gaussian, see \citep{bishwalBernsteinvonMisesTheorem2002,yanBayesianInferenceGaussian2020a}.

In this paper, the log-likelihood function for the observation $(X_t(y)\colon (t,y)\in [0,T]\times \Lambda)$ is obtained from an infinite-dimensional Girsanov theorem. It is a quadratic form in $f$ such that a conjugate Gaussian process prior gives rise to a Gaussian posterior. This allows for efficient posterior sampling techniques, see Section \ref{subsec:Numerics_Balanced} for numerical examples and the approximation of the posterior from discrete observations. The posterior mean provides a point estimator for the reaction function $f$, and coincides with the Bayesian maximum-a-posterior (MAP) estimator in our Gaussian setting. Uncertainty quantification with posterior credible sets is computationally feasible.

We provide theoretical guarantees for this approach relative to a data generating \enquote{truth} $f_0$ in the large domain limit: With the time horizon $T$ fixed and  $\Lambda=\lambda\bar\Lambda$ for an open interval $\bar\Lambda\subset\mathbb{R}$, we assume that $\lambda\to\infty$. Large domain asymptotics are common in spatial statistics \citep{mardia1984, zhang2005, giordano2023nonparametric, guan2008consistent}, but our analysis of the posterior hinges on establishing that the solution to \eqref{eq:SPDE} is \emph{spatially ergodic} with respect to spatial averages of $X_t(y)$ \citep{chenSpatialErgodicitySPDEs2021} and on identifying the (spatially) invariant measure. In fact, one of our main results shows for a density $p_{f_0}$ and a large class of functions functions $g$ the convergence in probability of
\begin{equation}
	\int_0^T\frac{1}{\lambda}\int_{\lambda\bar{\Lambda}} g(X_t(y))\,\D y\D t \longrightarrow \int_{\mathbb{R}}g(z)p_{f_0}(z)\D z,\quad\lambda\to\infty.\label{eq:spatial_ergodicity}
\end{equation}

Following the established test-based approach of Bayesian nonparametrics \citep{ghosalConvergenceRatesPosterior2000} and by exploiting the spatial ergodicity, we obtain posterior contraction rates as $\lambda\to\infty$ with $T$ fixed. Consequently, we deduce frequentist minimax optimal convergence rates for the posterior mean in $L^2$-loss over Hölder balls. Using the methodology introduced by \citet{castillo2013} (see also the recent extensions by \citet{nicklNonparametricStatisticalInference2020, giordano2025} to diffusion processes), we subsequently show that the posterior satisfies a nonparametric Bernstein--von Mises theorem. It demonstrates that the rescaled posterior, centred at the posterior mean, corresponds asymptotically as $\lambda\to\infty$ to the law of a Gaussian process (defined on suitable negative Sobolev spaces) with a covariance operator that is \enquote{minimal} in an information-theoretic sense. In particular, linear functionals of the posterior mean provide efficient estimators attaining the parametric rate for the corresponding linear functionals of the reaction function. To the best of our knowledge, these are the first nonparametric Bayesian contraction results and the first Bernstein--von Mises theorem for semi-linear SPDEs. 

The test-based approach relies on relating what can be considered the \enquote{natural distance} in our statistical model, namely the path-dependent \emph{Hellinger semi-metric} 
\begin{equation*}
	h_\lambda(f,g)^2
	\coloneq \int_0^T\frac{1}{\lambda}\int_{\lambda\bar{\Lambda}} (f(X_t(y))-g(X_t(y)))^2\,\D y\D t,
\end{equation*}
to the $L^2$ distance, which was first done for (temporal) semi-martingales by  \citet{vandermeulenConvergenceRatesPosterior2006}. Since classical techniques such as local times, martingale approximations or spectral properties (cf. \citet{Kutoyants2013,vandermeulenConvergenceRatesPosterior2006,nicklNonparametricStatisticalInference2020,aeckerle-willemsConcentrationScalarErgodic2021,dalalyanAsymptoticStatisticalEquivalence2007,trottnerConcentrationAnalysisMultivariate2023,comteNonparametricEstimationFractional2019}) are not available for the process $y\mapsto X_t(y)$, we develop uniform sub-Gaussian concentration inequalities for the spatio-temporal averages in \eqref{eq:spatial_ergodicity}. Such results have recently been obtained in several works starting with \citep{chenSpatialErgodicitySPDEs2021}, but these are suboptimal for our purposes as they depend on the Lipschitz constant of the function $g$ in \eqref{eq:spatial_ergodicity}. Instead, building on ideas from \citet{gaudlitzNonparametricEstimationReaction2023}, we use tools from Malliavin Calculus to rewrite the spatial averages as martingales. Suitably integrating by parts and upper bounding the marginal densities of $X_{t}(y)$ allows us to obtain sharp sub-Gaussian concentration inequalities in terms of the $L^1$-norm of $g$. 

Let us briefly mention that our measurement model is equivalent to the small-diffusivity setting of \citet{gaudlitzEstimationReactionTerm2023, gaudlitzNonparametricEstimationReaction2023}. Using Lemma A.2 of \citet{gaudlitzNonparametricEstimationReaction2023}, the rescaled process $Y_t(y)\coloneq X_t(\lambda y)$, $y\in \bar{\Lambda}$, $0\le t\le T$ solves
\begin{equation}
    \D Y_t(y) = \nu\Updelta Y_t(y)\,\D t + f(Y_t(y))\,\D t + \sigma\,\D W_t(y),\label{eq:Y_rescaled}
\end{equation}
with known diffusivity $\nu\coloneq\lambda^{-2}$ and noise level $\sigma\coloneq\lambda^{-1/2}$. It follows that the Bayesian approach applies equivalently to $Y$ and that all results derived for $X$ as $\lambda\to\infty$ transfer to $Y$ as $\nu,\sigma\to 0$. The latter is physically relevant, cf. \citep{ Altmeyer2020b}. The specific coupling $\nu/\sigma^4\asymp 1$ ensures that the spatially ergodic averages in \eqref{eq:spatial_ergodicity} do not degenerate as $\nu,\sigma\to 0$. This is closely related to the well-known scaling assumption in the Fluctuation-Dissipation theorem \citep{pavliotisStochasticProcessesApplications2014} on the drift and volatility coefficients of a finite-dimensional diffusion process, which guarantees the existence of a temporal invariant measure. For extensions of the SPDE model \eqref{eq:SPDE} to multivariate domains, linear operators different from the Laplacian and coloured noise, as well as further comparisons to other asymptotic regimes see Section \ref{subsec:Discussion}.

The paper is structured as follows. We introduce the basic notations and the statistical model in Sections \ref{subsec:Basic_Notation}- \ref{subsec:Bayes_Setting_Posterior}. Spatial ergodicity of $X$ is explained in Section \ref{subsec:Spatial_Ergodicity}. In Sections \ref{subsec:Posterior_Contraction} and \ref{subsec:BvM}, we state the posterior contraction and the Bernstein--von Mises Theorem for the posterior. Section \ref{sec:Bayes_Setting} is concluded by a numerical illustration in Section \ref{subsec:Numerics_Balanced} and a discussion in Section \ref{subsec:Discussion}. Section \ref{sec:Proof_Subgaussian_Concentration} contains the proof of the ergodic properties stated in Section \ref{subsec:Spatial_Ergodicity}, and the proofs of the posterior contraction and the Bernstein--von Mises Theorem are carried out in Section \ref{sec:Bayes_proofContraction}. The remaining technical proofs are included in the Appendices \ref{sec:Appendix_Stats} and \ref{sec:Appendix_Probability}.

\section{Main results}\label{sec:Bayes_Setting}

\subsection{Basic notation}\label{subsec:Basic_Notation}

For a metric space $(\mathcal{X},d)$ and a function $U\colon\mathcal{X}\to\mathbb{R}$, we define the supremum and Lipschitz norms
\begin{align*}
	\norm{U}_{L^\infty(\mathcal{X})}\coloneq \sup_{x\in\mathcal{X}}\abs{U(x)},\quad \Lip[\mathcal{X}]{U} \coloneq \norm{U}_{L^\infty(\mathcal{X})}+\sup_{x,y\in\mathcal{X}, x\neq y} \frac{\abs{U(x)-U(y)}}{d(x,y)},
\end{align*}
and let $B(a,d,L)$ denote the ball in $(\mathcal{X},d)$ with centre $a\in\mathcal{X}$ and radius $L>0$. We denote $\Lip{\,\cdot\,}\coloneq \Lip[\mathbb{R}]{\,\cdot\,}$. For an open or closed set $O\subset\mathbb{R}$, $s\in\mathbb{R}$ and $0<p\le \infty$, we denote by $L^p(O)$ and $H^s(O)$ the usual Lebesgue and $L^2$-Sobolev spaces equipped with the norms $\norm{}_{L^p(O)}$ and $\norm{}_{H^s(O)}$. For $s\ge 0$ and open $O\subset\mathbb{R}$ let $C^s(O)$ be the Hölder spaces with
\begin{equation*}
   \norm{f}_{C^{s}(O)}\coloneq \max_{k=0,\dots,\floor{s}}\norm{f^{(k)}}_{L^\infty(O)} + \sup_{x,y\in O, x\neq y}\frac{\abs{f^{(\floor{s})}(x)-f^{(\floor{s})}(x)}}{\abs{x-y}^{s-\floor{s}}}<\infty,
\end{equation*}
where $f^{(k)}$, $k\in\mathbb{N}_0$, denotes the $k$-th derivative of $f$. If $O$ is closed, then the derivatives are taken one-sided. Note that functions belonging to $C^s(O)$ have bounded derivatives up to order $\floor{s}$, even if $O$ is unbounded. We write $C(O)=C^0(O)$. For brevity, we set
\begin{equation*}
    \mathcal{F}_{s,K}\coloneq B(0,\norm{}_{C^s(\mathbb{R})},L),\quad s\geq 0.
\end{equation*}
We write $a\lesssim b$ (or $b\gtrsim a$) if there exists a constant $0<C<\infty$ depending only on nonasymptotic quantities such that $a\le Cb$ and $a\sim b$ if $a\lesssim b$ and $a\gtrsim b$.
 We highlight the dependency of the constant $C$ on a particular parameter $\xi$ by $a\lesssim_\xi b$. 
 We write $\xrightarrow{\prob*{}}$ for convergence in probability and $\xrightarrow{d}$ for convergence in distribution of random variables. 
 For equality in distribution under a law $\qrob{}$ we write $\overset{\qrob{}}{=}$. 
 For $a,b\in\mathbb{R}$ let $a\wedge b\coloneqq \min(a,b)$ and $a\vee b \coloneqq \max(a,b)$. 
 For a set $M\subset\mathbb{R}$ and a point $a\in\mathbb{R}$ we define their distance as $\operatorname{dist}(a,M)\coloneqq \inf\Set{\abs{a-m}\given m\in M}$. 
 The Lebesgue-measure of a Borel-set $M\subset\mathbb{R}$ is denoted by $\abs{M}$ and its indicator function by $\indicator_M$.

\subsection{SPDE and likelihood}\label{subsec:SPDE_Likelihood}

Consider the semi-linear SPDE \eqref{eq:SPDE} with a Lipschitz-continuous reaction function $f\colon\mathbb{R}\rightarrow\mathbb{R}$ and a continuous and deterministic initial condition $X_0$. Suppose that $\Lambda$ has length $\lambda = \abs{\Lambda}\geq 1$, such that $\Lambda=\lambda\bar\Lambda$ for an open interval $\bar\Lambda\subset\mathbb{R}$ of unit length. The time horizon $T>0$ is fixed throughout, and we sometimes write $F(z)\coloneq f\circ z$ for $z\in L^2(\Lambda)$. 

The classical mild solution concept of \citet{DaPrato2014} shows that \eqref{eq:SPDE} has a unique strong solution $X$ in the probability sense taking values in $C([0,T],L^2(\Lambda))$, which is also a mild solution in the analytical sense, see Theorem 7.5 of \citet{DaPrato2014}. Equivalently, as discussed by \citet{Dalang2011}, the mild solution can be obtained pointwise at $0\leq t\leq T$ and $y\in\Lambda$ as the solution to the integral equation
\begin{equation}
    X_t(y) =\int_\Lambda G_{t-s}^\lambda (y,\eta)X_0(\eta)\,\D\eta+ \int_0^t \int_\Lambda G_{t-s}^\lambda (y,\eta)\mathcal{W}(\D \eta,\D s) +\int_0^t\int_\Lambda G_{t-s}^\lambda (y,\eta)f(X_s(\eta))\,\D \eta\D s.
    \label{eq:mild_solution}
\end{equation}
Here, $G^\lambda$ denotes the Neumann heat kernel on $\Lambda$, and $\mathcal{W}$ (formally understood as $\D W_t/\D t$) is an isonormal Gaussian process on $L^2([0,T],L^2(\Lambda))$. 

A process satisfying the SPDE \eqref{eq:SPDE} in the sense of \eqref{eq:mild_solution} is called \emph{random field solution}  \citep{walshIntroductionStochasticPartial1986}. The representation \eqref{eq:mild_solution} is convenient for our purposes as we frequently rely on pointwise properties of $X_t(y)$. Moreover, the random field solution provides a useful framework for studying the marginal (Lebesgue-) density $p_{f,\lambda}(t,y,\MTemptyplaceholder)$ of $X_t(y)$, which exists for every $0<t\le T$ and $y\in\Lambda$ according to \citep{ballyMalliavinCalculusWhite1998}. We summarise the key probabilistic properties of the random field solution in Subsection \ref{subsec:ProbabilisticProperties}. 

Let us write $X^{\lambda}\coloneq(X_t(y)\,|\, (t,y)\in [0,T]\times\Lambda)$ and denote by $\prob[f]{}$ the law of $X^{\lambda}$ on the path space $C([0,T],L^2(\Lambda))$. In particular, $\prob[0]{}$ is the law of the solution to \eqref{eq:SPDE} with $f\equiv 0$. Our goal is to estimate the reaction function $f$ from a single realisation of $X^{\lambda}\sim \prob[f]{}$. As $f$ is Lipschitz-continuous, the Girsanov theorem (e.g.\ Theorem 10.18 of \citet{DaPrato2014}) implies that the laws $\prob[f]{}$ and $\prob[0]{}$ (with the same initial condition) are equivalent on $C([0,T],L^2(\Lambda))$. Their Radon-Nikodym derivative is given by
\begin{equation*}
	\frac{\D\prob[f]{}}{\D\prob[0]{}} = \exp\left(\int_0^T \iprod{ F(X_t)}{ \D X_t - \Updelta X_t\,\D t}_{L^2(\Lambda)}-\frac{1}{2}\int_0^T \norm{ F(X_t)}_{L^2(\Lambda)}^2\,\D t\right),
\end{equation*}
where we \emph{define} $\int_0^T \iprod*{F(X_t)}{\D X_t-\Updelta X_t\,\D t}_{L^2(\Lambda)}$ as the $L^2(\prob[0]{})$-limit of 
\begin{equation}
	\sum_{k\in\mathbb{N}}\int_0^T\iprod*{F(X_t)}{e_k}_{L^2(\Lambda)}\left(\iprod*{e_k}{\D X_t}_{L^2(\Lambda)} - \iprod*{\Updelta e_k}{X_t}_{L^2(\Lambda)}\,\D t\right)\label{eq:aux_Def_Likelihood}
\end{equation}
for an orthonormal basis $(e_k)_{k\in\mathbb{N}}$ of $L^2(\Lambda)$ contained in the domain of the Neumann Laplacian. In particular, for two Lipschitz-continuous functions $f,f_0\colon\mathbb{R}\to\mathbb{R}$, the log-likelihood function relative to $\prob[f_0]{}$ is equal to
\begin{align}
	\log\left(\frac{\D\prob*[f]{}}{\D\prob*[f_0]{}}\right) & \coloneq \log\left(\frac{\D\prob*[f]{}}{\D\prob*[0]{}}\frac{\D\prob*[0]{}}{\D\prob*[f_0]{}}\right) = \sqrt{\lambda} \mathcal{M}_\lambda-\frac{\lambda}{2}h_\lambda(f,f_0)^2,\label{eq:Bayes:Likelihood_Hellinger}
\end{align}
with $\mathcal{M}_\lambda\coloneq \lambda^{-1/2}\int_0^T\iprod{F(X_s)-F_0(X_s)}{\D W_s}_{L^2(\Lambda)}$ and with the squared \emph{Hellinger semi-metric} defined as
\begin{equation}
	h_\lambda(f,f_0)^2
	\coloneq \frac{1}{\lambda}\int_0^T\int_\Lambda (f(X_t(y))-f_0(X_t(y)))^2\,\D y\D t.\label{eq:Bayes_Intro_Hellinger}
\end{equation}
Since the Hellinger semi-metric can be understood as the quadratic variation of the (temporal) martingale $\mathcal{M}_\lambda$, it characterises the log-likelihood and should be viewed as the natural distance between $f$ and $f_0$. 

The analysis of the posterior distribution defined in the next section relies on finding upper and lower bounds on the expected Hellinger semi-metric $\EV[f_0]{h_{\lambda}(f,f_0)^2}$ in terms of the standard (squared) $L^2(\mathbb{R})$-metric. As $X_t(y)$ is supported on all of $\mathbb{R}$ for all $0<t\le T$, cf. the Gaussian-type density lower bound from Proposition \ref{prop:DensityBounds}, we consider reaction functions supported on a known compact set. Moreover, we restrict to initial conditions satisfying a scaling property.

\begin{assumption}\label{assump:Parameterspace} For a fixed compact interval $\Xi\subset\mathbb{R}$ the parameter set is 
\begin{equation*}
    \Theta\coloneq \Set{f\in C^3(\mathbb{R}) \given \operatorname{supp}(f)\subset \Xi}.
\end{equation*}
There exists a function $\chi\in H^{\tilde{s}}(\bar{\Lambda})\supset C(\bar{\Lambda})$ for some $\tilde{s}>3/2$ such that the initial condition satisfies $X_0(y)=\chi(y/\abs{\Lambda})$ for all $y\in\Lambda$.
\end{assumption}

Restricting to compactly supported regression functions is common when considering random (unbounded) design as in our case \citep{vandermeulenConvergenceRatesPosterior2006,comte2007, vandervaartRatesContractionPosterior2008}. As the Hellinger semi-metric depends on $f$ and $f_0$ only through the difference $f-f_0$, we could remove the support assumption if $f_0$ was known outside of $\Xi$. Note that any $f\in\Theta$ is Lipschitz, such that the random field solution \eqref{eq:mild_solution} always exists. The additional regularity of $f$ is required to control the marginal densities of $X_t(y)$, see Proposition \ref{prop:DensityBoundedDerivative} below. The scaling assumption on the initial conditions ensures the well-behavedness of $X$ in the large domain asymptotic $\lambda\to\infty$. 

\subsection{A Bayesian posterior for the reaction function}\label{subsec:Bayes_Setting_Posterior}

We consider a Bayesian approach for inference on the unknown reaction function $f$ belonging to the parameter set $\Theta$. We begin by defining a prior distribution on $\Theta$ based on a finite-dimensional truncated Gaussian series expansion on $L^2(\Xi)$. Such priors are a common choice in Bayesian nonparametrics \citep{vandervaartRatesContractionPosterior2008} and they are conjugate in our setting, see below. For details on Gaussian processes and the prior measures constructed from them see, e.g., Section 11.2 of \citet{Vaart2017}. 

We describe the series expansion in terms of an orthonormal wavelet basis $(\psi_\mu)_{\mu}$ of $L^2(\Xi)$, where following the notational convention from \citet{cohenNumericalAnalysisWavelet2003, gobetNonparametricEstimationScalar2004} each $\mu=(j,k)$ is a multiindex  and $\abs{(j,k)}\coloneq j$, $j\in \mathbb{N}$, $k=0,\dots, \max(0,2^{j}-1)$. We assume throughout that the $\psi_{\mu}$ are $S$-regular, $S\in\mathbb{N}\cup\{0\}$, boundary-corrected Daubechies wavelets, see Section 3.9 of \citet{cohenNumericalAnalysisWavelet2003} or Section 4.3.5 of \citet{gineMathematicalFoundationsInfiniteDimensional2015} for their construction. In particular, the $\psi_\mu:\mathbb{R}\to \mathbb{R}$ are $S$-times continuously differentiable and supported in $\Xi$. Taking $S\geq s\geq 0$ large enough, recall (e.g. from Equation (4.139) of \citet{gineMathematicalFoundationsInfiniteDimensional2015}) the wavelet characterisation of the Sobolev spaces on $\Xi$ given by
\begin{equation}
  H^s(\Xi) = \Set[\bigg]{f\colon\mathbb{R}\to\mathbb{R}\,\Big\vert\,  f|_{\Xi^c}=0,\,\norm{f}_{H^s(\Xi)}^2 =  \sum_{\mu}2^{2s \abs{\mu}}\iprod{f}{\psi_{\mu}}_{L^2(\Xi)}^2<\infty}.\label{eq:Sobolev_Statistician}
\end{equation}

In the following, we always assume that $S\geq \beta$ is large enough depending on the known regularity of a data generating \enquote{true} reaction function $f_0 \in C^\beta(\Xi)$. For a maximal frequency $M\in\mathbb{N}$, which may depend on $\lambda$, and some $\beta_0\ge 0$ we take as prior $\Pi_\lambda=\mathcal{L}(f)$ the law of the random function
\begin{equation}
	f = \sum_{\abs{\mu}\le M}2^{-\beta_0\abs{\mu}} Z_\mu \psi_{\mu},\quad Z_{\mu}\overset{i.i.d.}{\sim} N(0,1).\label{eq:Prior}
\end{equation}
This induces a centred Gaussian probability measure on $C^S(\mathbb{R})$ supported on the finite-dimensional approximation spaces 
\begin{align*}
	V_M &\coloneq \operatorname{span}(\psi_\mu\,\colon\, \abs{\mu}\leq M)\subset\Theta.
\end{align*} 
Since the $\psi_\mu$ are compactly supported, $V_M$ contains for each $j$ only finitely many $\psi_\mu$ with $\mu=(j,k)$ and by standard wavelet theory $V_M$ has dimension $\operatorname{dim}(V_M)\lesssim 2^M$. 

Denoting by $\ell_{\lambda}(f)\coloneq\log(\D \prob[f]{}/\D \prob[f_0]{})$ the log-likelihood function of \eqref{eq:Bayes:Likelihood_Hellinger}, the posterior distribution arises from Bayes' theorem and takes the form
\begin{align}
	\D \Pi_{\lambda}(f\,|\,X^\lambda) & = \frac{e^{\ell_\lambda(f)}\D \Pi_{\lambda}(f)}{\int_{V_M}{e^{\ell_\lambda(\bar{f})}\D \Pi_{\lambda}(\bar{f})}},\quad f\in C^S(\mathbb{R}).\label{eq:posterior}
\end{align}

Practical implementation of the posterior \eqref{eq:posterior} is straightforward as the Gaussian prior \eqref{eq:Prior} is conjugate noting the quadratic form of log-likelihood \eqref{eq:aux_Def_Likelihood}. To see this, let us write a candidate reaction function $f\in V_M$ relative to the first $\operatorname{dim}(V_M)$ wavelet basis functions as $f=\mathbf{f}\transpose\mathbf{\Psi}$ with coefficients $\mathbf{f}\in\mathbb{R}^{\operatorname{dim}(V_M)}$ and with a vector of basis functions $\mathbf{\Psi}=(\psi_{\mu}\,\colon\,\abs{\mu}\le M)\transpose$. Setting 
\begin{equation*}
		\mathbf{G}\coloneq \int_0^T\int_\Lambda \mathbf{\Psi}(X_t(y))\mathbf{\Psi}(X_t(y))\transpose\,\D y\D t,\quad \mathbf{a}\coloneq\int_0^T \iprod{\mathbf{\Psi}(X_t)}{\D X_t-\Updelta X_t\,\D t}_{L^2(\Lambda)}
\end{equation*}
the log-likelihood function equals $\ell_{\lambda}(f)=\mathbf{a}\transpose\mathbf{f}-(1/2)\mathbf{f}\transpose\mathbf{G}\mathbf{f}$. Denote furthermore by $\mathbf{\Sigma}$ the diagonal covariance matrix of the Gaussian prior $\Pi_{\lambda}$ with entries $2^{-2\beta_0|\mu|}$, $|\mu|\leq M$, on the diagonal according to \eqref{eq:Prior}. Setting $\mathbf{Q}\coloneq (\mathbf{G}+\mathbf{\Sigma}\inv)\inv$, $\boldsymbol{\mu}\coloneq \mathbf{Q} \mathbf{a}$, we conclude that 
\begin{align*}
    \log\left(\D \Pi_{\lambda}(f\,|\,X^\lambda)\right) & \sim \mathbf{a}\transpose\mathbf{f}-\frac{1}{2}\mathbf{f}\transpose(\mathbf{G}+\mathbf{\Sigma}^{-1})\mathbf{f}\sim-\frac{1}{2}(\mathbf{f}-\boldsymbol{\mu})\transpose \mathbf{Q}\inv (\mathbf{f}-\boldsymbol{\mu}),
\end{align*}
such that the posterior distribution of the parameter $\mathbf{f}$ is multivariate normal with mean $\boldsymbol{\mu}$ and covariance matrix $\mathbf{Q}$. This allows for direct sampling methods instead of computationally more demanding MCMC algorithms, see Section \ref{subsec:Numerics_Balanced} for an illustration. Even though there is no analytical formula for higher order Daubechies wavelets, they can be efficiently approximated using the cascade algorithm (Section 6.5 of \citet{daubechiesTenLecturesWavelets1992}). Gaussian conjugacy holds similarly for posteriors of continuously observed (finite-dimensional) diffusion processes with a Gaussian prior on the drift vector, cf. \citep{nicklNonparametricStatisticalInference2020, giordanoNonparametricBayesianInference2022}. 

\subsection{Limit process and spatial ergodicity}\label{subsec:Spatial_Ergodicity}

Our first two main results study the statistical properties of the model in the spatially ergodic setting in the large domain limit $|\Lambda|=\lambda\to\infty$. For $\bar{y}\in\bar{\Lambda}$ define the process $(Z_t^{\bar{y}})_{t\in[0,T]}$ as the random field solution to the SPDE
\begin{equation}
		\D Z_t^{\bar{y}}(y) = \Updelta Z_t^{\bar{y}}(y)\,\D t + f(Z_t^{\bar{y}}(y))\,\D t + \D W_t(y),\quad 0\le t\le T,y\in\mathbb{R},\label{eq:Z}
\end{equation}
on $\mathbb{R}$ with spatially constant initial condition $Z_0^{\bar{y}}(y)= \chi(\bar{y})$, $y\in\mathbb{R}$. By Lemma 7.1 of \citet{chenSpatialErgodicitySPDEs2021} the process $y\mapsto Z_t^{\bar{y}}(y)$ is stationary for every $t\ge 0$ and $\bar{y}\in\bar{\Lambda}$ in the sense that $Z_t^{\bar{y}}(y)\overset{d}{=}Z_t^{\bar{y}}(0)$ for all $y\in\mathbb{R}$.

We begin by showing that the process $(Z_t^{\bar{y}})_{t\in[0,T]}$ is the natural pointwise limit of the rescaled process $X_t(\lambda\bar{y})$, and that ergodic spatial averages of $X$ converge as $\lambda\to\infty$. While the Neumann boundary conditions disappear in the large domain limit, $Z^{\bar{y}}$ still depends on $\bar{y}$ due to the initial condition $X_0(\lambda\bar{y})=\chi(\bar{y})$ being constant and independent of $\lambda$. In particular, when $X_0\equiv 0$, $Z^{\bar{y}}$ does not depend on the specific point $\bar{y}$. 

For the case $f\equiv 0$ and $X_0\equiv 0$, the convergence of $X$ to the global solution $Z$ has been obtained by \citet{Altmeyer2021, pasemannNonparametricDiffusivityEstimation2025} using the convergence of the associated heat semi-groups. Since we are interested in pointwise convergence for both $\bar{y}\in\bar{\Lambda}$ and $0\le t\le T$ held fixed, our proof is based on comparing the heat kernels on $\Lambda$ with Neumann boundary conditions and on $\mathbb{R}$, respectively. The statement for general $f$ and $X_0$ follows using a Gronwall-type expansion and by tracing the impact of the Neumann heat semi-group for a short time $\lambda^{-2}t\to 0$ as $\lambda\to\infty$.

\begin{theorem}\label{thm:ErgodicProperties}
Grant Assumption \ref{assump:Parameterspace} and let $f_0\in C^3(\mathbb{R})$.
    \begin{enumerate}[(a)]
        \item\label{num:thm_Ergodicity1} For every $\bar{y}\in\bar{\Lambda}$ and $0\le t\le T$ we have $      X_t(\lambda \bar{y})\xrightarrow{d}Z_t^{\bar{y}}(0)$, $\lambda\to\infty$.
        \item\label{num:thm_Ergodicity2} For any $g\in L^1(\mathbb{R})$ the spatial ergodicity property holds, i.e.
\begin{equation*}
	\mathcal{G}_{\lambda}(g)\coloneq\int_0^T \frac{1}{\lambda} \int_{\lambda\bar{\Lambda}} g(X_t(y))\,\D y\D t - \int_0^T \int_{\bar{\Lambda}}\EV[f_0]{g(Z_t^{\bar{y}}(0))}\,\D \bar{y}\D t\xrightarrow{\prob[f_0]{}}0,\quad\lambda\to\infty.
\end{equation*}

    \end{enumerate}
\end{theorem}
\begin{proof}
    Part \eqref{num:thm_Ergodicity1} follows from Proposition \ref{prop:Convergencetoglobal} (below). Part \eqref{num:thm_Ergodicity2} follows from Theorem \ref{thm:SubgaussianConcentration} (below).
\end{proof}

While it is generally not true that the pointwise convergence in part \eqref{num:thm_Ergodicity1} of Theorem \ref{thm:ErgodicProperties} implies convergence of the respective densities, this is true here. Indeed, the convergence of the densities can be deduced from the converse Scheffé Theorem (Theorem 1 of \citet{sweetingConverseScheffesTheorem1986}) by combining part \eqref{num:thm_Ergodicity1} of the theorem with the bound on the first derivative of the marginal densities of $X_t(y)$ from Proposition \ref{prop:DensityBoundedDerivative} (below).

The following sub-Gaussian concentration inequality makes the convergence in part \eqref{num:thm_Ergodicity2}  of Theorem \ref{thm:ErgodicProperties} quantitative. This is the key technical tool for controlling the Hellinger semi-metric \eqref{eq:Bayes_Intro_Hellinger} in the proof of Theorem \ref{thm:PosteriorContraction}. The proof of the concentration inequality is nonstandard since the process $y\mapsto X_t(y)$ is neither Gaussian nor are techniques for semi-martingales directly applicable. The main steps of our proof rely on Malliavin Calculus. They are reviewed in Section \ref{sec:Proof_Subgaussian_Concentration}.

\begin{theorem}[Subgaussian concentration of $\mathcal{G}_{\lambda}(g)$]\label{thm:SubgaussianConcentration}
	Grant Assumption \ref{assump:Parameterspace}, let $f_0\in\Theta$ and $g\in L^1(\mathbb{R})$. There exists a constant $0<C<\infty$, depending only on $\norm{f_0}_{C^3(\mathbb{R})}$, $T$ and $\norm{\chi}_{H^{\tilde{s}}(\bar{\Lambda})}$, such that the following hold:
	\begin{enumerate}[(a)]
		\item\label{num:SubgaussianThm2} We have the sub-Gaussian concentration
		      \begin{equation}
			      \prob[f_0]{\abs{\mathcal{G}_{\lambda}(g)-\EV[f_0]{\mathcal{G}_{\lambda}(g)}}\ge \lambda^{-1/2} x}\le 2\exp\bigg(-\frac{x^2}{2 C \norm{g}_{L^1(\mathbb{R})}^2}\bigg),\quad x\ge 0.\label{eq:Subgaussian_2}
		      \end{equation}
	
		\item\label{num:SubgaussianThm3} We can bound the expectation
		      \begin{equation*}
			      \abs{\EV[f_0]{\mathcal{G}_{\lambda}(g)}}\leq \lambda^{-1}C\norm{g'}_{L^{\infty}(\mathbb{R})}.
		      \end{equation*}
	\end{enumerate}
\end{theorem}

\begin{proof}
	The proof of the sub-Gaussian concentration in \eqref{num:SubgaussianThm2} is carried out in Section \ref{sec:Proof_Subgaussian_Concentration}. Part \eqref{num:SubgaussianThm3} follows from Corollary \ref{cor:Convergence_polynomials} taking $k=1$. 
\end{proof}


\subsection{Posterior contraction}\label{subsec:Posterior_Contraction}

We proceed with a contraction result for the posterior distribution \eqref{eq:posterior} using the Gaussian process prior \eqref{eq:Prior} for $\beta_0\geq 0$ in the large domain limit $\lambda\to\infty$. We rely on the test-based approach of \citet{ghosalConvergenceRatesPosterior2000}  and the concentration results to first establish posterior contraction in the expected random Hellinger semi-metric \eqref{eq:Bayes_Intro_Hellinger}, which is equivalent to the $L^2(\Xi)$ metric. Borrowing ideas from \citet{giordano2020consistency}, we then show that the posterior concentrates on sets of bounded higher-order smoothness if $\beta_0>3/2$, cf. Equation \eqref{eq:Concentration_on_Lipschitzball} (below). Exploiting such additional regularisation of the posterior has recently been instrumental for dealing with different Bayesian nonlinear inverse problems, see \citep{nickl2023bayesian} for an overview. In our case, this allows us to use Theorem \ref{thm:SubgaussianConcentration} \eqref{num:thm_Ergodicity2} to relate the expected Hellinger semi-metric to the norm induced by the inner product
\begin{equation*}
    \iprod{f}{g}_0\coloneq \int_0^T\int_{\bar{\Lambda}}\EV[f_0]{f(Z_t^{\bar{y}}(0))g(Z_t^{\bar{y}}(0))}\,\D\bar{y}\D t = \int_0^T\int_{\bar{\Lambda}}\int_{\mathbb{R}}f(z)g(z)p_{f_0,\bar{y}}(t,z)\,\D z\D\bar{y}\D t,
\end{equation*}
where $p_{f_0,\bar{y}}(t,\MTemptyplaceholder)$ for $t>0$ and $\bar{y}\in\bar{\Lambda}$ denotes the density of $Z_t^{\bar{y}}(0)$. We note that by the density bounds from Corollary D.2 of \citet{gaudlitzNonparametricEstimationReaction2023}, there exist constants $0<c,C<\infty$ depending on $\Xi$, $\Lip{f_0}$ and $T$ satisfying the norm equivalence
\begin{equation}
    c\norm{f}_{L^2(\Xi)}\le \norm{f}_0\le C\norm{f}_{L^2(\Xi)},\quad f\in \Theta.\label{eq:norm_equivalence}
\end{equation}
 
\begin{theorem}[Posterior contraction]\label{thm:PosteriorContraction}
	Grant Assumption \ref{assump:Parameterspace} and suppose $f_0\in C^\beta(\mathbb{R})\cap\Theta$ for $\beta\geq 3$. Consider the Gaussian prior $\Pi_\lambda$ from \eqref{eq:Prior} with $0\le \beta_0<\beta$ and cut-off $M=M_\lambda\in\mathbb{N}$ such that $2^{M}\sim \lambda ^{1/(2\beta+1)}$.
    \begin{enumerate}[(a)]
        \item\label{num:PosteriorContraction_L2} Then for every $K_\lambda\to\infty$
		\begin{equation*}
			\Pi_\lambda(f\in V_M \colon \norm{f-f_0}_{L^2(\Xi)}\ge K_\lambda \lambda^{-\beta/(2\beta+1)}\log(\lambda)\,|\,X^\lambda)\xrightarrow{\prob[f_0]{}}0,\quad \lambda\to \infty.
		\end{equation*}
        \item\label{num:PosteriorContraction_0} If $3/2<\beta_0\le \beta$ we have for any large enough $L>0$ 
	\begin{equation}
        \Pi_\lambda(f\in V_M\,\colon\, \norm{f}_{C^1(\mathbb{R})}\geq L\,\vert\,X^\lambda)\xrightarrow{\prob[f_0]{}}0,\quad \lambda\to\infty,\label{eq:Concentration_on_Lipschitzball}
    \end{equation}
	and for every $K_\lambda\to\infty$
		\begin{equation*}
			\Pi_\lambda(f\in V_M \colon \norm{f-f_0}_0\ge K_\lambda \lambda^{-\beta/(2\beta+1)}\log(\lambda)\,|\,X^\lambda)\xrightarrow{\prob[f_0]{}}0,\quad \lambda\to \infty.
		\end{equation*}
    \end{enumerate}
\end{theorem}

\begin{proof}
	See Section \ref{sec:Bayes_proofContraction}.
\end{proof}

The minimal smoothness assumption $f_0\in C^3(\mathbb{R})$ is required to bound the derivative of the density $p_{f_0,\lambda}(t,y,\MTemptyplaceholder)$ of $X_t(y)$ in Proposition \ref{prop:DensityBoundedDerivative}, which plays a crucial role for the sub-Gaussian concentration of the Hellinger semi-metric $h_\lambda$ around the $\norm{}_0$-norm using Theorem \ref{thm:SubgaussianConcentration}. See Remark \ref{rmk:OnDensity} for a discussion on the necessity of this regularity assumption. 
Moreover, the restriction of $f_0\in\Theta$ to have compact support is not required for the spatial ergodicity results in Section \ref{subsec:Spatial_Ergodicity}, but it is necessary for the posterior contraction in Theorem \ref{thm:PosteriorContraction} to ensure that the expected (squared) Hellinger semi-metric $\EV{h_\lambda(f,f_0)^2}$ and $\norm{}_0^2$ are equivalent to the (squared) $L^2(\Xi)$-metric.

Contraction of the posterior immediately yields a frequentist convergence rate for the posterior mean $\hat{f} = \EV[\lambda]{f\given X^\lambda}$. Since the posterior is Gaussian, the posterior mean equals the maximum-a-posteriori estimator
\begin{align*}
	\hat{f} = \operatorname{argmax}_{f\in V_M} \D \Pi(f\,|\,X^\lambda) = \operatorname{argmax}_{f\in V_M} \left(\ell_\lambda(f)-\frac{1}{2}\norm{f}^2_{\mathbb{H}^\lambda}\right),
\end{align*}
where $\norm{}_{\mathbb{H}^{\lambda}}$ is the reproducing kernel Hilbert space (RKHS) norm with respect to \eqref{eq:Prior}. Lemma 11.43 of \citet{Vaart2017} shows that the RKHS is given by $(V_M,\norm{}_{\mathbb{H}^\lambda})$ and
\begin{equation}
    \norm{h}_{\mathbb{H}^\lambda}\coloneq \norm{h}_{H^{\beta_0}(\Xi)},\quad h\in V_M.\label{eq:RKHS}
\end{equation}

\begin{corollary}
    Let $\hat{f} = \EV[\lambda]{f\given X^\lambda}$ be the posterior mean. Under the conditions of Theorem \ref{thm:PosteriorContraction} we have for every $K_{\lambda}\to\infty$
    \begin{align*}
        \norm{\hat{f}-f_0}_0 = \mathcal{O}_{\prob[f_0]{}}\left(K_{\lambda}\lambda^{-\beta/(2\beta+1)}\log(\lambda)\right),\quad \lambda\to\infty.
    \end{align*}
\end{corollary}

\begin{proof}
    The posterior is a Gaussian measure on $V_M$ with centring $\hat{f}\in V_M$. By Anderson's Lemma (Theorem 2.4.5 of \citet{gineMathematicalFoundationsInfiniteDimensional2015}), this means that the function $Q(g)\coloneq \Pi_{\lambda}(C+g\,\vert\,X^\lambda)$, $g\in V_M$, with $C\coloneq \left\{f\in V_M\,\colon\, \norm{f}_{L^2(\Xi)}\leq K_\lambda \lambda^{-\beta/(2\beta+1)}\log(\lambda)\right\}$ is maximal at $g=\hat{f}$. Theorem 2.5 of \citet{ghosalConvergenceRatesPosterior2000} and Theorem \ref{thm:PosteriorContraction} therefore imply the rate of convergence $ \norm{\hat{f}-f_0}_{L^2(\Xi)}= \mathcal{O}_{\prob[f_0]{}}(K_{\lambda}\lambda^{-\beta/(2\beta+1)})\log(\lambda)$. This transfers to the $\norm{}_0$-norm by \eqref{eq:norm_equivalence}.
\end{proof}

The rate of convergence $\lambda^{-\beta/(2\beta+1)}\log(\lambda)$ is minimax-optimal over $\beta$-regular Hölder balls up to a log-factor.
\begin{proposition}\label{prop:Bayes_lowerBounds}
	Grant Assumption \ref{assump:Parameterspace}, let $\beta\ge 1$ and $L>0$. Then
	\begin{equation*}
		\liminf_{\lambda\to \infty}\inf_{T_\lambda}\sup_{f\in\mathcal{F}_{\beta,L}\cap \Theta}\prob*[f]{\norm{T_\lambda-f}_{0}\ge \lambda^{-\beta/(2\beta+1)}}\ge C,
	\end{equation*}
	where $\inf_{T_\lambda}$ denotes the infimum over all estimators $T_\lambda$ for $f$ based on $X^\lambda$ solving the semi-linear SPDE \eqref{eq:SPDE} with domain size $\lambda$, and where the constant $0<C<\infty$ depends on $L,T$.
\end{proposition}
\begin{proof}
	See Section \ref{sec:Bayes_proofContraction}.
\end{proof}

\begin{remark}
	The rate of contraction in Theorem \ref{thm:PosteriorContraction} hinges on the optimal choice of the cut-off $M$. Using standard arguments (e.g., Section 2.2.3 of \citet{giordano2020consistency}), posterior contraction at the same rate extends to adaptive priors, where $M$ is randomised independent of the unknown smoothness $\beta$. In this case, however, the conjugacy for the prior is lost and sampling needs to be done using computationally more demanding MCMC algorithms. Moreover, similar to Section 2.2.2 of \citet{giordano2020consistency} it can be shown by modifying the proof of Lemma \ref{lem:Prior} (below) that the posterior contracts at the minimax-optimal rate for a rescaled Gaussian process prior $\bar{\Pi}_{\lambda}=\mathcal{L}(\bar{f})$, where $\bar{f}=\lambda^{-1/(4\beta+2)}f$ with $f$ from \eqref{eq:Prior} and with $\beta_0=\beta$. In the case of rescaled Gaussian priors, however, achieving adaptation is an open problem.
\end{remark}

\subsection{A nonparametric Bernstein--von Mises theorem}\label{subsec:BvM}

In classical parametric statistical models the Bernstein--von Mises theorem (e.g., \citet{Vaart1998}) states that in the large data limit the posterior distribution approximately equals a normal distribution centred at an efficient estimator (usually the posterior mean) with the variance attaining the Cramér-Rao bound in the underlying statistical model. Several works such as \citep{cox1993analysis, freedman1999} have shown that such a theorem cannot hold in nonparametric statistical models. 

Following the semiparametric ideas introduced by \citet{castillo2013, castillo2014a}, which were recently extended to the finite-dimensional diffusion setting by \citet{nicklNonparametricStatisticalInference2020,giordano2025}, we show that a Bernstein--von Mises-type theorem for the posterior, centred at the posterior mean, holds in function spaces with norms weaker than $L^2$. Here, we consider the negative Sobolev spaces $(\mathcal{X},d)=(H^{-\rho}(\Xi),\norm{}_{H^{-\rho}(\Xi)})$ for $\rho>3/2$ to ensure that Lemma \ref{lem:LAN} can be applied, cf. \eqref{eq:gamma0_bound} (below).  Since weak convergence in $H^{-\rho}(\Xi)$ corresponds by duality to uniform weak convergence of linear functionals, it suffices to analyse the posterior by the action $\gamma\mapsto \iprod[]{f-\hat{f}}{\gamma}_{L^2(\Xi)}$, $\gamma\in H^\rho(\Xi)$. 

Weak convergence of probability measures on $(\mathcal{X},d)$ is metrised by the bounded Lipschitz metric, see Chapter 18 of \citet{dudleyRealAnalysisProbability2002}. For two probability measures $\tau,\tau'$ on $(\mathcal{X},d)$ it is defined by
\begin{align*}
	\beta_{\mathcal{X}}(\tau,\tau') = \sup_{F\colon\Lip[\mathcal{X}]{F}\leq 1}\abs*{\int_{\mathcal{X}} F\,\D\tau - \int_{\mathcal{X}} F\,\D\tau'}.
\end{align*}

The limiting distribution of the posterior will be given below by a generalised Wiener process $\mathbb{W}=(\mathbb{W}(\gamma):\gamma\in H^\rho(\Xi))$ on $H^{\rho}(\Xi)$ with covariances
\begin{align}
	\EV[f_0]{\mathbb{W}(\gamma)\mathbb{W}(\bar\gamma)}=\int_{\Xi}\gamma(z)\bar\gamma(z)/p_{f_0}(z)\,\D z,\quad f_0\in\Theta.\label{eq:Bvm_cov}
\end{align}
Here, we set $p_{f_0}(z)\coloneq\int_0^T\int_{\bar{\Lambda}}p_{f_0,\bar{y}}(t,z)\,\D \bar{y}\D t$ such that
 \begin{equation*}
    \norm{h}^2_0=\norm{hp_{f_0}^{1/2}}^2_{L^2(\mathbb{R})}=\int_0^T\int_{\bar{\Lambda}}\int_{\mathbb{R}}h^2(z)p_{f_0,\bar{y}}(t,z)\,\D z\D\bar{y}\D t.
\end{equation*}
The numerical example in Section \ref{subsec:Numerics_Balanced} illustrates the role of the density $p_{f_0}$ appearing in the limiting covariance. Note that for $\rho>3/2$, the Wiener process $\mathbb{W}$ indeed yields a tight Gaussian measure on $H^{-\rho}(\Xi)$, see Theorem 3.1 of \citet{nicklNonparametricStatisticalInference2020} for details.

At the heart of the analysis in this section is the following LAN (\enquote{local asymptotic normality}) expansion with the associated LAN-norm $\norm{}_0$. Consequently, well-known arguments from semiparametric statistics such as in Chapter of \citet{Vaart1998} imply that the asymptotic variances appearing in Theorem \ref{thm:BvM} and Corollary \ref{cor:BvM_Truth} below are optimal in an information-theoretic sense.

\begin{lemma}[LAN expansion]\label{lem:LAN}
	Grant Assumption \ref{assump:Parameterspace} and let $h\colon\mathbb{R}\to\mathbb{R}$ be  Lipschitz-continuous. Set $W_\lambda(h)\coloneq \frac{1}{\sqrt{\lambda}}\int_0^T\iprod{h(X_t)}{\D W_t}_{L^2(\Lambda)}$, $\mathcal{I}_\lambda(h) \coloneq \frac{1}{\lambda}\int_0^T\norm{h(X_t)}_{L^2(\Lambda)}^2\,\D t$. Then we have the likelihood expansion
	\begin{equation}
		\ell_\lambda\left(f_0+\frac{h}{\sqrt{\lambda}}\right) -  \ell_\lambda(f_0) = W_\lambda(h) - \frac{1}{2}\mathcal{I}_\lambda(h),\quad f_0\in\Theta\label{eq:LikelihoodExpansion}
	\end{equation}
	along with the convergences $W_\lambda(h)\xrightarrow{d}N(0,\norm{h}_{0}^2)$,  $\mathcal{I}_\lambda(h)\xrightarrow{\prob[f_0]{}}\norm{h}_{0}^2$ as $\lambda\to \infty$.
\end{lemma}
\begin{proof}
	The likelihood expansion \eqref{eq:LikelihoodExpansion} follows from \eqref{eq:Bayes:Likelihood_Hellinger}. By Theorem \ref{thm:SubgaussianConcentration} with $f=f_0$ and $g=h^2$, we find $\mathcal{I}_\lambda/\norm{h}_0^2\to 1$ in $\prob[f_0]{}$-probability as $\lambda\to \infty$. The convergence in distribution of the (time-) martingale $W_\lambda$ follows by noting that its quadratic variation is given by $\mathcal{I}_\lambda$. Since the latter converges to $\norm{h}_0^2$ in $\prob[f_0]{}$-probability, we conclude by a standard martingale central limit theorem (Theorem 5.5.4 of \citet{Liptser1989}).
\end{proof}

Besides weak convergence on $H^{-\rho}(\Xi)$, we also have convergence in total variation distance
\begin{align}
	\operatorname{TV}(\tau,\tau') = \sup_{F\colon\norm{F}_{L^\infty(\mathcal{X})}\leq 1}\abs*{\int_{\mathcal{X}} F\,\D\tau - \int_{\mathcal{X}} F\,\D\tau'}\label{eq:TV}
\end{align}
for the posterior distribution acting on a single linear functional. 

\begin{theorem}[Semiparametric Bernstein--von Mises]\label{thm:BvM}
	Grant Assumption \ref{assump:Parameterspace} and suppose $f_0\in C^\beta(\mathbb{R})\cap\Theta$ for $\beta\geq 3$. Let $\rho>3/2$ and let $\mathbb{W}$ be the generalised Wiener process from \eqref{eq:Bvm_cov}. Consider the Gaussian prior $\Pi_\lambda$ from \eqref{eq:Prior} with $3/2<\beta_0<\beta$ and cut-off $M=M_\lambda\in\mathbb{N}$ such that $2^{M}\sim \lambda ^{1/(2\beta_0+1)}$. Suppose $f\sim \Pi_\lambda(\MTemptyplaceholder|\,X^\lambda)$ and let $\hat{f} = \EV[\lambda]{f\given X^\lambda}$ be the posterior mean. Then the following holds as $\lambda\to\infty$:
	\begin{enumerate}[(a)]
		\item If $\gamma\in H^\rho(\Xi)$, then $\operatorname{TV}\left(\mathcal{L}(\sqrt{\lambda}\iprod{f-\hat f}{\gamma}_{L^2(\Xi)}), \mathcal{L}(\mathbb{W}(\gamma))\right) \overset{\prob[f_0]{}}\rightarrow 0$.\label{num:BvM_TVConvergence}
		\item $\beta_{H^{-\rho}(\Xi)}\left(\mathcal{L}(\sqrt{\lambda}(f-\hat{f})), \mathcal{L}(\mathbb{W})\right) \overset{\prob[f_0]{}}\rightarrow 0$.\label{num:BvM_BetaConvergence}
	\end{enumerate}
\end{theorem}
\begin{proof}
	See Section \ref{sec:Bayes_proofContraction}.
\end{proof}

Note that the rate of convergence in Theorem \ref{thm:BvM} is $1/\sqrt{\lambda}$, regardless of the regularity of $f_0$ (as long as $\beta\geq 3$), which is always faster than the rate of convergence in Theorem \ref{thm:PosteriorContraction}. Moreover, we have chosen a larger cut-off value $M$ compared to Theorem \ref{thm:PosteriorContraction}, which is now, however, independent of the unknown smoothness $\beta$.

The Bernstein--von Mises theorem yields asymptotic results for the posterior mean as estimator of the ground truth $f_0$.

\begin{corollary}\label{cor:BvM_Truth}
	Under the assumptions of Theorem \ref{thm:BvM} the following holds true as $\lambda\to\infty$:
	\begin{enumerate}[(a)]
		\item If $\gamma\in H^\rho(\Xi)$, then $\sqrt{\lambda}\iprod{\hat{f}-f_0}{\gamma}_{L^2(\Xi)} \xrightarrow{d}N(0,\norm{\gamma/p_{f_0}}_0^2)$.\label{num:cor_BvM_pointwise}
		\item $\sqrt{\lambda}(\hat{f}-f_0) \xrightarrow{d}\mathbb{W}$ on $H^{-\rho}(\Xi)$.\label{num:cor_BvM_uniform}
	\end{enumerate}
\end{corollary}
\begin{proof}
	See Section \ref{sec:Bayes_proofContraction}.
\end{proof}

Theorem \ref{thm:BvM} \eqref{num:BvM_BetaConvergence} can be used to derive asymptotic confidence sets for $f_0$ in weak norms from posterior credible balls, see \citep{castillo2013} for details. We continue with concrete examples in which asymptotic normality for linear or nonlinear functionals can be used for statistical inference.

\begin{example}[Test for polynomial degree]\label{examp:BvM1}
	Of particular interest in applications are stochastic reaction-diffusion equations with polynomial reaction functions \citep{cristofolUniquenessPointwiseObservations2012} where the steady states correspond to the roots of the polynomials. Such models are compatible with Assumption \ref{assump:Parameterspace} if we suppose that $f_0(x)=\sum_{k=0}^q a_k x^k$ for $x\in\bar{\Xi}$, $a_1,\dots,a_q\in\mathbb{R}$, $q\in\mathbb{N}_0$, where $\bar{\Xi}\subset\Xi$ is a compact interval, and that $f_0$ extends outside of $\bar{\Xi}$ smoothly to zero with support in $\Xi$.
		
	Suppose we want to derive a statistical test for the null hypothesis that $f_0$ is a polynomial of degree at most $q^\ast\geq 0$ on $\bar{\Xi}$. Let $\gamma$ be a sufficiently smooth function with compact support in $\bar{\Xi}$ that is orthogonal to all polynomials up to degree $q^\ast$ such that $\iprod[]{(\MTemptyplaceholder)^{\bar{q}}}{\gamma}_{L^2(\Xi)}=0$ for all $0\leq\bar{q}\leq q^\ast$. Such functions are constructed by \citet{Tsybakov2009}, or we may take suitable \enquote{mother} wavelets, cf. Proposition 4.2.7 of \citet{gineMathematicalFoundationsInfiniteDimensional2015}. Under the null hypothesis the inner product $\iprod[]{f_0}{\gamma}_{L^2(\Xi)}$ vanishes, such that by Corollary \ref{cor:BvM_Truth} \eqref{num:cor_BvM_pointwise} the test statistic $\tau\coloneq\sqrt{\lambda}\iprod[]{\hat{f}}{\gamma}_{L^2(\Xi)}$ is for large $\lambda$ approximately normally distributed with variance $\norm{\gamma/p_{f_0}}^2_0$. If an upper bound on the variance is known (e.g. using the density bounds from Corollary \ref{cor:DensityBounds}), we can derive approximate critical values for the test statistic $\tau$ relative to the standard normal distribution. 
\end{example}

\begin{example}[Estimating polynomial coefficients]\label{examp:BvM2}
	Continuing Example \ref{examp:BvM1}, suppose that $f_0$ is a polynomial of degree $q^\ast\in\mathbb{N}$ on $\bar{\Xi}$. If $\gamma$ is orthogonal to all polynomials of degree $q^\ast-1$ and $\iprod[]{(\MTemptyplaceholder)^{q^\ast}}{\gamma}_{L^2(\Xi)}\neq 0$, then $\iprod[]{f_0}{\gamma}_{L^2(\Xi)}=a_{q^\ast}\iprod[]{x^{q^\ast}}{\gamma}_{L^2(\Xi)}$ and using Corollary \ref{cor:BvM_Truth} \eqref{num:cor_BvM_pointwise} we can estimate $a_{q^\ast}$ by $\iprod[]{\hat{f}}{\gamma}_{L^2(\Xi)}/\iprod[]{(\MTemptyplaceholder)^{q^\ast}}{\gamma}_{L^2(\Xi)}$ with the rate of convergence $\sqrt{\lambda}$. Using this approach, we can estimate all coefficients $a_1,\dots,a_{q^\ast}$.
\end{example}

\begin{example}[Estimating power functions]\label{examp:BvM3}
	Besides linear functionals, as discussed in Section 2.3.2 of \citet{castillo2013}, the Bernstein--von Mises theorem \ref{thm:BvM} also yields asymptotic normality for smooth nonlinear functionals $\Psi(f_0)$ as long as $\Psi\colon L^2\to\mathbb{R}$ is differentiable in a functional sense near $f_0$ with a sufficiently small remainder term. Skipping details, we can argue as \citet{giordano2025} in their Section 2.2 and show that this yields asymptotic normality for the plug-in estimator $\Psi(\hat{f})$ for $\Psi(f_0)$ and functionals of the form $\Psi(f)=\int_{\mathbb{R}}f(x)^q\,\D x$, $q\geq 2$. In particular, this allows one to estimate the $L^2$ norm of $f_0$.
\end{example}

\subsection{Numerical exploration}\label{subsec:Numerics_Balanced}

To visualise our results we consider a (modified) stochastic Allen-Cahn equation with stable states $\pm 3$ given by
\begin{equation}
	\D X_t(y) =\Updelta X_t(y)\,\D t + f(X_t(y))\,\D t + \,\D W_t(y),\quad 0\le t\le 1,\quad X_0\equiv 0,\label{eq:Bayes_Simulation_SPDE}
\end{equation}
on $\Lambda =\lambda\bar{\Lambda}$ with $\bar{\Lambda}=(-1/2,1/2)$ and Neumann boundary conditions. Let $\Xi=[-3.5,3.5]$ and let the data-generating reaction function $f_0\colon\mathbb{R}\to\mathbb{R}$ be given by
\begin{equation*}
	f_0(x) = -(x^3-9x),\quad -3.25\le x\le 3.25,
\end{equation*}
with (polynomial) spline interpolation on $[-3.5,-3.25)\cup (3.25,3.5]$ such that $f_0(x)=f_0'(x)=f_0''(x)=f_0'''(x)=0$ for all $\abs{x}\ge 3.5$. For this illustration, the size of the domain is set to $\lambda=50$. The SPDE \eqref{eq:Bayes_Simulation_SPDE} is simulated using a semi-implicit Euler scheme with a finite difference approximation of $\Updelta$ based on \citet{Lord2014}. In Figure \ref{fig:Nonpara_Balanced} (left), a typical realisation of $X$ solving \eqref{eq:Bayes_Simulation_SPDE} is displayed.

We choose the sieved Gaussian wavelet prior from \eqref{eq:Prior} with Haar wavelets on $L^2(\Xi)$ and frequency cut-off $M=7$, $\beta_0 = 1/2$. Using the Gaussianity of the posterior and the explicit formulae for the posterior mean and covariance of Section \ref{subsec:Bayes_Setting_Posterior}, we arrive at Figure \ref{fig:Nonpara_Balanced} (middle and right).
   Since we approximate the spatial and temporal integrals in the likelihood \eqref{eq:Bayes:Likelihood_Hellinger} by Riemann sums, the figure displays an approximation to the true posterior distribution. A precise control of approximation error, or an analysis of Bayesian inference based on the exact likelihood of discrete observation as in \citep{nicklConsistentInferenceDiffusions2024, giordanoStatisticalAlgorithmsLowfrequency2025} present interesting avenues for future research. Although Haar wavelets only benefit from regularity up to $\beta=1$ in the approximation bound \eqref{eq:Bayes_SupNorm_Approximation}, Figure \ref{fig:Nonpara_Balanced} shows that the posterior nevertheless has most of its mass close to the data-generating reaction function $f_0$.
   An inspection of the proof of Theorem \ref{thm:PosteriorContraction} shows that the contraction rate for Haar-wavelets is $\lambda^{-1/3}$ instead of $\lambda^{-\beta/(2\beta+1)}$, $\beta\ge 3$.

In Figure \ref{fig:Nonpara_Balanced} (right), we contrast the (pointwise) posterior variance with the (rescaled) occupation time
\begin{equation}
	M(x) \coloneq 2^{7}\int_0^T\int_\Lambda \indicator(\abs{X_t(y)-x}\le 2^{-8})\,\D y\D t,\quad -3.5\le x\le 3.5.\label{eq:Bayes_OccupationTime}
\end{equation}
The comparison shows that the posterior variance is large where the occupation time is low. This connection is natural, since a shorter occupation time around a point $x\in\mathbb{R}$ implies that less information can be gathered about $f_0$ around $x$. Since the rescaled occupation time is an estimator for the density $p_{f_0}$, this intuition aligns well with the covariance structure of the limiting Gaussian process $\mathbb{W}$ in the Bernstein--von Mises Theorem \ref{thm:BvM}, which is inversely proportional to $p_{f_0}$. The relatively high occupation time around $x=0$ is due to the effect of the initial condition $X_0\equiv 0$. Note that the smooth appearance of the occupation time $M(x)$ does not necessarily yield insights into the regularity of the local time or the occupation density of the spatio-temporal process $(X_t(y)\,\colon\,0\le t\le T, y\in\Lambda)$ due to smoothing.
\begin{figure}
	\centering
	\begin{subfigure}{0.32\textwidth}
		\includegraphics[width = \textwidth]{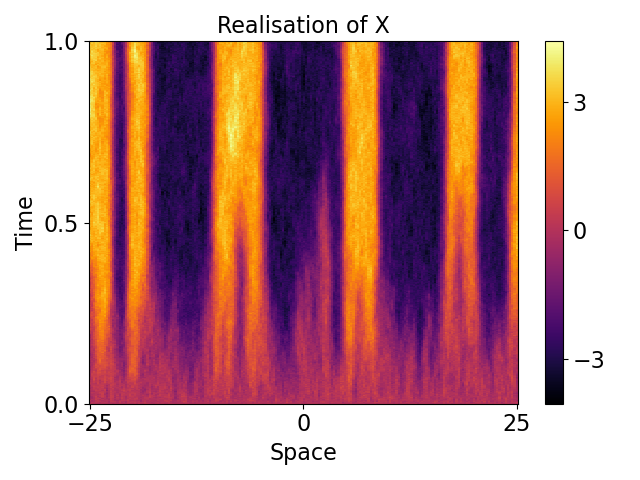}
	\end{subfigure}
	\begin{subfigure}{0.32\textwidth}
		\includegraphics[width = \textwidth]{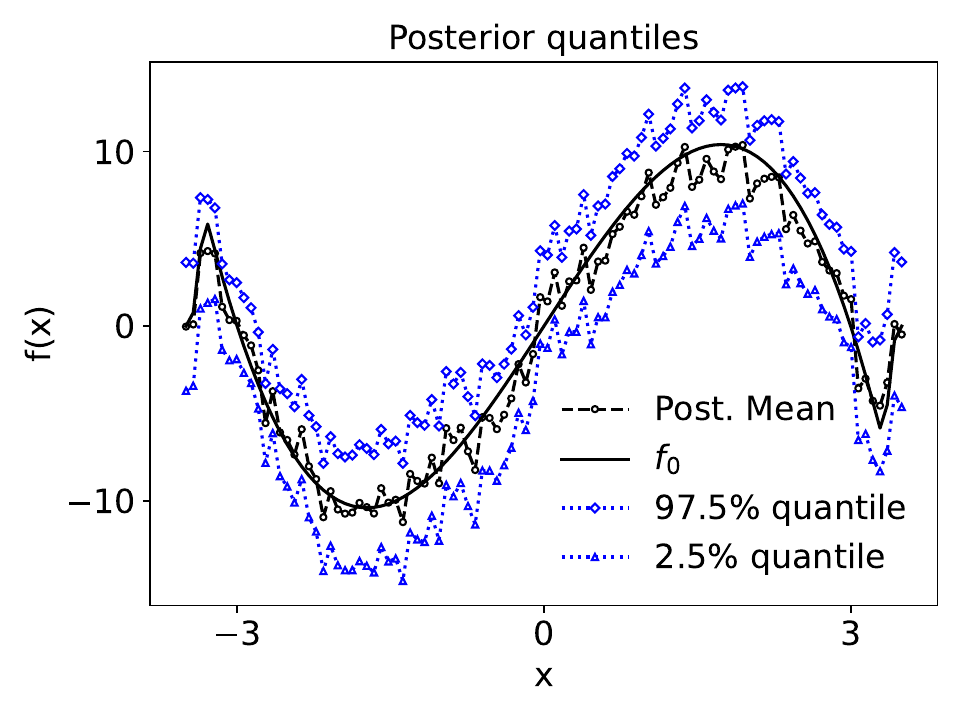}
	\end{subfigure}
	\begin{subfigure}{0.32\textwidth}
		\includegraphics[width = \textwidth]{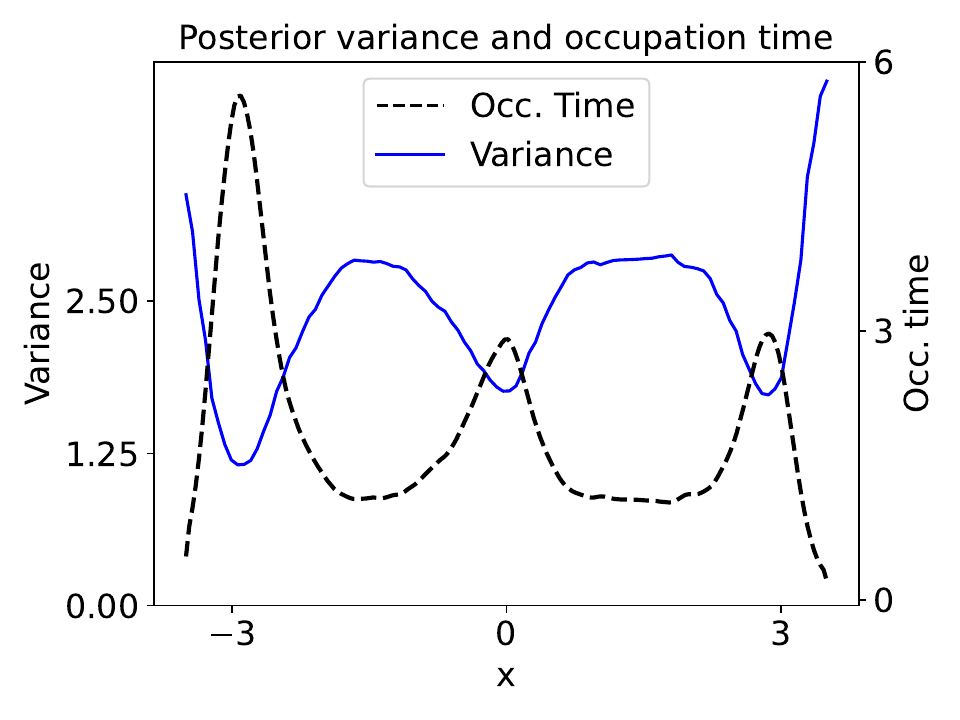}
	\end{subfigure}
	\caption{Illustration for $\lambda=50$. Left: Realisation of $X$ solving the semi-linear SPDE \eqref{eq:Bayes_Simulation_SPDE}. Middle: (Pointwise) quantiles of the corresponding posterior distribution with prior specified in Subsection \ref{subsec:Numerics_Balanced}. Right: Comparison of the posterior variance with the (rescaled) occupation time \eqref{eq:Bayes_OccupationTime}.}
	\label{fig:Nonpara_Balanced}
\end{figure}

\subsection{Discussion}\label{subsec:Discussion}

\subsubsection{The Bayesian approach for more general SPDEs}\label{subsec:Generalisation1}

A natural extension to \eqref{eq:SPDE} is the semi-linear SPDE
\begin{equation}
    \D X_t = A X_t \,\D t + F(X_t)\,\D t + B \,\D W_t,\quad 0\le t\le T,\label{eq:SPDE_GeneralSemiLinear}
\end{equation}
defined on a bounded domain $\Lambda\subset \mathbb{R}^d$, where $A\colon\operatorname{dom}(A)\subset L^2(\Lambda)\to L^2(\Lambda)$ generates a $C_0$-semi-group on $L^2(\Lambda)$, $B\colon L^2(\Lambda)\to L^2(\Lambda)$ is linear bounded operator and $F\colon L^2(\Lambda)\to L^2(\Lambda)$ is not necessarily of Nemytskii type as in \eqref{eq:SPDE}. Typical examples for $A$ include $A=\nu\Updelta$ (diffusivity $\nu>0$), $A=-(-\Updelta)^\rho$ (fractional Laplacian with $\rho \ge 1/2$) and $A=\partial_y (a(y)\partial_y\MTemptyplaceholder)$ (divergence form), while $B$ can for example have the form $B=\sigma\identity$ (space-time white noise with intensity $\sigma$) or $B=-(-\Updelta)^{-\kappa}$ with $\kappa>0$, the latter inducing higher spatial smoothness for the solution. Sufficient conditions for the wellposedness of \eqref{eq:SPDE_GeneralSemiLinear} as a (unique) mild and weak solution on $L^2(\Lambda)$ are provided in Theorems 7.2 and 7.5 of \citet{DaPrato2014} and Proposition G.0.5 of \citet{Roeckner2015}. If
\begin{equation*}
    \sup_{t\in[0,T]}\EV[f]{\exp(\delta\norm{B\inv F(X_t)}_{L^2(\Lambda)}^2)}<\infty
\end{equation*}
for some $\delta>0$, Theorem 10.18 of \citet{DaPrato2014} shows that the log-likelihood with respect to the nonlinearity $F$ is given by
\begin{equation*}
    \ell (F) = \int_0^T \iprod{B\inv F(X_t)}{B\inv[\D X_t - A X_t\,\D t]}_{L^2(\Lambda)}-\frac{1}{2}\int_0^T \norm{B\inv F(X_t)}_{L^2(\Lambda)}^2\,\D t.
\end{equation*}
Analogously to \eqref{eq:aux_Def_Likelihood}, we define the first integral as the $L^2(\prob[0]{})$-limit
\begin{equation*}
	\sum_{k\in\mathbb{N}}\int_0^T\iprod*{F(X_t)}{(B\inv)^\ast e_k}_{L^2(\Lambda)}\left(\iprod*{(B\inv)^\ast e_k}{\D X_t}_{L^2(\Lambda)} - \iprod*{A^\ast (B\inv)^\ast e_k}{X_t}_{L^2(\Lambda)}\,\D t\right)
\end{equation*}
for an orthonormal basis $(e_k)_{k\in\mathbb{N}}$ of $L^2(\Lambda)$ such that $e_k\in \operatorname{dom}(A^\ast(B\inv)^\ast)$, $k\in\mathbb{N}$. 

This shows that the likelihood is well-defined and the Bayesian procedure is applicable under mild assumptions. The likelihood is quadratic in the operator $F$, so Gaussian conjugacy is preserved (assuming a suitable Gaussian prior). The computation of the likelihood, however, requires knowledge of the operators $A$ and $B$. By adjusting the likelihood, the analysis presented in this paper transfers to the setting $A=\nu\Updelta$ and $B=\sigma\identity$ for known $\nu,\sigma^2>0$.

\subsubsection{Spatial ergodicity for more general SPDEs}\label{subsec:Generalisation2}

The analysis of the posterior distribution in the general setting \eqref{eq:SPDE_GeneralSemiLinear} is delicate. On the one hand, by following the proofs of Propositions \ref{prop:DensityBoundedDerivative} and Theorem \ref{thm:SubgaussianConcentration}, the concentration results for the averages $\int_0^T\frac{1}{\lambda} \int_{\lambda\bar{\Lambda}} g(X_t(y))\,\D y\D t$ can be generalised beyond $A=\Updelta$ and $B=\identity$, provided that there exists some $0<\alpha<1$ such that
\begin{equation*}
    \int_0^t \norm{B G_{t-s}^\lambda(y,\eta)}_{L^2(\Lambda)}^2\,\D s\sim t^{-\alpha},\quad 0\le t\le T,
\end{equation*}
for all $\lambda\ge 1$ where $G^\lambda$ is the Green function of $A$ on $\Lambda$. On the other hand, if $B\inv$ is an unbounded operator on $L^2(\Lambda)$ or $F$ is not of Nemytskii-type, the Hellinger semi-metric $h_\lambda$ no longer takes the form of an average, since 
\begin{equation*}
	h_\lambda(F,G)^2 = \frac{1}{\lambda}\int_0^T\norm{B\inv(F(X_t)-G(X_t))}_{L^2(\Lambda)}^2\,\D t.
\end{equation*}
Whereas $B\inv$ is in general a local operator and spatial ergodicity is expected, the analysis requires a precise understanding of the joint action of $B\inv$ and any candidate functions $F, G$ on $X_t$ and lies beyond the scope of this work.

\subsubsection{Small noise asymptotics}\label{subsec:Rescaling}

Suppose we observe $Y^\sigma = (Y_t(y)\,\colon \, (t,y)\in[0,T] \times\bar{\Lambda})$ solving the SPDE \eqref{eq:Y_rescaled} on a domain $\bar{\Lambda}\subset\mathbb{R}$ with known diffusivity $\nu>0$ and known noise level $\sigma>0$. Provided that $f$ is Lipschitz-continuous, the log-likelihood reads
\begin{equation*}
    \ell(f) = \frac{1}{\sigma^2}\int_0^T \iprod{F(Y_t)}{\D Y_t-\nu\Updelta Y_t}_{L^2(\bar{\Lambda})} - \frac{1}{2\sigma^2}\int_0^T \norm{F(Y_t)}_{L^2(\bar{\Lambda})}^2\,\D t.
\end{equation*}
Analogously to Subsection \ref{subsec:Bayes_Setting_Posterior}, Gaussian conjugacy holds, and the posterior distribution can be analysed in the small noise asymptotic $\sigma\to 0$ with diffusivity $\nu>0$ held constant. Using Theorem 3 of \citet{Ibragimov2003} and the proof strategy of Proposition 5.2 of \citet{vanzantenRateConvergenceMaximum2005a}, it can be shown that the appropriate limiting metric is given by
\begin{equation*}
    \int_0^T\int_{\bar{\Lambda}}\left(g(Y_t^{\operatorname{det}}(y))-\tilde{g}(Y_t^{\operatorname{det}}(y))\right)^2\,\D y\D t,\quad g,\tilde{g}\in\Theta,
\end{equation*}
where $Y^{\operatorname{det}}$ satisfies the deterministic PDE \eqref{eq:Y_rescaled} with $\sigma=0$ and $f=f_0$. Posterior contraction in this metric can be shown along the lines of Subsection 3.1 of \citet{vandermeulenConvergenceRatesPosterior2006}.

\section{The Malliavin calculus and the proof of Theorem \ref{thm:SubgaussianConcentration}}\label{sec:Proof_Subgaussian_Concentration}

The Malliavin calculus was initiated by \citet{malliavinStochasticCalculusVariation1978} to obtain a probabilistic proof of Hörmander's Theorem and has since experienced significant extensions and new areas of application. The interested reader is referred to the monographs \citet{SanzSole2005, nualartMalliavinCalculusRelated2006, bogachevDifferentiableMeasuresMalliavin2010, nualartIntroductionMalliavinCalculus2018, ikedaIntroductionMalliavinsCalculus1984, nourdinNormalApproximationsMalliavin2012, malliavinStochasticCalculusVariations2006, malliavinStochasticAnalysis1997} and references therein. Its interpretation as an \enquote{infinite-dimensional differential calculus on the Wiener space} \citep{nualartMalliavinCalculusRelated2006} makes the Malliavin calculus appealing in our setting, since we are aiming to control quantities, e.g.\ $\mathcal{G}_\lambda$, which are nonlinear transformations of the Gaussian process $(W_t)_{t\in[0,T]}$. If $X$ is the solution to the semi-linear SPDE \eqref{eq:SPDE} and $f\in C^1(\mathbb{R})$, then the Malliavin derivative $(\tau,z)\mapsto\mathcal{D}_{\tau,z}X_t(y)\in L^2([0,T],L^2(\Lambda))$ of $X_t(y)$ for any time point $0<t\le T$ and location $y\in\Lambda$ exists. It satisfies the linear PDE with random coefficients
\begin{equation}
	\begin{cases}
		\partial_t \mathcal{D}_{\tau,z} X_t(y) = \partial_y^2 \mathcal{D}_{\tau,z}X_t(y) + f'(X_t(y))\mathcal{D}_{\tau,z}X_t(y),& 0\le\tau<t\le T,\quad  z\in\Lambda\\
		\mathcal{D}_{\tau,z}X_\tau(y) = \delta_{\tau,z}(y),&z\in\Lambda,
	\end{cases}\label{eq:malliavin_derivative}
\end{equation}
and $\mathcal{D}_{\tau,z}X_t(y)=0$ for $0< t<\tau\le T$, $z,y\in\Lambda$. Assume that $g\in C^1(\mathbb{R})$. We apply the Malliavin calculus at two instances in the proof of the sub-Gaussian concentration of Theorem \ref{thm:SubgaussianConcentration}. The first instance is the well-known Clark-Ocone formula in the form given by Proposition 6.3 in \citep{chenSpatialErgodicitySPDEs2021}, which implies
	\begin{equation}
		\mathcal{G}_{\lambda} - \EV*[f]{\mathcal{G}_{\lambda}} = \mathcal{M}_T\coloneq \int_0^T \int_\Lambda \EV*[f]{\mathcal{D}_{\tau,z}\mathcal{G}_{\lambda}\given \mathscr{F}_{\tau}}\,\mathcal{W}(\D z,\D \tau),\label{eq:Clark_ocone}
	\end{equation}
	almost surely, where $(\mathscr{F}_t)_{t\ge 0}$ is the natural filtration generated by $\mathcal{W}$.
	 The stochastic integral $\mathcal{M}$ is a martingale (in time) with quadratic variation
	\begin{equation*}
		\langle \mathcal{M}\rangle_T = \int_0^T \int_\Lambda \EV[f]{\mathcal{D}_{\tau,z}\mathcal{G}_\lambda\given \mathcal{F}_\tau}^2 \,\D z\D \tau = \int_0^T\int_\Lambda \bigg(\int_\tau^T \frac{1}{\lambda}\int_\Lambda  \EV[f]{g'(X_t(y))\mathcal{D}_{\tau,z} X_t(y)\given \mathcal{F}_\tau}\,\D y\D t\bigg)^2\,\D z\D \tau.
	\end{equation*}
	This representation allows us to control deviations of $\mathcal{G}_{\lambda}$ by bounding the quadratic variation $\langle \mathcal{M}\rangle_T$ and applying a (martingale) Bernstein inequality.
	 Lemma 4.6 of \citet{gaudlitzNonparametricEstimationReaction2023} shows that $\langle \mathcal{M}\rangle_T$ can be controlled once we can bound the conditional expectation
	 \begin{equation*}
		\abs{\EV[f]{g'(X_t(y))\given \mathcal{F}_\tau}}.
	 \end{equation*}
  At this point, a second application of the Malliavin calculus allows us to control the conditional expectation via density estimates. Importantly, the subsequent bound on the derivatives of the density of $X_t(y)$ does not depend on the initial condition.
\begin{proposition}\label{prop:DensityBoundedDerivative}
    Let $X_0\in C(\Lambda)$ be any deterministic initial condition, $N\in \mathbb{N}_0$, $K>0$ and $f\in\mathcal{F}_{N+2,K}$. Then there exists a constant $0<p_{\max}^{(N)}<\infty$, depending on $T$, $K$ and $N$, such that we have for all $0\le n\le N$, $0<t\le T$ and $\lambda\ge 1$
		\begin{equation*}
			\abs{\partial_x^n p_{f,\lambda}(t,y,x)}\le p_{\max}^{(N)}t^{-(N+1)/4},\quad x\in\mathbb{R}, y\in\Lambda.
		\end{equation*}
\end{proposition}
\begin{proof}
    The proof follows the classical approach for proving the smoothness of densities using the Malliavin calculus of \citet{watanabeLecturesStochasticDifferential1984, ikedaIntroductionMalliavinsCalculus1984}, see also Section 2.1.4 of \citet{nualartMalliavinCalculusRelated2006} and Section 7.2 of \citet{nualartIntroductionMalliavinCalculus2018} for modern expositions. It is carried out in Subsection \ref{subsec:DensityDerivatives}.
\end{proof}
\begin{remark}[Assuming $f\in \mathcal{F}_{N+2,K}$]\label{rmk:OnDensity}
	Proposition \ref{prop:DensityBoundedDerivative} shows that assuming $f\in \mathcal{F}_{N+2,K}$ in Theorem \ref{thm:SubgaussianConcentration} allows us to control the derivatives of the density up to the order $N\in\mathbb{N}_0$. The gap of 2 between the regularity of $f$ and the regularity of the density is commonly found when proceeding similarly to Theorem 2.1.4 of \citet{nualartMalliavinCalculusRelated2006}, see \citep{carmonaRandomNonlinearWave1988a, milletStochasticWaveEquation1999, marquez-carrerasStochasticPartialDifferential2001, dalangPotentialTheoryHyperbolic2004, SanzSole2005, nualartExistenceSmoothnessDensity2007a, dalangHittingProbabilitiesSystems2009,  marinelliExistenceRegularityDensity2013}. The gap has been reduced by \citet{nourdinDensityFormulaConcentration2009, romitoSimpleMethodExistence2018}, but their results do not apply to our required case $N=1$. Providing a tighter connection between the regularity of $f$ and of the density poses an intricate question for future research.
\end{remark}
  
We proceed with the proof of the sub-Gaussian concentration for spatio-temporal averages in Theorem \ref{thm:SubgaussianConcentration} (\ref{num:SubgaussianThm2}) in three steps.
 When $g\in C^1(\mathbb{R})$, we combine the integration-by-parts formula with the bounds for $\partial_x p_{f,\lambda}(t,y,x)$ from Proposition \ref{prop:DensityBoundedDerivative} to obtain a bound for $\abs{\EV[f]{g'(X_t(y))\given \mathcal{F}_\tau}}$ in terms of $\norm{g}_{L^1(\mathbb{R})}$ in Lemma \ref{lem:Bayes_BoundConditionalExpectation}.
  Subsequently, when $g\in C^1(\mathbb{R})$ we apply Lemma 4.6 of \citet{gaudlitzNonparametricEstimationReaction2023} to obtain the claimed sub-Gaussian concentration for $\mathcal{G}_{\lambda}(g)$ in Lemma \ref{lem:General_Deviation_Control}.
   We conclude by an approximation argument for $g\in L^1(\mathbb{R})$.

\subsubsection*{Step 1: Bounding the conditional expectation}

We denote by $\mathcal{P}$ the set of all functions $\mathbb{R}\to\mathbb{R}$ that are of at most polynomial growth.
 The subsequent lemma bounds the conditional expectation $\abs{\EV[f_0]{g'(X_t(y))\given \mathcal{F}_\tau}}$ in terms of $\norm{g}_{L^1(\mathbb{R})}$ by combining the integration-by-parts formula with the bound for the derivative of the density $\partial_xp_{f,\lambda}(t,y,x)$ from Proposition \ref{prop:DensityBoundedDerivative}.
\begin{lemma}\label{lem:Bayes_BoundConditionalExpectation}
	Let $f_0\in\Theta$ and $g\in L^1(\mathbb{R})\cap C^1(\mathbb{R})\cap \mathcal{P}$. Then
	\begin{equation*}
		\abs{\EV[f_0]{g'(X_t(y))\given \mathcal{F}_\tau}}\le \norm{g}_{L^1(\mathbb{R})}p_{\max}^{(1)}(t-\tau)^{-1/2},\quad y\in\Lambda,\quad 0\le\tau<t\le T.
	\end{equation*}
\end{lemma}
\begin{proof}
	Using the Markov property of $t\mapsto X_t$ we can write
	\begin{equation*}
		\EV[f_0]{g'(X_t(y))\given \mathcal{F}_\tau}  = \EV[f_0]{g'(X_t(y))\given X_\tau} = \left.\EV[f_0]{g'(X_{t-\tau}(y))}\right|_{X_0=X_\tau}
	\end{equation*}
	Let $\xi=X_\tau$. In the next step, we rewrite the expectation on the right-hand-side as an integral with respect to the density $p_{f_0,\lambda}(s,y,\MTemptyplaceholder)$ of the process $(X_s(y)\,\vert\, y\in\Lambda, 0<s<T-\tau)$ conditioned to start at the initial condition $\xi$, which is treated as deterministic due to the conditioning.
	 Noting that the density $p_{f_0,\lambda}(s,y,\MTemptyplaceholder)$ depends on initial condition $\xi\in C(\Lambda)$, we find
	\begin{equation*}
		\EV[f_0]{g'(X_t(y))\given \mathcal{F}_\tau}=\left.\int_{\mathbb{R}} g'(x)p_{f_0,\lambda}(t-\tau,y,x) \,\D x\right|_{X_0=\xi}
	\end{equation*}
	for all $y\in\Lambda$ and $0\le\tau<t\le T$.
	 Proposition \ref{prop:DensityBoundedDerivative} with $f=f_0$ shows that $p_{f_0,\lambda}(t,y,\MTemptyplaceholder)\in C^1(\mathbb{R})$ for any deterministic and continuous initial condition.
	  Furthermore, Part \eqref{num:cordensitybounds_c} of Corollary \ref{cor:DensityBounds} (below) shows that $p_{f_0,\lambda}(t,y,\MTemptyplaceholder)$ has Gaussian tails for any deterministic initial condition.
	   Since $g'$ grows at most polynomially, we integrate by parts without boundary terms to obtain
	\begin{equation*}
		\abs{\EV[f_0]{g'(X_t(y))\given \mathcal{F}_\tau}} = \abs*{\int_{\mathbb{R}} g(x) \partial_x p_{f_0,\lambda}(t-\tau,y,x)\,\D x}.
	\end{equation*}
	Applying the $\xi$-independent upper bound for $\norm{\partial_x p_{f_0,\lambda}(t,y,\MTemptyplaceholder)}_{L^\infty(\mathbb{R})}$ from Proposition \ref{prop:DensityBoundedDerivative} with $f=f_0$, we find
	\begin{equation*}
		\abs{\EV[f_0]{g'(X_t(y))\given \mathcal{F}_\tau}}  \le \norm{g}_{L^1(\mathbb{R})}p_{\max}^{(1)}(t-\tau)^{-1/2}.
	\end{equation*}
\end{proof}

\subsubsection*{Step 2: Subgaussian concentration for $g\in C^1(\mathbb{R})$}

\begin{lemma}[Subgaussian concentration of $\mathcal{G}_{\lambda}(g)$ -- $C^1(\mathbb{R})$-version]\label{lem:General_Deviation_Control}
	Let $f_0\in \Theta$ and $g\in L^1(\mathbb{R})\cap C^1(\mathbb{R})\cap\mathcal{P}$. Then there exists a constant $C$, depending on $\norm{f_0}_{C^3(\mathbb{R})}$ and $T$, such that we can bound
	\begin{equation}
		\prob*[f_0]{\abs{\mathcal{G}_{\lambda}(g)-\EV[f_0]{\mathcal{G}_{\lambda}(g)}}\ge \lambda^{-1/2} x}\le 2\exp\bigg(-\frac{x^2}{2 C \norm{g}_{L^1(\mathbb{R})}^2}\bigg),\quad x\ge 0.\label{eq:Subgaussian_1}
	\end{equation}
\end{lemma}
\begin{proof}
	Using the Clark-Ocone formula, we can write
	\begin{equation*}
		\mathcal{G}_{\lambda}(g)-\EV[f_0]{\mathcal{G}_{\lambda}(g)}=\mathcal{M}_T\coloneq\int_0^T\int_\Lambda \EV*[f_0]{\mathcal{D}_{\tau,z}\mathcal{G}_{\lambda}(g)\given\mathcal{F}_\tau}\,\mathcal{W}(\D z,\D \tau),
	\end{equation*}
	which is a martingale in time under $\prob[f_0]{}$ evaluated at time $T$ with quadratic variation
	\begin{equation*}
		\langle \mathcal{M} \rangle_T=\frac{1}{\lambda^2}\int_0^T \int_\Lambda \bigg(\int_\tau^T \int_\Lambda \EV[f_0]{g'(X_t(y))\mathcal{D}_{\tau,z}X_t(y)\given \mathcal{F}_\tau}\,\D y\D t\bigg)^2\,\D z\D\tau.
	\end{equation*}
    The bound $\abs{\EV[f_0]{g'(X_t(y))\given \mathcal{F}_\tau}}\le \norm{g}_{L^1(\mathbb{R})}p_{\max}^{(1)}(t-\tau)^{-1/2}$ from Lemma \ref{lem:Bayes_BoundConditionalExpectation} implies that we can use Lemma 4.6 of \citet{gaudlitzNonparametricEstimationReaction2023} with $\kappa(t,\tau) = \norm{g}_{L^1(\mathbb{R})} p_{\max}^{(1)}(t-\tau)^{-1/2}$ to obtain
	\begin{align*}
		\langle \mathcal{M}\rangle_T&
		 \le \frac{1}{\lambda}\norm{g}_{L^1(\mathbb{R})}^2 (p_{\max}^{(1)})^2  \int_0^T \bigg(2(T-\tau)^{1/2} + e^{\Lip{f_0} T}\Lip{f_0} \frac{4}{3}(T-\tau)^{3/2}\bigg)^2\,\D\tau\\
		& = \frac{1}{\lambda}\norm{g}_{L^1(\mathbb{R})}^2 (p_{\max}^{(1)})^2 \left(2T^2 + \frac{16}{9}T^3 e^{\Lip{f_0} T}\Lip{f_0}+\frac{4}{9}T^4e^{2\Lip{f_0} T}\Lip{f_0}^2\right).
	\end{align*}
	Therefore, we can almost surely bound the quadratic variation by
	\begin{equation*}
		\langle \mathcal{M}\rangle_T \le C^2 \frac{1}{\lambda} \norm{g}_{L^1(\mathbb{R})}^2,
	\end{equation*}
	with a suitable constant $C$ depending on $T$ and $\norm{f_0}_{C^3(\mathbb{R})}$. The Bernstein inequality for continuous martingales (see page 153 of \citet{RevuzYor1999}) implies that for $x\ge 0$
	\begin{equation*}
		\prob*[f_0]{\abs{ \mathcal{G}_{\lambda}(g)-\EV[f_0]{\mathcal{G}_{\lambda}(g)}} \ge \frac{C}{\sqrt{\lambda}}  x} = \prob*[f_0]{\abs{ \mathcal{M}_T} \ge \frac{C}{\sqrt{\lambda}} x, \langle \mathcal{M}\rangle_T \le\frac{C^2}{\lambda} \norm{g}_{L^1(\mathbb{R})}^2}  \le 2\exp\bigg(-\frac{x^2}{2  \norm{g}_{L^1(\mathbb{R})}^2}\bigg).
	\end{equation*}
	This concludes the proof of the sub-Gaussian concentration inequality \eqref{eq:Subgaussian_1}.
\end{proof}

\subsubsection*{Step 3: Approximation and conclusion}

\begin{proof}[Proof of Theorem \ref{thm:SubgaussianConcentration} \eqref{num:SubgaussianThm2}]
	Let $\tilde{\mathcal{G}}_\lambda(g) =  \mathcal{G}_\lambda(g)-\EV[f_0]{\mathcal{G}_{\lambda}(g)}$. Fix any $x\ge 0$ and take an approximating sequence $(g_n)_{n\in\mathbb{N}}\subset L^1(\mathbb{R})\cap C^1(\mathbb{R})$ with compact, growing supports and $\norm{g-g_n}_{L^1(\mathbb{R})}\to 0$ as $n\to \infty$. Since the map $g\mapsto \tilde{\mathcal{G}}_{\lambda}(g)$ is linear, we obtain for all $n\in\mathbb{R}$ the bound
	\begin{align*}
		\prob[f_0]{\abs{\tilde{\mathcal{G}}_{\lambda}(g)}\ge \lambda^{-1/2} x}&=\prob[f_0]{\abs{\tilde{\mathcal{G}}_{\lambda}(g_n) - \tilde{\mathcal{G}}_{\lambda}(g_n-g)}\ge \lambda^{-1/2} x}\\
        &\le  \prob[f_0]{\abs{\tilde{\mathcal{G}}_{\lambda}(g_n)}\ge \lambda^{-1/2} x/2} + \prob[f_0]{\abs{\tilde{\mathcal{G}}_{\lambda}(g_n-g)}\ge \lambda^{-1/2} x/2}.
	\end{align*}
	Applying Lemma \ref{lem:General_Deviation_Control} to $\tilde{\mathcal{G}}_{\lambda}(g_n)$ and Markov's inequality to $\tilde{\mathcal{G}}_{\lambda}(g_n-g)$, we find
	\begin{equation*}
		\prob[f_0]{\abs{\tilde{\mathcal{G}}_{\lambda}(g)}\ge \lambda^{-1/2} x}\le 2\exp\bigg(-\frac{x^2}{4C^2\norm{g_n}_{L^1(\mathbb{R})}^2}\bigg) + 2\lambda^{1/2}\frac{\EV[f_0]{\abs{\tilde{\mathcal{G}}_{\lambda}(g_n-g)}}}{ x}.
	\end{equation*}
	The triangle inequality and the upper density bound from Part \eqref{num:cordensitybounds_a} of Corollary \ref{cor:DensityBounds} imply
	\begin{equation*}
		\EV[f_0]{\abs{\tilde{\mathcal{G}}_{\lambda}(g_n-g)}}\le \frac{2}{\lambda} \EV[f_0][][\bigg]{\int_0^T\int_\Lambda \abs*{g_n(X_t(y))-g(X_t(y))}\,\D y\D t}\lesssim_{T,\norm{f_0}_{C^3(\mathbb{R})}}\norm{g_n-g}_{L^1(\mathbb{R})}\to 0.
	\end{equation*}
    as $n\to\infty$, which shows the sub-Gaussian concentration inequality \eqref{eq:Subgaussian_2}. 
\end{proof}

\section{Proofs for Sections \ref{subsec:Posterior_Contraction} and \ref{subsec:BvM}} \label{sec:Bayes_proofContraction}

Let us first recall a few facts for the wavelet approximation spaces $V_M=\operatorname{span}(\psi_\mu\,\colon\, \abs{\mu}\leq M)$. Define the associated projection operators 
\begin{align*}
P_Mf\coloneq \sum_{\abs{\mu}\leq M}\iprod{f}{\psi_{\mu}}_{L^2(\Xi)}\psi_{\mu},\quad f\in L^2(\Xi). 
\end{align*}
Standard wavelet-theory (e.g., Proposition 4.1.5 of \citet{gineMathematicalFoundationsInfiniteDimensional2015} which also applies to the boundary-corrected Daubechies wavelets) gives for $f\in H^\beta(\Xi)$ (or $f\in C^\beta(\Xi)$) the classical wavelet (Jackson) estimates 
\begin{equation}
\norm*{f-P_Mf}_{L^2(\Xi)}\lesssim 2^{-M\beta}\norm{f}_{H^\beta(\Xi)} \quad \left(\text{ or }\norm*{f-P_Mf}_{L^\infty(\Xi)}\lesssim 2^{-M\beta}\norm{f}_{C^\beta(\Xi)}\quad\right)\label{eq:Bayes_SupNorm_Approximation}
\end{equation}
with implied constants depending on $\beta$, the wavelet basis and the domain $\Xi$. 

The proof of the posterior contraction (Theorem \ref{thm:PosteriorContraction}) proceeds using a variant of Theorem 2.1 and Lemma 2.2 of \citet{vandermeulenConvergenceRatesPosterior2006}, see also Theorem 7 of \citet{nicklNonparametricStatisticalInference2020}. It relies on the Bernstein inequality for martingales and carries over to the present setting with minor notational changes that are left to the reader. Define the metric
\begin{equation*}
    d_\lambda(f,g)^2\coloneq \EV[f_0]{h_\lambda(f,g)^2},\quad f,g\in L^\infty(\mathbb{R}),
\end{equation*}
For $f,g\in \Theta$ the upper and lower density bounds of Corollary \ref{cor:DensityBounds} imply the norm equivalence
\begin{equation}
    \tilde{c}\norm{f-g}_{L^2(\Xi)}\le d_\lambda(f,g)\le \tilde{C}\norm{f-g}_{L^2(\Xi)},\label{eq:norm_equivalence_d}
\end{equation}
where $0<\tilde{c},\tilde{C}<\infty$ are constants depending only on $\Lip{f_0}$, $T$ and $\Xi$.
\begin{theorem}[Theorem 2.1, Lemma 2.2 of \citet{vandermeulenConvergenceRatesPosterior2006}]\label{thm:TripleVan}
	Let $f_0\in \Theta$ and let $\mathfrak{B}_\lambda\subset \Theta$ be a $\lambda$-dependent parameter set. Let $\bar{\epsilon}_\lambda\to 0$ such that $\lambda\bar{\epsilon}_\lambda^2\to\infty$. Suppose that for any constants $0<\smash{\underbar{c}},\bar{c},C_1,C_2,C_3,C_4<\infty$ with $C_1-C_3\ge \max(\bar{c},C_4)$ there exist measurable sets $\mathcal{B}_\lambda\subset\mathfrak{B}_\lambda$ and $C_2,C_3>0$ such that
	\begin{align}
		 & \text{Sieve set: } \Pi_\lambda(\mathcal{B}_\lambda^c)\le e^{-C_1 \lambda\bar{\epsilon}_\lambda^2}\label{eq:Cond1}                            \\
		 & \text{Entropy bound: } \log(N(\mathcal{B}_\lambda,d_\lambda,\bar{\epsilon}_\lambda))\le C_2 \lambda\bar{\epsilon}_\lambda^2,\label{eq:Cond2}            \\
		 & \text{Small ball probability: } \Pi_\lambda(f\colon  d_\lambda(f,f_0)\le\bar{\epsilon}_\lambda)\ge e^{-C_3\lambda\bar{\epsilon}_\lambda^2}\label{eq:Cond3}, \\
		\begin{split}
			 & \lim_{\lambda\to \infty}\mathrm{P}_{f_0}\Big(\big\{\Forall f\in\mathfrak{B}_\lambda\text{ s.t. }h_\lambda(f,f_0)\ge C_4 \bar{\epsilon}_\lambda\colon \quad \smash{\underbar{c}}d_\lambda(f,f_0)\le h_\lambda (f,f_0)\big\} \\
			 & \quad\cap \big\{\Forall f,g\in\mathfrak{B}_\lambda\text{ s.t. }h_\lambda(f,g)\ge C_4 \bar{\epsilon}_\lambda\colon \quad h_\lambda(f,g)\le \bar{c} d_\lambda(f,g)\big\}\Big)=1.
		\end{split}\label{eq:Cond4}
	\end{align}
	Then, for every $K_\lambda\to\infty$,
	\begin{equation*}
		\Pi_\lambda(f\in \mathfrak{B}_\lambda\colon  h_\lambda(f,f_0)\ge K_\lambda\bar{\epsilon}_\lambda\,\vert\, X^\lambda)\xrightarrow{\prob[f_0]{}}0,\quad \lambda\to \infty.
	\end{equation*}
\end{theorem}

\begin{proof}[Proof of Theorem \ref{thm:PosteriorContraction}]

	We first prove posterior concentration in the Hellinger semi-metric $h_{\lambda}$ by verifying the conditions of  Theorem \ref{thm:TripleVan} with $\bar{\epsilon}_{\lambda}\coloneq L\varepsilon_{\lambda}$ and with the sets $\mathfrak{B}_\lambda=V_M$, $\mathcal{B}_{\lambda}=V_M\cap B(0,\norm{}_{H^{\beta_0}(\Xi)},L\sqrt{\lambda}\epsilon_{\lambda})$ for some large enough $L>0$. For such $L$, parts \eqref{lem:aux_rescaled_Prior_sieve}-\eqref{lem:aux_rescaled_Prior_smallball} of Lemma \ref{lem:Prior} (below) already show that conditions \eqref{eq:Cond1}-\eqref{eq:Cond3} hold, recalling the norm equivalence \eqref{eq:norm_equivalence_d}. We are therefore left with verifying condition \eqref{eq:Cond4}. To this end, the main ingredients are the approximation bound \eqref{eq:Bayes_SupNorm_Approximation} and the sub-Gaussian concentration of the Hellinger semi-metric $h_\lambda$ (Lemma \ref{lem:supremum_Control_Hellinger}). Suppose that
    \begin{equation}
        2\smash{\underbar{c}} d_\lambda(f,g)\le h_\lambda(f,g)\le \frac{1}{2}\bar{c}d_\lambda(f,g),\quad f,g\in V_M.\label{eq:Claim_Proof_Contraction}
    \end{equation}
    for some $\bar{c}>0$ and $0<\smash{\underbar{c}}<1/2$. This condition already implies the second condition in \eqref{eq:Cond4}. To prove its first condition, we note that $P_Mf_0$ is the projection of $f_0$ onto $V_M$. Since $f_0\in C^{\beta}(\Xi)$ and the frequency cut-off satisfies $2^{-M\beta}\lesssim 2^{-\beta/(2\beta+1)}\leq  \epsilon_\lambda$, we can apply the Jackson estimate \eqref{eq:Bayes_SupNorm_Approximation} to obtain $\norm{f_0-P_M f_0}_{L^\infty(\Xi)}\lesssim \norm{f_0}_{C^\beta(\Xi)}\epsilon_\lambda$. Since $h_\lambda(f_0,P_Mf_0)\le \sqrt{T}\norm{f_0-P_Mf_0}_{L^\infty(\Xi)}$, we deduce the bias bounds $h_\lambda(f_0,P_Mf_0)\le C_{T,f_0}\epsilon_\lambda$ and $d_\lambda(f_0,P_M f_0)\le C_{T,f_0}\epsilon_\lambda$ for some constant $0<C_{T,f_0}<\infty$. Consequently, for $f\in V_M$ with $h_\lambda(f,f_0)\ge C_4\epsilon_\lambda$, we have
    \begin{align*}
        \smash{\underbar{c}} d_\lambda(f,f_0)&\le \smash{\underbar{c}} d_\lambda(f,P_Mf_0) + \smash{\underbar{c}} d_\lambda(f_0,P_Mf_0)\le \frac{1}{2}h_\lambda(f,P_Mf_0) + \smash{\underbar{c}}C_{T,f_0}\epsilon_\lambda\\
        &\le \frac{1}{2}\left(h_\lambda(f,f_0) + h_\lambda(f_0,P_Mf_0)\right) + \smash{\underbar{c}} C_{T,f_0}\epsilon_\lambda\\
        &\le \frac{1}{2}\left(h_\lambda(f,f_0) + C_{T,f_0}\epsilon_\lambda\right) + C_{T,f_0}\epsilon_\lambda\\
        &\le h_\lambda(f,f_0)\left(\frac{1}{2} + \frac{C_{T,f_0}}{2C_4} + \frac{\smash{\underbar{c}}C_{T,f_0}}{C_4}\right).
    \end{align*}
    By readjusting the constant $C_4$, we find $\smash{\underbar{c}} d_\lambda(f,f_0)\le h_\lambda(f,f_0)$, as required.	The probability of the event \eqref{eq:Claim_Proof_Contraction} equals
	\begin{align*}
		 & \prob[f_0][][\bigg]{\Forall f,g\in V_M,f\neq g\colon\quad 4\smash{\underbar{c}}^2-1\le \frac{h_\lambda(f,g)^2}{d_\lambda(f,g)^2}-1 \le \frac{1}{4}\bar{c}^2 -1} \\
		 &\hspace{15em}  \ge 1-\prob[f_0][][\bigg]{\sup_{f,g\in V_M, f\neq g}\frac{\abs{h_\lambda(f,g)^2-d_\lambda(f,g)^2}}{d_\lambda(f,g)^2}>K},
	\end{align*}
	where we chose $K=\min(1-4\smash{\underbar{c}}^2,\bar{c}^2/4-1)>0$. Lemma \ref{lem:supremum_Control_Hellinger} (below) with $x=K\lambda^{1/2}/\hat{C}\ge 0$ yields 
	\begin{equation}
		\prob[f_0][][\bigg]{\forall f,g\in V_M \colon\quad 2\smash{\underbar{c}} d_\lambda(f,g)\le h_\lambda(f,g)\le \frac{1}{2}\bar{c}d_\lambda(f,g)}\le \exp\bigg(\hat{c}\operatorname{dim}(V_M) - \frac{K^2}{\lambda \hat{C}^2}\bigg)\to 0,
        \label{eq:ProofContractionGoodEvent}
	\end{equation}
    as $\lambda\to\infty$, because $\lambda^{-1/2}\sqrt{\operatorname{dim}(V_M)}\sim \lambda^{-\beta/(2\beta+1)}\to 0$. This statement verifies \eqref{eq:Cond4} and we can apply Theorem \ref{thm:TripleVan} to find posterior contraction rates in the (random) Hellinger semi-metric $h_\lambda$ such that
	\begin{equation*}
		\Pi_\lambda(f\in V_M \colon h_\lambda(f,f_0)\ge K_\lambda \epsilon_\lambda\, \vert\,  X^\lambda)\xrightarrow{\prob[f_0]{}}0
	\end{equation*}
	for every $K_\lambda\to\infty$ as $\lambda\to\infty$. To extend the posterior concentration to $d_\lambda$, we recall \eqref{eq:ProofContractionGoodEvent} and the bias bound $\max(h_\lambda(f_0,P_Mf_0),d_\lambda(f_0,P_Mf_0))\le C_{T,f_0}\epsilon_\lambda$ to obtain for any $\tilde{K}_\lambda\to\infty$ as $\lambda\to\infty$ the convergence
	\begin{align*}
		\Pi_\lambda\left(f\in V_M\colon d_\lambda(f,f_0)\ge \tilde{K}_\lambda \epsilon_\lambda\, \vert\,  X^\lambda \right)\hspace{-10em}&\\
        &\le \Pi_\lambda\bigg(f\in V_M\colon \frac{h_\lambda(f,P_M f_0)}{2\smash{\underbar{c}}} + C_{T,f_0}\epsilon_\lambda \ge \tilde{K}_\lambda \epsilon_\lambda\, \vert\,  X^\lambda \bigg) + \smallo_{\prob[f_0]{}}(1)\\
        &\le \Pi_\lambda\bigg(f\in V_M\colon \frac{h_\lambda(f,f_0) + h_\lambda(f_0,P_Mf_0)}{2\smash{\underbar{c}}} + C_{T,f_0}\epsilon_\lambda \ge \tilde{K}_\lambda \epsilon_\lambda\, \vert\,  X^\lambda \bigg) + \smallo_{\prob[f_0]{}}(1)\\
        &\le \Pi_\lambda\bigg(f\in V_M\colon \frac{h_\lambda(f,f_0) + C_{T,f_0}\epsilon_\lambda}{2\smash{\underbar{c}}} + C_{T,f_0}\epsilon_\lambda \ge \tilde{K}_\lambda \epsilon_\lambda\, \vert\,  X^\lambda \bigg) + \smallo_{\prob[f_0]{}}(1)\\
        &= \Pi_\lambda\bigg(f\in V_M\colon h_\lambda(f,f_0)\ge \Big(\tilde{K}_\lambda - C_{T,f_0} - \frac{C_{T,f_0}}{2\smash{\underbar{c}}}\Big)2\smash{\underbar{c}}\epsilon_\lambda\, \vert\,  X^\lambda \bigg) + \smallo_{\prob[f_0]{}}(1)\xrightarrow{\prob[f_0]{}} 0.
	\end{align*}
    Combining the result from the last display with the norm equivalence \eqref{eq:norm_equivalence_d} between $d$ and $\norm{}_{L^2(\Xi)}$ yields \eqref{num:PosteriorContraction_L2}.
    In the final step, we extend the posterior convergence to $\norm{}_0$. By \eqref{num:PosteriorContraction_L2} and Lemma \ref{lem:Prior} \eqref{lem:aux_rescaled_Prior_restrict}, we find that the posterior concentrates on $B(0,\norm{}_{H^{\beta_0}(\Xi)},L'')$ with probability approaching 1 for any large enough $L''$. Since $\beta_0> 3/2$, the posterior further concentrates by the Sobolev embedding on $\mathcal{F}_{1,L}$ for $L>0$ large enough. This proves \eqref{eq:Concentration_on_Lipschitzball}. Note that for $g,h\colon\mathbb{R}\to\mathbb{R}$, $h_\lambda(g,h)^2 - \norm{g-h}_0^2 = \mathcal{G}_{\lambda}( (g-h)^2)$. Corollary \ref{cor:Convergence_polynomials} (below) with $k=2$, $f=f_0$ and $h=g-\tilde{g}$ then yields
    \begin{equation*}
        \sup_{f\in \mathcal{F}_{1,L}}\abs*{d_\lambda(f,f_0) - \norm{f-f_0}_0}\le C_{L,\Lip{f_0}}\lambda^{-1/2}\le \epsilon_\lambda
    \end{equation*}
    whenever $\lambda$ is large enough. Consequently,
    \begin{equation*}
        \Pi_\lambda\left(f\colon \norm{f-f_0}_0\ge K_\lambda \epsilon_\lambda\,\vert\, X^\lambda \right) \le \Pi_\lambda\left(f\colon d_\lambda(f,f_0)\ge (\tilde{K}_\lambda+1) \epsilon_\lambda\,\vert\, X^\lambda \right) + \smallo_{\prob[f_0]{}}(1)
    \end{equation*}
    converges to zero in $\prob[f_0]{}$-probability as $\lambda\to\infty$. 
\end{proof}

\begin{lemma}\label{lem:Prior}
	For $0\leq\beta_0<\beta$ and $f_0\in H^\beta(\Xi)$ consider the Gaussian wavelet prior $\Pi_\lambda$ from \eqref{eq:Prior} on $V_M$  with cut-off $M=M_\lambda\in\mathbb{N}$ such that $2^M\sim \lambda^{1/(2\beta+1)}$. Let $\epsilon_\lambda = \lambda^{-\beta/(2\beta+1)}\log(\lambda)$ and let $0<C<\infty$ such that $2^M\le C_1 \lambda^{1/(2\beta+1)}$. Then for any large enough $L,L',L''$ and any $L'''>0$ there exist constants $0<C_1,C_2,C_3<\infty$ depending on $L,L',L''$ and $\Xi$ such that the following properties hold:
	\begin{enumerate}[(a)]
		\item\label{lem:aux_rescaled_Prior_sieve} Sieve set: $\Pi_\lambda(f\in V_M:\norm{f}_{H^{\beta_0}(\Xi)}>L\sqrt{\lambda}\varepsilon_{\lambda})\le \exp(-C_1\lambda\varepsilon^2_{\lambda})$.
		\item\label{lem:aux_rescaled_Prior_entropy} Entropy bound: $\log\left(N\left(V_M\cap B(0,\norm{}_{H^{\beta_0}(\Xi)},L\sqrt{\lambda}\varepsilon_\lambda), \norm{}_{L^2(\Xi)}, \epsilon_\lambda\right)\right)\le C_2\lambda \epsilon_\lambda^2$.
        \item\label{lem:aux_rescaled_Prior_smallball} Small ball probability: $\Pi_\lambda (f\in V_M:\norm{f-f_0}_{L^2(\Xi)}\leq L'\epsilon_\lambda)\ge \exp(-C_3\lambda\varepsilon^2_{\lambda})$.	
        \item\label{lem:aux_rescaled_Prior_restrict} Posterior regularity:
        \begin{equation*}
            \Pi_\lambda(f\in V_M: \norm{f}_{H^{\beta_0}(\Xi)}> L'',\norm{f-f_0}_{L^2(\Xi)}\leq L'''\varepsilon_{\lambda}\,\vert\, X^\lambda)\xrightarrow{\prob[f_0]{}} 0,\qquad\lambda\to\infty.
        \end{equation*} 
	\end{enumerate}
\end{lemma}
\begin{proof}
	See Section \ref{sec:Appendix_Stats}.	
\end{proof}

\begin{lemma}\label{lem:supremum_Control_Hellinger}
	Grant Assumption \ref{assump:Parameterspace}, let $f_0\in\Theta$ and let $M\in\mathbb{N}$ be the frequency cut-off. Then
	\begin{equation}
		\prob[f_0][][\bigg]{\sup_{g,h\in V_M\colon g\neq h} \abs[\bigg]{\frac{h_\lambda(g,h)^2- d_\lambda(g,h)^2}{d_\lambda(g,h)^2}}\ge \hat{C}\lambda^{-1/2} x}\le \exp\left(\hat{c}\operatorname{dim}(V_M)-x^2/2\right),\label{eq:Bayes_Probability}
	\end{equation}
	for any $x\ge 0$, where $0<\hat{c},\hat{C}<\infty$ are absolute constants depending only on $\norm{f_0}_{C^3(\mathbb{R})}$, $T$ and $\Xi$.
\end{lemma}

\begin{proof}
	See Section \ref{sec:Appendix_Stats}.
\end{proof}

\begin{proof}[Proof of Proposition \ref{prop:Bayes_lowerBounds}]
	Since $\norm{}_{L^2(\Xi)}$ and $\norm{}_0$ are equivalent norms by \eqref{eq:norm_equivalence}, it suffices to prove the lower bound in terms of $\norm{}_{L^2(\Xi)}$. The proof follows from combining the methodology developed in Section 2.6 of \citet{Tsybakov2009} with the upper density bound of Part \eqref{num:cordensitybounds_a} of Corollary \ref{cor:DensityBounds}. For simplicity, we restrict to the case $\Xi=[0,1]$ and proceed as in Section 2.6.1 of \citet{Tsybakov2009} with the sample size $n$ replaced by the size of the domain $\lambda$. For some constant $c_0>0$ to be chosen later we can construct $N\gtrsim 2^{c_0\ceil{\lambda^{1/(2\beta+1)}}/8}$ many $\lambda$-dependent hypotheses $f_{j}$, $j=1,\dots, N$, which satisfy the following properties:
	\begin{enumerate}[(a)]
		\item $f_{j}\in \mathcal{F}_{\beta,L}\cap\Theta$ for all $j=1,\dots, N$;
		\item $\norm{f_{j}-f_{k}}_{L^2(\Xi)}\gtrsim \lambda^{-2\beta/(2\beta+1)}$ uniformly in $j\neq k$, for $\lambda$ sufficiently large;
		\item $\norm{f_{j}}_{L^2(\mathbb{R})}^2=\norm{f_{j}}_{L^2(\Xi)}^2\le c\lambda^{-2\beta/(2\beta +1)}$, $j=1,\dots, M$ for some $c>0$ not depending on $j$, or $N$.
	\end{enumerate}
	In view of Theorem 2.5 of \citet{Tsybakov2009}, it suffices to prove with respect to the Kullback-Leibler distance
	\begin{equation}
		d_{\operatorname{KL}}\left(\prob[f_{j}]{}\,\Vert\, \prob[f_{0}]{}\right)\le \lambda\tilde{C} \norm{f_{j}}_{L^2(\Xi)}^2,\quad j=1,\dots, N,\quad  \lambda\ge 1,\label{eq:Bayes_Lowerbound_Claim}
	\end{equation}
	for some constant $0<\tilde{C}<\infty$. Suppose \eqref{eq:Bayes_Lowerbound_Claim} is true. Then $d_{\operatorname{KL}}(\prob[f_{j}]{}\,\Vert\,\prob[f_{0}]{})\le c\tilde{C}\lambda^{1/(2\beta +1)}\lesssim \log(N)$
	for an appropriate choice of the constant $c$, similarly to Section 2.6.1 of \citet{Tsybakov2009}. This verifies condition (2.46) of Theorem 2.5 of \citet{Tsybakov2009} and the claimed lower bound follows. To show \eqref{eq:Bayes_Lowerbound_Claim} we apply the density bound of Part \eqref{num:cordensitybounds_a} of Corollary \ref{cor:DensityBounds} to bound
    \begin{equation*}
        \int_{0}^T \int_{\Lambda}\EV*[f_{j}]{f_{j}(X_t(y))^2}\,\D y\D t \lesssim_{T,L} \lambda \norm{f_j}_{L^2(\mathbb{R})}^2.
    \end{equation*}
    Consequently, we find for $j=1,\dots, N$
	\begin{align*}
		d_{\operatorname{KL}}\left(\prob[f_{j}]{}\,\Vert\,\prob[f_0]{}\right) & = \EV[f_{j}][][\bigg]{\operatorname{log}\bigg(\frac{\D \prob[f_{j}]{}}{\D \prob[f_0]{}}\bigg)}                                                       
		= \frac{1}{2}\int_{0}^T \int_{\Lambda}\EV*[f_{j}]{f_{j}(X_t(y))^2}\,\D y\D t  \lesssim_{T,L} \lambda\norm{f_{j}}_{L^2(\mathbb{R})}^2.
	\end{align*}\qedhere
\end{proof}

The proof of the Bernstein--von Mises Theorem \ref{thm:BvM} follows the general proof strategy of \citet{castillo2013} to show weak convergence in $H^{-\rho}(\Xi)$ by pointwise convergence of the Laplace transforms for the finite-dimensional distributions. Before proving the theorem, we state a lemma about expansion of the Laplace transform followed by a number of maximal inequalities that ensure convergence of the finite-dimensional distributions is all we need to show. Fix any $\gamma\in H^\rho(\Xi)$. For $K,\bar{K}>0$, $2^{M}\sim \lambda^{1/(2\beta_0+1)}$ and $\epsilon_{\lambda}=\lambda^{-\beta_0/(2\beta_0+1)}\log(\lambda)$ define the set
    \begin{align}
    	D_\lambda
    	 & = \Set{f\in V_M\colon  \norm{f-f_0}_{L^2(\Xi)}\leq K\epsilon_{\lambda}, \norm{f}_{H^\beta(\Xi)}\leq \bar{K}}\label{eq:D_t}
    \end{align}
    and introduce the restricted posterior 
    \begin{align}
        \Pi_\lambda^{D_\lambda}(A\,|\,X^\lambda)=\Pi_\lambda(A\cap D_\lambda\,|\,X^\lambda)/\Pi_\lambda(D_\lambda\,|\,X^\lambda),\quad A\subset C(\mathbb{R})\,\, \text{any Borel set}. \label{eq:restr_posterior}
    \end{align}
    For large enough $\bar{K}$ we find from Lemma \ref{lem:Prior} \eqref{lem:aux_rescaled_Prior_restrict} that $D_\lambda$ is asymptotically negligible under the posterior:
    \begin{align}
        \Pi_\lambda(D^c_\lambda\,|\,X^\lambda)\xrightarrow{\prob[f_0]{}} 0,\quad \lambda\to\infty.\label{eq:D_negligible}
    \end{align}

\begin{lemma}[Asymptotic expansion of the Laplace Transform]\label{lem:ExpansionLaplaceTransform}
	Grant Assumption \ref{assump:Parameterspace} and let $f_0\in\Theta$. Suppose that $2^{M}\sim \lambda^{1/(2\beta_0+1)}$ and $\epsilon_{\lambda}=\lambda^{-\beta_0/(2\beta_0+1)}\log(\lambda)$. For $\gamma\in H^{\rho}(\Xi)$, $\rho>3/2$, write $\gamma_0=\gamma/p_{f_0}$ and define the process
	\begin{align*}
		G_\lambda(\gamma) = \sqrt{\lambda}\iprod{f-f_0}{\gamma}_{L^2(\Xi)} + \sqrt{\lambda}\mathcal{G}_{\lambda}((f-f_0)P_M\gamma_0) - W_\lambda(P_M\gamma_0),\quad f\sim \Pi_\lambda(\MTemptyplaceholder\vert\, X^\lambda).
	\end{align*}
	Then the following holds for $\gamma,\bar\gamma\in H^{\rho}(\Xi)$, $u\in \mathbb{R}$:
	\begin{enumerate}[(a)]
		\item\label{num:Laplace_Metrics} There exist constants $0<C,\bar C<\infty$ depending only on $T,\Xi,K$ such that
			\begin{align*}
				\EV[\lambda][D_\lambda]{e^{u (G_\lambda(\gamma)-G_\lambda(\bar\gamma))}\given X^\lambda} \leq \frac{C}{\Pi_\lambda(D_\lambda\,|\,X^\lambda)}e^{u^2d_1(\gamma,\bar \gamma)^2/2+\abs{u}d_2(\gamma,\bar \gamma)},
			\end{align*}
			where $d_1$, $d_2$ are metrics on $H^{\rho}(\Xi)$ given by
			\begin{align}
            \begin{split}
				d_1(\gamma,\bar\gamma) &= \bar{C}\norm{P_M(\gamma_0-\bar\gamma_0)}_{\infty},\\
                d_2(\gamma,\bar\gamma) &= \bar{C}\sqrt{\lambda} \epsilon_\lambda\norm{\gamma_0-\bar\gamma_0-P_M (\gamma_0-\bar\gamma_0)}_{L^2(\Xi)}.
            \end{split}\label{eq:Metrics_d1d2}
			\end{align}

		\item\label{num:Laplace_Convergence} $\EV[\lambda][D_\lambda]{e^{u G_{\lambda}(\gamma)}\given X^\lambda}\overset{\prob[f_0]{}}{\rightarrow}e^{u^2\norm{\gamma_0}^2_{0}/2}$ as $\lambda\to\infty$.
\end{enumerate} 
\end{lemma}
\begin{proof}
	See Section \ref{sec:Appendix_Stats}.
\end{proof}

\begin{lemma}[Maximal inequalities]\label{lem:SubGaussianBounds}
	Suppose the assumptions of Lemma \ref{lem:ExpansionLaplaceTransform} hold and consider the set
	\begin{equation*}
		\mathcal{F} = \Set{g=(f-f_0)P_M\gamma_0\given f\in D_\lambda, \gamma\in H^{\rho}(\Xi), \norm{\gamma}_{H^{\rho}(\Xi)}\leq 1}.
	\end{equation*}
	Then we have for any $J\geq 1$:
	\begin{enumerate}[(a)]
		\item $\EV[f_0]{\sup_{g\in\mathcal{F}}\abs{\mathcal{G}_{\lambda}(g)}}= o(\lambda^{-1/2})$.\label{num:Maximal1}
		\item $\EV[f_0]{\sup_{\gamma\in B(0,\norm{}_{H^\rho(\Xi)},1)}\abs{\mathbb{W}(\gamma-P_J\gamma)}}\lesssim 2^{-J/2}$. \label{num:Maximal2}
		\item $\EV[\lambda][D_\lambda]{\sup_{\gamma\in B(0,\norm{}_{H^\rho(\Xi)},1)}\abs{G_\lambda(\gamma-P_J\gamma)}\given X^\lambda}\lesssim 2^{-J(\rho-1)}+\lambda^{-(\rho-3/2)/(2\beta_0+1)}$. \label{num:Maximal3}
		\item\label{num:Maximal4} $\EV[f_0]{\sup_{\gamma\in B(0,\norm{}_{H^\rho(\Xi)},1)}\abs{W_\lambda(P_M(\gamma_0-P_J\gamma_0))}}\lesssim 2^{-J/2}$.
	\end{enumerate}
\end{lemma}
\begin{proof}
	See Section \ref{sec:Appendix_Stats}.
\end{proof}

\begin{proof}[Proof of Theorem \ref{thm:BvM}]~
\begin{enumerate}[(a)]
\item By a standard inequality for the total variation distance (e.g., Page 142 of \citet{Vaart1998}) and \eqref{eq:D_negligible} we have
	\begin{align}
        \operatorname{TV}\left(\Pi_\lambda(\MTemptyplaceholder|\,X^\lambda),\Pi_\lambda^{D_\lambda}(\MTemptyplaceholder|\,X^\lambda)\right)\leq 2\Pi_\lambda(D^c_\lambda\,|\,X^\lambda)\xrightarrow{\prob[f_0]{}} 0,\quad \lambda\rightarrow \infty.\label{eq:TV_convergence}
	\end{align}
    Consider for $\gamma\in H^\rho(\Xi)$ and $\gamma_0=\gamma/p_{f_0}$ the processes 
	\begin{align*}
		G_\lambda(\gamma) 
            &= \sqrt{\lambda}\iprod{f-f_0}{\gamma}_{L^2(\Xi)} + \sqrt{\lambda}\mathcal{G}_{\lambda}((f-f_0)P_M\gamma_0) - W_\lambda(P_M\gamma_0),\\
        \bar{G}_\lambda(\gamma) 
			&= \sqrt{\lambda}\iprod{f-f_0}{\gamma}_{L^2(\Xi)}  - W_\lambda(P_M\gamma_0).
	\end{align*}
	By Lemma \ref{lem:ExpansionLaplaceTransform}   \eqref{num:Laplace_Convergence} the Laplace transform of $G_\lambda(\gamma)$ with $f\sim \Pi^{D_\lambda}_\lambda(\MTemptyplaceholder|\,X^\lambda)$ converges pointwise conditionally on $X^\lambda$ in $\prob[f_0]{}$-probability to the Laplace transform of a centred normal distribution with variance $\norm{\gamma_0}^2_0=\norm{\gamma/p_{f_0}^{1/2}}^2_{L^2(\Xi)}$. This convergence transfers to the Laplace transform of $\bar G_\lambda(\gamma)$ with $f\sim \Pi_\lambda(\MTemptyplaceholder|\,X^\lambda)$ since $G_\lambda(\gamma)=\bar G_\lambda(\gamma)+\smallo_{\prob[f_0]{}}(1)$ by Lemma \ref{lem:SubGaussianBounds} \eqref{num:Maximal1}. Recall from Proposition 29 of \citet{nickl2020b} that convergence of Laplace transforms in probability is equivalent to weak convergence of the respective laws in probability. As $\bar G_\lambda(\gamma)$ is conditionally on the data Gaussian, this means that also the total variation distance between $\mathcal{L}(\bar G_\lambda(\gamma))$ and $\mathcal{L}(\mathbb{W}(\gamma))=N(0,\norm{\gamma_0}^2_0)$ converges to zero in $\prob[f_0]{}$-probability. 
    
    Finally, we argue that in this convergence $\bar{G}_{\lambda}(\gamma)$ can be replaced by $\sqrt{\lambda}\iprod{f-\hat f}{\gamma}_{L^2(\Xi)}$. Again, conditional Gaussianity of $\bar{G}_{\lambda}(\gamma)$ and the arguments in the proof of Theorem 2.7 of \citet{monard2019} show that the moments of $\bar G_\lambda(\gamma)$ converge conditionally on the data to the moments of $\mathcal{L}(\mathbb{W}(\gamma))$ in $\prob[f_0]{}$-probability. So, in particular,
	\begin{align}
		\sqrt{\lambda}\iprod{\hat f-f_0}{\gamma}_{L^2(\Xi)} - W_\lambda(P_M\gamma_0)=\EV[\lambda]{\bar G_\lambda(\gamma)\given X^\lambda}=\smallo_{\prob[f_0]{}}(1),\label{eq:moment_convergence}
	\end{align} 
	and thus $\sqrt{\lambda}\iprod{f-\hat f}{\gamma}_{L^2(\Xi)}=\bar G_\lambda(\gamma)+\smallo_{\prob[f_0]{}}(1)$, implying the claimed convergence in total variation distance in $\prob[f_0]{}$-probability.
	\item For an index $J\in\mathbb{N}$ to be chosen later define on $H^{\rho}(\Xi)$ the projected processes $\bar G_{\lambda,J}$ and $\mathbb{W}_J$ through $\bar G_{\lambda,J}(\gamma)=\bar G_\lambda(P_J\gamma)$ and $\mathbb{W}_J(\gamma)=\mathbb{W}(P_J\gamma)$. By the triangle inequality for the metric $\beta_{-\rho}=\beta_{H^{-\rho}(\Xi)}$ we find 
	\begin{align*}
		\beta_{-\rho}\left(\prob[][\bar G_\lambda]{}, \prob[][\mathbb{W}]{}\right) 
			& \leq \beta_{-\rho}\left(\prob[][\bar G_{\lambda,J}]{}, \prob[][\mathbb{W}_J]{}\right)
			+ \beta_{-\rho}\left(\prob[][\bar G_\lambda]{}, \prob[][\bar G_{\lambda,J}]{}\right) + \beta_{-\rho}\left(\prob[][\mathbb{W}]{}, \prob[][\mathbb{W}_J]{}\right).
	\end{align*}
	From \eqref{num:BvM_TVConvergence} we know for any fixed $\gamma\in H^\rho(\Xi)$ that $\prob[][\bar G_{\lambda,J}(\gamma)]{}$ converges weakly to $\prob[][\mathbb{W}_J(\gamma)]{}$. By linearity, the same holds true for any finite linear combination  $\gamma=\sum_{i=1}^m a_i\gamma_i$ with $\gamma_i\in V_J$, $i=1,\dots, m$, $a\in\mathbb{R}^m$, $m\in\mathbb{N}$. By the Cramér-Wold device we thus have weak convergence of the finite-dimensional distributions of $\bar G_J$. In particular, for $J$ fixed,
	\begin{align}
		\beta_{-\rho}\left(\prob[][\bar G_{\lambda,J}]{}, \prob[][\mathbb{W}_J]{}\right)\xrightarrow{\prob[f_0]{}} 0,\quad \lambda\to \infty.\label{eq:BVM_term1}
	\end{align}
	Let $\tilde{G}_\lambda$ and $\tilde{G}_{\lambda,J}$ be defined as $\bar{G}_\lambda$ and $\bar{G}_{\lambda,J}$, respectively, but with $f\sim \Pi^{D_\lambda}_\lambda(\MTemptyplaceholder|\,X^\lambda)$ instead of $f\sim \Pi_\lambda(\MTemptyplaceholder|\,X^\lambda)$. The standard inequality \eqref{eq:TV_convergence} yields $\operatorname{TV}(\prob[][\tilde G_\lambda]{}, \prob[][\tilde G_{\lambda,J}]{})=\smallo_{\prob[f_0]{}}(1)$ and $\operatorname{TV}(\prob[][\tilde G_{\lambda,J}]{}, \prob[][\bar G_{\lambda,J}]{})=\smallo_{\prob[f_0]{}}(1)$ uniformly in $J\in\mathbb{N}$. Recalling the variational characterisation of the total variation distance \eqref{eq:TV} we get
	\begin{align*}
		&\beta_{-\rho}\left(\prob[][\bar G_\lambda]{}, \prob[][\bar G_{\lambda,J}]{}\right)
			\leq\beta_{-\rho}\left(\prob[][\bar G_\lambda]{}, \prob[][\tilde G_\lambda]{}\right) + \operatorname{TV}\left(\prob[][\tilde G_\lambda]{}, \prob[][\tilde G_{\lambda,J}]{}\right) + \operatorname{TV}\left(\prob[][\tilde G_{\lambda,J}]{}, \prob[][\bar G_{\lambda,J}]{}\right)\\
			&\quad = \sup_{F\colon\Lip{F}\leq 1} \abs[\bigg]{\int_{V_M}F(\bar G_\lambda)\D\Pi^{D_\lambda}_{\lambda}(\MTemptyplaceholder|\,X^\lambda)-\int_{V_M}F(\bar G_{\lambda,J})\D\Pi^{D_\lambda}_{\lambda}(\MTemptyplaceholder|\,X^\lambda)}
            + \smallo_{\prob[f_0]{}}(1).
	\end{align*}
	Using the Lipschitz property of the functions $F$ and then applying Lemma \ref{lem:SubGaussianBounds} \eqref{num:Maximal1} and \eqref{num:Maximal3}, the first term is up to error terms uniform in $J$ upper bounded by 
	\begin{align*}
		\EV*[\lambda][D_\lambda]{\norm{\bar G_\lambda-\bar G_{\lambda,J}}_{H^{-\rho}(\Xi)}\given X^\lambda} &=\EV[\lambda][D_\lambda][\Big]{\sup_{\norm{\gamma}_{H^\rho(\Xi)}\leq 1}\abs*{\bar G_\lambda(\gamma-P_J\gamma)}\given X^\lambda}\\
        &= \EV*[\lambda][D_\lambda][\Big]{\sup_{\norm{\gamma}_{H^\rho(\Xi)}\leq 1}\abs*{G_\lambda(\gamma-P_J\gamma)}\given X^\lambda} + \smallo_{\prob[f_0]{}}(1)\\
        &= \mathcal{O}_{\prob[f_0]{}}(2^{-J/2}) + \smallo_{\prob[f_0]{}}(1).
	\end{align*}
	From this deduce
	\begin{align}
		\beta_{-\rho}\left(\prob[][\bar G_\lambda]{}, \prob[][\bar G_{\lambda,J}]{}\right)= \mathcal{O}_{\prob[f_0]{}}(2^{-J/2})+\smallo_{\prob[f_0]{}}(1).\label{eq:BVM_term2}
	\end{align}
	Finally, using Lemma \ref{lem:SubGaussianBounds} \eqref{num:Maximal2} we find that
	\begin{equation}
		\beta_{-\rho}\left(\prob[][\mathbb{W}]{}, \prob[][\mathbb{W}_J]{}\right)\leq \EV[f_0]{\norm{\mathbb{W}-\mathbb{W}_J}_{H^{-\rho}(\Xi)}}\lesssim 2^{-J/2}.\label{eq:Projection_Betadistance_GaussianPart}
	\end{equation}
	Combining this with \eqref{eq:BVM_term1}, \eqref{eq:BVM_term2} and because $J\in\mathbb{N}$ was arbitrary, we obtain in all $\beta_{-\rho}(\prob[][\bar{G}_\lambda]{}, \prob[][\mathbb{W}]{})= \smallo_{\prob[f_0]{}}(1)$ as $\lambda\to\infty$. Consequently, the (Gaussian) law of $\bar{G}_\lambda$ converges weakly in $H^{-\rho}(\Xi)$ to the law of $\mathbb{W}$. As in \eqref{num:BvM_TVConvergence} we can invoke the arguments of Theorem 2.7 of \citet{monard2019} to conclude that $\EV[\lambda]{\bar{G}_\lambda\given X^\lambda}\rightarrow 0$ in $\prob[f_0]{}$-probability relative to $H^{-\rho}(\Xi)$) and $\sqrt{\lambda}(f-\hat{f}) =\bar{G}_\lambda+\smallo_{\prob[f_0]{}}(1)$ in $H^{-\rho}(\Xi)$.
\end{enumerate}\qedhere
\end{proof}

\begin{proof}[Proof of Corollary \ref{cor:BvM_Truth}]
	We only need to prove Part \eqref{num:cor_BvM_uniform} as Part \eqref{num:cor_BvM_pointwise} is a special case. The proof of Theorem \ref{thm:BvM} \eqref{num:BvM_BetaConvergence} already argues that 
	\begin{align}
		\sqrt{\lambda}(\hat f-f_0)=W_\lambda(P_M(\MTemptyplaceholder)_0)+\EV[\lambda]{\bar G_\lambda\given X^\lambda}=W_\lambda(P_M(\MTemptyplaceholder)_0)+\smallo_{\prob[f_0]{}}(1)
	\end{align} 
	in $H^{-\rho}(\Xi)$. It therefore suffices to show that $W_\lambda(P_M(\MTemptyplaceholder)_0)$ converges to $\mathbb{W}$ in distribution in $H^{-\rho}(\Xi)$. As in \eqref{eq:Projection_Betadistance_GaussianPart} we find from Lemma \ref{lem:SubGaussianBounds} \eqref{num:Maximal4} and linearity of $W_\lambda$ that
	\begin{equation*}
		\beta_{-\rho}\left(\prob[][W_{\lambda}(P_M(\MTemptyplaceholder)_0)]{}, \prob[][W_{\lambda}(P_J(\MTemptyplaceholder)_0)]{}\right)\le \EV[f_0]{\norm{W_{\lambda}(P_M(\MTemptyplaceholder)_0)-W_{\lambda}(P_J(\MTemptyplaceholder)_0)}_{H^{-\rho}(\Xi)}}=\mathcal{O}_{\prob[f_0]{}}(2^{-J/2}).
	\end{equation*}
    By the triangle inequality for the $\beta_{-\rho}$ metric as in the proof of Theorem \ref{thm:BvM} \eqref{num:BvM_BetaConvergence} we find by Lemma \ref{lem:SubGaussianBounds} \eqref{num:Maximal2} and \eqref{num:Maximal4} for $1\leq J\leq M$ the bound
	\begin{equation*}
		\beta_{-\rho}\left(\prob[][W_{\lambda}(P_M(\MTemptyplaceholder)_0)]{}, \prob[][\mathbb{W}]{}\right)
			 \leq \beta_{-\rho}\left(\prob[][W_{\lambda}(P_J(\MTemptyplaceholder)_0)]{}, \prob[][\mathbb{W}_J]{}\right) + \mathcal{O}_{\prob[f_0]{}}(2^{-J/2}).
	\end{equation*}
	From Lemma \ref{lem:LAN} we know that $W_{\lambda}(P_J\gamma_0)$ converges in distribution to the random variable $\mathbb{W}_J(\gamma)\sim N(0,\norm{P_J\gamma_0}^2_0)$ for any fixed $\gamma\in V_m$ and $J\in\mathbb{N}$. By linearity, the same holds true for any finite linear combination of functions in $V_J$, and thus by the Cramér-Wold device $W_{\lambda}(P_J\gamma)$ converges in distribution in $H^{-\rho}(\Xi)$ to $\mathbb{W}$. From this obtain the claim.
\end{proof}

\section*{Acknowledgements} 
	We thank Matteo Giordano and Markus Reiß for very helpful comments and discussions. S.\ G.\ gratefully acknowledges financial support from DFG CRC/TRR 388 \enquote{Rough Analysis, Stochastic Dynamics and Related Fields}, Project B07, and DFG IRTG2544 \enquote{Stochastic Analysis in Interaction}, Project-ID 410208580.

\appendix

\section{Remaining proofs for the statistical results}\label{sec:Appendix_Stats}

This section contains the remaining proofs for Lemmas \ref{lem:Prior}, \ref{lem:supremum_Control_Hellinger}, \ref{lem:ExpansionLaplaceTransform} and \ref{lem:SubGaussianBounds}.

\begin{proof}[Proof of Lemma \ref{lem:Prior}]~
\begin{enumerate}[(a)]
\item We have $\operatorname{dim}(V_M)\lesssim 2^M\lesssim \lambda^{1/(2\beta+1)}\leq\lambda\varepsilon^2_{\lambda}$. So, for $L$ large enough $\operatorname{dim}(V_M)\leq (L^2/2)\lambda\varepsilon^2_{\lambda}$. For this $L$, we find with the random variables $Z_{\mu}\overset{i.i.d.}{\sim} N(0,1)$, $|\mu|\leq M$, from \eqref{eq:Prior} that
	\begin{equation*}
		\Pi_{\lambda}(f\in V_M:\norm{f}^2_{H^{\beta_0}(\Xi)}> L^2\lambda\varepsilon^2_{\lambda})
			=\prob[][][]{\sum_{|\mu|\leq M}Z^2_{\mu}>L^2\lambda\varepsilon^2_{\lambda}}\leq\prob[][][]{\sum_{|\mu|\leq M}(Z^2_{\mu}-1)>(L^2/2)\lambda\varepsilon^2_{\lambda}},
	\end{equation*}
	where the last inequality holds since there are at most $\operatorname{dim}(V_M)$ many $Z_{\mu}$ in the sum. The random variables $Z_{\mu}^2-1$ are independent centred chi-squared random variables. A standard subexponential inequality (e.g., Equation 3.29 of \citet{gineMathematicalFoundationsInfiniteDimensional2015}) implies for the last probability the upper bound 
	\begin{align*}
		\exp\bigg(-\frac{(L^2/2)^2(\lambda\varepsilon^2_{\lambda})^2}{4\operatorname{dim}(V_M)+(L^2/2)\lambda\varepsilon^2_{\lambda}}\bigg) \leq \exp(-(L^2/10)\lambda\varepsilon^2_{\lambda}).
	\end{align*}
	\item Note $\norm{f}_{H^{\beta_0}(\Xi)}\leq 2^{M\beta_0}\norm{f}_{L^2(\Xi)}$ for $f\in V_M$. Hence, 
	\begin{align*}
		&\log\left(N\left(V_M\cap B(0,\norm{}_{H^{\beta_0}(\Xi)},L\sqrt{\lambda}\varepsilon_\lambda), \norm{}_{L^2(\Xi)}, \epsilon_\lambda\right)\right)\\
		&\hspace{10em}  \leq \log\left(N\left(B(0,\norm{}_{\mathbb{R}^{\operatorname{dim}(V_M)}},L2^{M\beta_0}\sqrt{\lambda}\varepsilon_\lambda), \norm{}_{\mathbb{R}^{\operatorname{dim}(V_M)}}, \epsilon_\lambda\right)\right)\\
		&\quad \leq \operatorname{dim}(V_M)\log(3L2^{M\beta_0}\sqrt{\lambda})\lesssim \lambda^{1/(2\beta+1)}\log(\lambda)\lesssim \lambda\epsilon_\lambda^2,
	\end{align*} 
	using a metric entropy estimate for finite-dimensional Euclidean balls (Proposition 4.3.34 of \citet{gineMathematicalFoundationsInfiniteDimensional2015}).
	\item For $f_0\in H^{\beta}(\Xi)$ we have $\norm{f_0-P_Mf_0}_{L^2(\Xi)}\lesssim 2^{-M\beta}\norm{f_0}_{H^\beta(\Xi)}\lesssim \varepsilon_{\lambda}$, so for $L'$ large enough with $C=L'/2$
	\begin{align*}
		&\Pi_\lambda (f\in V_M:\norm{f-f_0}_{L^2(\Xi)}\leq L'\epsilon_\lambda) \geq \Pi_\lambda (f\in V_M:\norm{f-P_Mf_0}_{L^2(\Xi)}\leq C\epsilon_\lambda)\\
		&\quad \geq \Pi_\lambda (f\in V_M:\norm{f-P_M f_0}_{H^{\beta_0}(\Xi)}\leq C\epsilon_\lambda)\gtrsim \Pi_{\lambda}(f\in V_M:\norm{f}_{H^{\beta_0}(\Xi)}\leq C\epsilon_\lambda),
	\end{align*}
	concluding with Corollary 2.6.18 of \citep{gineMathematicalFoundationsInfiniteDimensional2015}, noting $P_Mf_0\in V_M$ and $\norm{P_Mf_0}_{\mathbb{H}^{\lambda}}\leq \norm{f_0}_{H^{\beta}(\Xi)}$, cf. \eqref{eq:RKHS}. Using the notation from (a) and letting $Z\overset{d}{\sim} N(0,1)$, the last expression is equal to
	\begin{align*}
		&\prob[][][]{\sum_{|\mu|\leq M}Z^2_{\mu}\leq C^2\varepsilon^2_{\lambda}}\geq \prob[][][]{\operatorname{dim}(V_M)\max_{|\mu|\leq M}Z^2_{\mu}\leq C^2\varepsilon^2_{\lambda}}\\
		&\quad = \prob[][][]{\operatorname{dim}(V_M)Z^2\leq C^2\varepsilon^2_{\lambda}}^{\operatorname{dim}(V_M)}\geq \prob[][][]{|Z|\leq C'\lambda^{-1/2}}^{\operatorname{dim}(V_M)}
	\end{align*}
	for some $C'>0$, using again that $\operatorname{dim}(V_M)\lesssim \lambda\varepsilon^2_{\lambda}$. Since the standard normal density is uniformly bounded away from zero near the origin, the last probability is up to a constant lower bounded by $\lambda^{-\operatorname{dim}(V_M)/2}\gtrsim \exp(-c\lambda^{1/(2\beta+1)}\log(\lambda))=\exp(-c\lambda\varepsilon^2_{\lambda})$ for some $c>0$. Adjusting the constant for multiplicative factors yields the desired small ball lower bound.
	\item Suppose $f\in V_M$ with $\norm{f-f_0}_{L^2(\Xi)}\leq L'''\varepsilon_{\lambda}=L'''\lambda^{-\beta/(2\beta+1)}\log(\lambda)$. Then, by the triangle inequality, we have
\begin{align*}
		\norm{f}_{H^{\beta_0}(\Xi)}
			&\leq \norm{f-P_Mf_0}_{H^{\beta_0}(\Xi)} + \norm{P_Mf_0}_{H^{\beta_0}(\Xi)}\\
            &\le 2^{M\beta_0}\norm{f-P_Mf_0}_{L^2(\Xi)}+ \norm{f_0}_{H^{\beta}(\Xi)}\\
            &\leq 2^{M\beta_0}\norm{f-f_0}_{L^2(\Xi)} + 2^{M\beta_0}\norm{f_0-P_Mf_0}_{L^2(\Xi)} + \norm{f_0}_{H^{\beta}(\Xi)}\\
            &\leq 2L''' \lambda^{(\beta_0-\beta)/(2\beta+1)}\log(\lambda) + \norm{f_0}_{H^{\beta}(\Xi)},
	\end{align*}
	which is bounded by a constant uniformly in $\lambda$ for $\lambda\to\infty$ as long as $\beta>\beta_0$. This proves the claim taking $L''>\norm{f_0}_{H^\beta(\Xi)}$ large enough, depending on $L'''$.
    \end{enumerate}
\end{proof}

\begin{proof}[Proof of Lemma \ref{lem:supremum_Control_Hellinger}]
	Given the sub-Gaussian concentration of the spatio-temporal averages of Theorem \ref{thm:SubgaussianConcentration} \eqref{num:SubgaussianThm2}, the proof of Lemma \ref{lem:supremum_Control_Hellinger} can be carried out along the lines of the proof of Lemma 4 of \citet{nicklNonparametricStatisticalInference2020}. We include the proof for completeness. To this end, define $\tilde{M}\coloneq\operatorname{dim}(V_M) \lesssim 2^{M}$. For $g,h\in V_M$ we write $g = \mathbf{g}\transpose \mathbf{\Psi}$ and $h=\mathbf{h}\transpose\mathbf{\Psi}$, where $\mathbf{g},\mathbf{h}\in \mathbb{R}^{\tilde{M}}$ and $\mathbf{\Psi}=(\psi_{\mu}\,\colon\,\abs{\mu}\le M)$.\\
	\textbf{Step 1 (Rewriting the probability in \eqref{eq:Bayes_Probability}).}  Introduce the empirical Gram matrix
    \begin{equation*}
        \hat{\mathbf{G}} = \frac{1}{\lambda}\int_0^T\int_\Lambda \mathbf{\Psi}(X_t(y))\mathbf{\Psi}(X_t(y))\transpose\,\D y\D t\in\mathbb{R}^{\tilde{M}\times\tilde{M}},
    \end{equation*}
    which depends on $\lambda$ and on $f_0$ through the law of $X$. Then
	\begin{align*}
		h_\lambda(g,h)^2 & = \frac{1}{\lambda}\int_0^T\int_\Lambda \bigg(\sum_{k=1}^{\tilde{M}} (\mathbf{g}_k-\mathbf{h}_k)\mathbf{\Psi}_k(X_t(y))\bigg)^2\,\D y\D t          \\
		& = \frac{1}{\lambda}\sum_{k,j=1}^{\tilde{M}} \mathbf{g}_k\mathbf{h}_j \int_0^T\int_\Lambda \mathbf{\Psi}_k(X_t(y))\mathbf{\Psi}_j(X_t(y))\,\D y\D t \\
		& =(\mathbf{g}-\mathbf{h})\transpose \hat{\mathbf{G}}(\mathbf{g}-\mathbf{h}).
	\end{align*}
	Similarly, with the Gram matrix $\mathbf{G} = \frac{1}{\lambda}\EV[f_0]{\int_0^T\int_\Lambda \mathbf{\Psi}(X_t(y))\mathbf{\Psi}(X_t(y))\transpose\,\D y\D t}$, we find
	\begin{equation*}
		d_\lambda(g,h)^2=(\mathbf{g}-\mathbf{h})\transpose \mathbf{G}(\mathbf{g}-\mathbf{h}).
	\end{equation*}
	Thus, the probability in \eqref{eq:Bayes_Probability} equals
	\begin{equation*}
		\prob[f_0][][\bigg]{\sup_{\mathbf{g},\mathbf{h}\in \mathbb{R}^{\tilde{M}}\colon (\mathbf{g}-\mathbf{h})\transpose\mathbf{G} (\mathbf{g}-\mathbf{h})\neq 0}\abs{\frac{(\mathbf{g}-\mathbf{h})\transpose(\hat{\mathbf{G}}-\mathbf{G})(\mathbf{g}-\mathbf{h})}{(\mathbf{g}-\mathbf{h})\transpose\mathbf{G} (\mathbf{g}-\mathbf{h})}}\ge \lambda^{-1/2}\hat{C}x}.
	\end{equation*}
	Define $\mathbf{u}\coloneq \mathbf{g}-\mathbf{h}\in \mathbb{R}^{\tilde{M}}$, the $\mathbf{G}$-unit ball $\Gamma\coloneq\Set{\mathbf{u}\in\mathbb{R}^{\tilde{M}}\given \norm{\mathbf{u}}_\mathbf{G}^2\coloneq \mathbf{u}\transpose \mathbf{G} \mathbf{u}\le 1}$ and $\mathbf{M} = \hat{\mathbf{G}}-\mathbf{G}$.  Then the probability in \eqref{eq:Bayes_Probability} equals
	\begin{equation}
		\prob*[f_0]{\sup_{\mathbf{u}\in \Gamma} \mathbf{u}\transpose\mathbf{M} \mathbf{u}\ge \lambda^{-1/2} \hat{C}x}.\label{eq:Bayes_Probability_2}
	\end{equation}
	\textbf{Step 2 (Bounding the probability \eqref{eq:Bayes_Probability_2} using a $\delta$-covering).} For $\delta>0$ let $(\mathbf{u}^l)_{l=1}^{N(\delta)}$ be a minimal $\delta$-covering of $\Gamma$ in $\norm{}_{\mathbf{G}}$-distance. Denote by $\mathbf{u}^l=\mathbf{u}^l(\mathbf{u})$ the closest point in the covering to $\mathbf{u}\in\Gamma$ such that $\norm{\mathbf{u}-\mathbf{u}^l}_{\mathbf{G}}\le \delta$. We obtain the bound
	\begin{equation}
		\abs*{(\mathbf{u}-\mathbf{u}^l)\transpose\mathbf{M} (\mathbf{u}-\mathbf{u}^l)}\le  \delta^2 \sup_{\mathbf{w}\in \Gamma}\abs{\mathbf{w}\transpose \mathbf{M} \mathbf{w}},\quad \mathbf{u}\in\Gamma.\label{eq:aux_Bayes_Prob_1}
	\end{equation}
	Moreover, using the norm equivalence \eqref{eq:norm_equivalence_d}, we find for any $\mathbf{v}\in\mathbb{R}^{\tilde{M}}$ the bound
	\begin{equation*}
		\norm{\mathbf{v}}_{\mathbb{R}^{\tilde{M}}}=\norm*{\mathbf{v}\transpose \mathbf{\Psi}}_{L^2(\Xi)}\lesssim_{T,K,\Xi}  \bigg(\int_0^T\frac{1}{\lambda}\int_\Lambda\EV[f_0]{(\mathbf{v}\transpose \mathbf{\Psi}(X_t(y)))^2}\,\D y\D t\bigg)^{1/2}=\norm{\mathbf{v}}_{\mathbf{G}}.
	\end{equation*}
	Denote by $\lambda_{\max}$ the maximal eigenvalue of $\mathbf{M}$, then we find
	\begin{equation*}
		\abs{(\mathbf{u}-\mathbf{u}^l)\transpose \mathbf{M} \mathbf{u}^l}  \le \norm{\mathbf{u}-\mathbf{u}^l}_{\mathbb{R}^{\tilde{M}}}\norm{\mathbf{M} \mathbf{u}^l}_{\mathbb{R}^{\tilde{M}}}\lesssim_{T,K,\Xi} \delta \lambda_{\max}\norm{\mathbf{u}^l}_{\mathbb{R}^{\tilde{M}}} \lesssim_{T,K,\Xi}  \delta \sup_{\norm{\mathbf{v}}_{\mathbb{R}^{\tilde{M}}}\le 1}\abs{\mathbf{v}\transpose\mathbf{M} \mathbf{v}}.
	\end{equation*}
	The norm equivalence \eqref{eq:norm_equivalence_d} implies $\norm{\mathbf{v}}_{\mathbb{R}^{\tilde{M}}}^2\gtrsim_{T,K}\norm{\mathbf{v}}_{\mathbf{G}}^2$ for any $v\in\mathbb{R}^{\tilde{M}}$ and thus
	\begin{equation}
		\abs{(\mathbf{u}-\mathbf{u}^l)\transpose\mathbf{M} \mathbf{u}^l}\lesssim_{T,K,\Xi} \delta \sup_{\mathbf{w}\in\Gamma}\abs{\mathbf{w}\transpose\mathbf{M} \mathbf{w}},\quad \mathbf{u}\in\Gamma.\label{eq:aux_Bayes_Prob_2}
	\end{equation}
	Combining \eqref{eq:aux_Bayes_Prob_1} and \eqref{eq:aux_Bayes_Prob_2}, we find for any $0<\delta<1$ the bound
	\begin{equation*}
		\sup_{\mathbf{u}\in\Gamma}\abs{\mathbf{u}\transpose\mathbf{M} \mathbf{u}}\le \tilde{C}\big((\delta^2+2\delta )\sup_{\mathbf{w}\in\Gamma}\abs{\mathbf{w}\transpose \mathbf{M} \mathbf{w}} + \max_{1\le l\le N(\delta)}\abs{(\mathbf{u}^l)\transpose \mathbf{M} \mathbf{u}^l}\big),
	\end{equation*}
	where the constant $\tilde{C}$ depends on $K$, $T$ and $\Xi$. Fixing any $\delta$ small enough such that $\tilde{C}(\delta^2+2\delta)\le 1/2$, we obtain
	\begin{equation*}
		\sup_{\mathbf{u}\in\Gamma}\abs{\mathbf{u}\transpose \mathbf{M} \mathbf{u}}\le 2 \tilde{C}\max_{1\le l\le N(\delta)}\abs{(\mathbf{u}^l)\transpose \mathbf{M} \mathbf{u}^l}.
	\end{equation*}
	The union bound implies that the probability \eqref{eq:Bayes_Probability_2} is bounded by
	\begin{equation*}
		N(\delta)\sup_{\mathbf{u}\in \Gamma}\prob*[f_0]{\abs{\mathbf{u}\transpose\mathbf{M} \mathbf{u}}\ge \sigma Cx/(2\tilde{C})}.
	\end{equation*}
	\textbf{Step 3 (Conclusion).} Introduce for $\mathbf{u}\in\mathbb{R}^{\tilde{M}}$ the function $u=\mathbf{u}\transpose \mathbf{\Psi}$ and note that
	\begin{equation*}
		\mathbf{u}\transpose \mathbf{M} \mathbf{u} = \int_0^T\frac{1}{\lambda}\int_\Lambda u(X_t(y))^2\,\D y\D t - \EV[f_0][][\bigg]{\int_0^T\frac{1}{\lambda}\int_\Lambda u(Z_t(y))^2\,\D y\D t} = \mathcal{G}_{\lambda}(u^2),
	\end{equation*}
	we can apply Theorem \ref{thm:SubgaussianConcentration} \eqref{num:SubgaussianThm2} to find
	\begin{equation*}
		\prob*[f_0]{\abs{\mathbf{u}\transpose \mathbf{M} \mathbf{u}} \ge \lambda^{-1/2} \hat{C}x/(2\tilde{C})}\le 2\exp\left(-\frac{x^2\hat{C}^2}{8C^2\tilde{C}^2\norm{u}_{L^2(\mathbb{R})}^4}\right),\quad \mathbf{u}\in\Gamma, x\ge 0,
	\end{equation*}
	with $\norm{\mathbf{u}}_{\mathbb{R}^{\tilde{M}}}=\norm{u}_{L^2(\mathbb{R})}\lesssim_{T,K,\Xi} \norm{\mathbf{u}}_{\mathbf{G}}\le 1$ by the norm equivalence \eqref{eq:norm_equivalence_d}. Since we have $\norm{}_{\mathbb{R}^{\tilde{M}}}\lesssim_{T,K,\Xi} \norm{}_{\mathbf{G}}\lesssim_{T,K}\norm{}_{\mathbb{R}^{\tilde{M}}}$ and $\delta$ is fixed, the covering number $N(\delta)$ satisfies $N(\delta)\le e^{\hat{c} \tilde{M}}$ for some constant $\hat{c}<\infty$ depending on $K$, $T$ and $\Xi$. Consequently, we find
	\begin{equation*}
		\prob[f_0][][\bigg]{\sup_{g,h\in V_M\colon g\neq h} \abs[\bigg]{\frac{h_\lambda(g,h)^2- d_\lambda(g,h)^2}{d_\lambda(g,h)^2}}\ge \lambda^{-1/2} \hat{C}x} \le N(\delta)2\exp\left(-\frac{x^2\hat{C}^2}{8C_{T,K}^2\tilde{C}^2}\right) \le 2\exp\left(c\tilde{M} - \frac{x^2\hat{C}^2}{8C^2\tilde{C}^2}\right).
	\end{equation*}
	Readjusting the constant $\hat{C}$ completes the proof.
\end{proof}

\begin{proof}[Proof of Lemma \ref{lem:ExpansionLaplaceTransform}]~
\begin{enumerate}[(a)]
\item Let $Z_{\lambda} \coloneq \int_{D_\lambda}e^{\ell_\lambda(f)}\D\Pi_{\lambda}(f)= \Pi_\lambda(D_\lambda)$ and write $\tilde{\gamma}=P_M\gamma_0\in V_M$, $g=f-u\lambda^{-1/2} \tilde{\gamma}$. Note that 
    \begin{align*}
        \sqrt{\lambda}\mathcal{G}_\lambda((f-f_0)\tilde{\gamma}) = \frac{1}{\sqrt{\lambda}} \int_0^T \iprod{(f-f_0)(X_t)}{\tilde{\gamma}(X_t)}_{L^2(\Lambda)} - \sqrt{\lambda}\iprod{f-f_0}{\tilde{\gamma}}_0.
    \end{align*}
    We thus find from the LAN-expansion (Lemma \ref{lem:LAN}) under $\prob[f_0]{}$ that $\ell_\lambda(f)-\ell_\lambda(g)$ equals
	\begin{align*}
			& u W_{\lambda}(\tilde{\gamma}) + \frac{u^2}{2}\mathcal{I}_\lambda(\tilde{\gamma}) -\frac{u}{\sqrt{\lambda}} \int_0^T \iprod{(f-f_0)(X_t)}{\tilde{\gamma}(X_t)}_{L^2(\Lambda)}\,\D t\\
			&\hspace{10em}=u W_{\lambda}(\tilde{\gamma}) + \frac{u^2}{2}\mathcal{I}_\lambda(\tilde{\gamma}) -u\sqrt{\lambda} \mathcal{G}_{\lambda}((f-f_0)\tilde{\gamma}) - u\sqrt{\lambda}\iprod{f-f_0}{\tilde{\gamma}}_0.
	\end{align*}
	With this and $$r_{\lambda}\coloneq\norm{p_{f_0}}_{L^\infty(\Xi)}\abs{u}\sqrt{\lambda} K\epsilon_\lambda\norm{\gamma_0-\tilde{\gamma}}_{L^2(\Xi)}$$ we get
	\begin{align}
		\EV[\lambda][D_\lambda]{e^{uG_\lambda(\gamma)}\given X^\lambda} 
			& =Z_\lambda^{-1}\int_{D_\lambda}e^{uG_\lambda(\gamma)+\ell_\lambda(f)}\,\D \Pi_\lambda(f)\nonumber\\
			&=Z_\lambda^{-1}\int_{D_\lambda}e^{\frac{u^2}{2}\mathcal{I}_\lambda(\tilde{\gamma}) +u\sqrt{\lambda} \iprod{f-f_0}{\gamma_0-\tilde{\gamma}}_0+\ell_\lambda(g)}\,\D \Pi_\lambda(f)\nonumber\\
			& \leq e^{\frac{u^2}{2}\mathcal{I}_\lambda(\tilde{\gamma})+r_{\lambda}} Z_\lambda^{-1}\int_{D_\lambda}e^{\ell_\lambda(g)}\,\D \Pi_\lambda(f).\label{eq:LaplaceTransformUpperBound_1}
	\end{align}
    Set $D_{\lambda,u}\coloneq \Set{g=f-u\lambda^{-1/2} \tilde{\gamma}\given f\in D_\lambda}\subset V_M$ and let $\Pi_{\lambda,u}$ be the law of $f-u\lambda^{-1/2} \tilde{\gamma}$ under the rescaled prior. Recalling the RKHS norm $\norm{h}_{\mathbb{H}^\lambda}= \norm{h}_{H^{\beta_0}(\Xi)}$, $h\in V_M$, from \eqref{eq:RKHS} the Cameron-Martin theorem \citep[Theorem 2.6.13]{gineMathematicalFoundationsInfiniteDimensional2015} implies
	\begin{align*}
		\int_{D_{\lambda}}e^{\ell_{\lambda}(g)}\D \Pi_{\lambda}(f)
		 & = \int_{D_{\lambda,u}}e^{\ell_{\lambda}(g)}\D \Pi_{\lambda,u}(g) = \int_{D_{\lambda,u}}e^{\ell_{\lambda}(g)}\frac{\D \Pi_{\lambda,u}}{\D \Pi_\lambda}(g)\,\D \Pi_{\lambda}(g)                        \\
		 & = \int_{D_{\lambda,u}}e^{\ell_\lambda(g)}\exp{\left(-u\lambda^{-1/2}\iprod{\tilde{\gamma}}{g}_{H^{\beta_0}(\Xi)}-\frac{u^2}{2\lambda}\norm{\tilde{\gamma}}_{H^{\beta_0}(\Xi)}^2\right)}\,\D\Pi_\lambda(g).
	\end{align*}
    For $g\in D_{\lambda,u}$ and $f\in D_\lambda$ we have $f=P_Mf$ and $\norm{P_M(f-f_0)}_{H^{\beta_0}(\Xi)}\leq 2^{M\beta_0}\norm{f-f_0}_{L^2(\Xi)}\leq K 2^{M\beta_0}\epsilon_\lambda\lesssim \log(\lambda)$, so
	\begin{align}
		\norm{g}_{H^{\beta_0}(\Xi)} \lesssim \log(\lambda) + \norm{f_0}_{H^{\beta_0}(\Xi)} + \abs{u}\lambda^{-1/2}\norm{\tilde\gamma}_{H^{\beta_0}(\Xi)}.\label{eq:g_upper}
	\end{align}
    We thus have for a constant $c$ depending on $f_0$ and $K$
	\begin{align*}
		&\abs*{-u\lambda^{-1/2}\iprod{\tilde{\gamma}}{g}_{H^{\beta_0}(\Xi)}-\frac{u^2}{2\lambda}\norm{\tilde{\gamma}}_{H^{\beta_0}(\Xi)}^2} \\
        &\quad \leq c\left(u^2\lambda^{-1}\norm{\tilde{\gamma}}^2_{H^{\beta_0}(\Xi)} + \abs{u}\lambda^{-1/2}\log(\lambda)\norm{\tilde{\gamma}}_{H^{\beta_0}(\Xi)}\right)=\colon \bar{r}_{\lambda},
	\end{align*}
	hence 
	\begin{align*}
		Z_\lambda^{-1}\int_{D_\lambda}e^{\ell_\lambda(g)}\,\D \Pi_\lambda(g) 
		&\leq e^{\bar{r}_\lambda}Z_\lambda^{-1}\int_{D_{\lambda,u}}e^{\ell_\lambda(g)}\,\D \Pi_\lambda(g)
		=e^{\bar{r}_\lambda} \frac{\Pi_\lambda(D_{\lambda,u}\,|\,X^\lambda)}{\Pi_\lambda(D_{\lambda}\,|\,X^\lambda)}.
	\end{align*}	
	Combining this with \eqref{eq:LaplaceTransformUpperBound_1} we have shown
	\begin{align}
		\EV*[\lambda][D_\lambda]{e^{uG_\lambda(\gamma)}\given X^\lambda} \leq \frac{e^{\frac{u^2}{2}\mathcal{I}_\lambda(\tilde{\gamma})+r_\lambda+\bar r_\lambda}}{\Pi_\lambda(D_{\lambda}\,|\,X^\lambda)}. \label{eq:LaplaceTransformUpperBound}
	\end{align}	
	To obtain the result for $\bar\gamma=0$, use the basic inequality $a\leq a^2+1$ for $a\geq 0$ such that $\bar{r}_\lambda\leq c(2u^2\lambda^{-1}\log(\lambda)^2\norm{\tilde\gamma}^2_{H^{\beta_0}(\Xi)}+1)$, as well as the relations $\norm{\tilde{\gamma}}_{H^{\beta_0}(\Xi)}\leq 2^{M\beta_0}\norm{\tilde\gamma}_{L^2(\Xi)}\lesssim \norm{p^{-1}_{f_0}}_{L^\infty(\Xi)}2^{M\beta_0}\norm{\gamma}_{L^\infty(\Xi)}$ and $\lambda^{-1/2}2^{M\beta_0}\log(\lambda) =\lambda^{-(1/2)/(2\beta_0+1)}\log(\lambda)=o(1)$. The result for general $\bar\gamma$ follows by linearity of $g\mapsto G_\lambda(g)$.
\item The proof of \eqref{eq:LaplaceTransformUpperBound} also shows
	\begin{align*}
		\EV[\lambda][D_\lambda]{e^{uG_\lambda(\gamma)}\given X^\lambda} \geq e^{\frac{u^2}{2}\mathcal{I}_\lambda(\tilde\gamma)-r_\lambda-\bar r_\lambda}\frac{\Pi_\lambda(D_{\lambda,u}\,|\,X^\lambda)}{\Pi_\lambda(D_{\lambda}\,|\,X^\lambda)}.	
	\end{align*}	
	Since $p_{f_0}\in C^1(\Xi)$ by Proposition \ref{prop:DensityBoundedDerivative} and because $p_{f_0}$ is bounded away from zero on $\Xi$, it follows that $1/p_{f_0}$ induces a bounded multiplication operator on $C^1(\Xi)$, hence $\norm{\gamma_0}_{C^1(\Xi)}\lesssim \norm{\gamma}_{C^1(\Xi)}$. The Sobolev embedding yields \begin{align}
	    \norm{\gamma_0}_{C^1(\Xi)}\lesssim \norm{\gamma}_{H^\rho(\Xi)}.\label{eq:gamma0_bound}
	\end{align}
    The wavelet Jackson estimates \eqref{eq:Bayes_SupNorm_Approximation} thus imply $\norm{\gamma_0-P_M\gamma_0}_{L^\infty(\Xi)}\lesssim 2^{-M}$. Moreover, $\norm{\tilde\gamma}_{H^1(\Xi)}\leq \norm{\gamma_0}_{H^1(\Xi)}\lesssim_{\Xi} \norm{\gamma_0}_{C^1(\Xi)}$, so with $\beta_0>3/2$
    \begin{align}
        \norm{\tilde\gamma}_{H^{\beta_0}(\Xi)}\leq 2^{M(\beta_0-1)}\norm{\tilde\gamma}_{H^{1}(\Xi)}\lesssim_{\Xi} 2^{M(\beta_0-1)}\norm{\gamma_0}_{C^1(\Xi)}.\label{eq:projected_gamma_norm}
    \end{align}
    Since $\lambda^{-1/2}2^{M\beta_0}\log(\lambda) =\sqrt{\lambda}\epsilon_{\lambda}2^{-M}=\lambda^{-(1/2)/(2\beta_0+1)}\log(\lambda)=o(1)$, this means $r_\lambda,\bar{r}_\lambda\rightarrow 0$  as $\lambda\rightarrow\infty$, such that
    \begin{align*}
		\frac{\EV[\lambda][D_\lambda]{e^{uG_\lambda(\gamma)}\given X^\lambda}}{e^{\frac{u^2}{2}\mathcal{I}_\lambda(\tilde\gamma)}\frac{\Pi_\lambda(D_{\lambda,u}\,|\,X^\lambda)}{\Pi_\lambda(D_{\lambda}\,|\,X^\lambda)}}\overset{\prob[f_0]{}}{\rightarrow}1.	
	\end{align*}
    We proceed by studying the denominator in the last display. Since $\gamma_0$ is supported on $\Xi$ and is Lipschitz-continuous as a function on $\Xi$, by the Kirszbraun extension theorem (Exercise 3.6.14 of \citet{gineMathematicalFoundationsInfiniteDimensional2015}) $\gamma_0$ extends to a Lipschitz-continuous function on $\mathbb{R}$ with the same Lipschitz constant. We can therefore apply Lemma \ref{lem:LAN} to show that $\mathcal{I}_\lambda(\gamma_0)\rightarrow \norm{\gamma_0}^2_{0}$ in $\prob[f_0]{}$-probability. Since also 
    \begin{equation*}
        \abs{\mathcal{I}_{\lambda}(\tilde\gamma)-\mathcal{I}_{\lambda}(\gamma_0)}\leq T\norm{\gamma_0-\tilde\gamma}_{\infty}(\norm{\gamma_0}_{\infty}+\norm{\tilde\gamma}_{\infty})\rightarrow 0,
    \end{equation*}
    we conclude $\mathcal{I}_\lambda(\tilde\gamma)\rightarrow \norm{\gamma_0}^2_{0}$ in $\prob[f_0]{}$-probability. By \eqref{eq:D_negligible} the claimed convergence in \eqref{num:Laplace_Convergence} follows if we can show $\Pi_\lambda(D^c_{\lambda,u}\,|\,X^{\lambda})=\smallo_{\prob[f_0]{}}(1)$. By conditioning on the data and using Lemma \ref{lem:Prior} \eqref{lem:aux_rescaled_Prior_restrict} together with the posterior contraction from Theorem \ref{thm:PosteriorContraction} we only have to argue that $D_{\lambda,u}\subset B(0,\norm{}_{H^{\beta_0}(\Xi)},L)$ for a large enough radius $L>0$ (depending on $u$). But this follows already from \eqref{eq:g_upper}, since by \eqref{eq:projected_gamma_norm}, $\lambda^{-1/2}\norm{\tilde\gamma}_{H^{\beta_0}(\Xi)}\lesssim \lambda^{-1/2}2^{M(\beta_0-1)}=o(1)$.
    \end{enumerate}\qedhere
\end{proof}

Recall from Definition 2.3.5 of \citet{gineMathematicalFoundationsInfiniteDimensional2015} that a centred process $Z=(Z(f)\colon f\in\mathcal{F})$ for an index set $\mathcal{F}$ with pseudo-distance $d$ is sub-Gaussian if its increments are sub-Gaussian, i.e.
\begin{align*}
    \prob{\abs{Z(f)-Z(g)}\geq x} \leq 2\exp{\left(-\frac{x^2}{2d^2(f,g)}\right)},\quad x>0,\quad f,g\in\mathcal{F}.
 \end{align*}
The next lemma is a basic deviation inequality for $Z$ uniformly over $\mathcal{F}$ and follows from standard sub-Gaussian calculus (see e.g., Chapter 2.3 of \citet{gineMathematicalFoundationsInfiniteDimensional2015}, and also Lemma 1 in \citet{nicklNonparametricStatisticalInference2020}). 

\begin{lemma}\label{lem:UniformSubgaussianConcentration}
	Let $(Z(f):f\in \mathcal{F})$ be a sub-Gaussian process with a pseudo-distance $d$. Denote by $\operatorname{diam}(\mathcal{F})=\sup_{f,g\in\mathcal{F}}d(f,g)$ the $d$-diameter of $\mathcal{F}$ and let
	\begin{equation*}
		J_{\mathcal{F}}\coloneq \int_0^{\operatorname{diam}(\mathcal{F})}\sqrt{\log{2N(\mathcal{F},6d,\tau)}},\,\D\tau<\infty
	\end{equation*}
    be the Dudley entropy integral, where $N(\mathcal{F},6d,\tau)$ denotes the covering number of the set $\mathcal{F}$ by $d$-balls of radius $\tau/6$. Then we have $\EV{\sup_{f\in\mathcal{F}}\abs{Z(f)}}\leq 4\sqrt{2}J_{\mathcal{F}}$ and
	\begin{align*}
		\prob*{\sup_{f\in\mathcal{F}}\abs{Z(f)}\geq 4\sqrt{2}J_{\mathcal{F}}+196J_{\mathcal{F}}x}\leq e^{-x^2/2}.
	\end{align*}
\end{lemma}
\begin{proof}[Proof of Lemma \ref{lem:SubGaussianBounds}]~
\begin{enumerate}[(a)]
\item Fix any $g\in\mathcal{F}$. Since $g$ is supported on $\Xi$, by Theorem \ref{thm:SubgaussianConcentration} \eqref{num:SubgaussianThm2} and linearity of $\gamma\mapsto \mathcal{G}_\lambda(\gamma)$, $\sqrt{\lambda}(\mathcal{G}_{\lambda}(g)-\EV[f_0]{\mathcal{G}_\lambda(g)})$ has sub-Gaussian increments with respect to $\norm{}_{L^1(\Xi)}$. Moreover, $g$ is supported on $\Xi$ and thus globally Lipschitz-continuous. This allows us to apply Lemma \ref{lem:UniformSubgaussianConcentration} and Theorem \ref{thm:SubgaussianConcentration} \eqref{num:SubgaussianThm3} to obtain
    \begin{equation} 
        \EV[f_0]{\sup_{g\in\mathcal{F}}\abs{\mathcal{G}_{\lambda}(g)}} \leq \EV[f_0]{\sup_{g\in\mathcal{F}}\abs{\mathcal{G}_{\lambda}(g)-\EV[f_0]{\mathcal{G}_{\lambda}(g)}}}+\sup_{g\in\mathcal{F}}\abs{\EV[f_0]{\mathcal{G}_{\lambda}(g)}}\lesssim \lambda^{-1/2}J_{\mathcal{F}} + \lambda^{-1}\sup_{g\in\mathcal{F}}\norm{g}_{C^1(\Xi)}.\label{eq:Maximal_aux1}
    \end{equation}
    We proceed by bounding these two summands. We have $\gamma\in H^{\rho'}(\Xi)$ for $3/2<\rho'<3/2+\varepsilon$ for any small enough $\varepsilon>0$. Using the Sobolev embedding and \eqref{eq:projected_gamma_norm} with $\rho'$ instead of $\beta$ for $\gamma\in H^{\rho'}(\Xi)$ we find $\norm{P_M\gamma_0}_{C^1(\Xi)}\lesssim 2^{M(\rho'-1)}$. This means for $g\in\mathcal{F}$
    \begin{align}
		\norm{g}_{C^1(\Xi)} 
            &\leq \norm{f-f_0}_{C^1(\Xi)}\norm{P_M\gamma_0}_{C^1(\Xi)}\lesssim 2^{M(\rho'-1)}\norm{f-f_0}_{H^{\rho'}(\Xi)}\nonumber\\
            &\leq 2^{M(\rho'-1)} \norm{P_M(f-f_0)}_{H^{\rho'}(\Xi)} + 2^{M(\rho'-1)}\norm{f_0-P_Mf_0}_{H^{\rho'}(\Xi)}\nonumber \\
            &\leq 2^{M(2\rho'-1)} \norm{f-f_0}_{L^2(\Xi)} + 2^{-M(\beta-2\rho'+1)}\nonumber \\
            &\lesssim 2^{M(2\rho'-1)}\varepsilon_\lambda + 2^{-M(\beta-2\rho'+1)} \lesssim \lambda^{-(\beta_0-2\rho'+1)/(2\beta_0+1)},\label{eq:lipschitz_g}
	\end{align}
    such that $\mathcal{F}$ is contained in a $C^1(\Xi)$-ball of radius $\Gamma\coloneq C\lambda^{-(\beta_0-2\rho'+1)/(2\beta_0+1)}$ for some constant $C>0$. We obtain from Corollary 4.3.38 of \citet{gineMathematicalFoundationsInfiniteDimensional2015} for any $\epsilon>0$ the metric entropy estimate
	\begin{align*}
		N(\mathcal{F}, \norm{}_{L^1(\Xi)}, \epsilon) \lesssim N(B(0,\norm{}_{C^1(\Xi)},\Gamma),\norm{}_{L^1(\Xi)},\epsilon)\lesssim \exp{(\Gamma/\epsilon)}.
	\end{align*}
    Moreover, we find that $\mathcal{F}$ has $\norm{}_{L^1(\Xi)}$-diameter 
	\begin{align*}
		\sup_{g,\bar{g}\in\mathcal{F}}\norm{g-\bar{g}}_{L^1(\Xi)}&\leq 2\sup_{g\in\mathcal{F}}\norm{g}_{L^1(\Xi)}\lesssim \sup_{f\in D_{\lambda}}\norm{f-f_0}_{L^2(\Xi)}\norm{P_M\gamma_0}_{L^\infty(\Xi)}
        \leq 2^{M(\rho'-1)}\epsilon_\lambda,
	\end{align*}
	so the Dudley entropy integral in Lemma \ref{lem:UniformSubgaussianConcentration} satisfies 
	\begin{align}
    \begin{split}
		J_{\mathcal{F}} &\lesssim \int_0^{\Gamma 2^{-M\rho'}}\sqrt{\log (2) + \Gamma/\epsilon}\,\D\epsilon \lesssim \Gamma 2^{-M\rho'/2}.
    \end{split}\label{eq:Dudley_Integral}
	\end{align} 
	Combining this bound with \eqref{eq:Maximal_aux1} and \eqref{eq:lipschitz_g} gives 
	\begin{align*}
		\EV[f_0]{\sup_{g\in\mathcal{F}}\abs{\mathcal{G}_{\lambda}(g)}} \lesssim \lambda^{-1/2}\Gamma 2^{-M\rho'/2} + \lambda^{-1}\Gamma = \lambda^{-1/2}(\lambda^{-\frac{\beta_0+1-3\rho'/2}{2\beta_0+1}} + \lambda^{-\frac{2\beta_0+3/2-2\rho'}{2\beta_0+1}}).
	\end{align*}
	The claimed maximal inequality follows for $\beta_0>3/2$ from taking $\varepsilon$ sufficiently small.
\item Fix any $\gamma\in B(0,\norm{}_{H^\rho(\Xi)},1)$. We have $\EV{e^{u\mathbb{W}(\gamma)}} = e^{u^2\norm{\gamma_0}_0^2/2}\leq e^{u^2c\norm{\gamma}_{L^\infty(\Xi)}^2/2}$ for any $u\in\mathbb{R}$ and a constant $c>0$ depending on $p_{f_0}$. In particular, $\mathbb{W}$ is sub-Gaussian on $\bar{\mathcal{F}}_{J}=\Set{h-P_Jh\,\colon\, h\in H^{\rho}(\Xi),\norm{h}_{H^{\rho}(\Xi)}\leq 1}$, $J\geq 1$, relative to $\norm{}_{L^{\infty}(\Xi)}$, and by the Jackson estimate \eqref{eq:Bayes_SupNorm_Approximation} and the Sobolev embedding $\bar{\mathcal{F}}_J$ has $\norm{}_{L^{\infty}(\Xi)}$-diameter 
	\begin{align}
		\sup_{\gamma,\bar\gamma\in B(0,\norm{}_{H^{\rho}(\Xi)},1)}\norm{\gamma-\bar\gamma-P_J(\gamma-\bar\gamma)}_{L^{\infty}(\Xi)}\lesssim 2^{-J}\norm{\gamma-\bar\gamma}_{C^1(\Xi)}\lesssim  2^{-J}.\label{eq:diameter_bound}
	\end{align}
	By another application of the Sobolev embedding and $\rho>3/2$, $\bar{\mathcal{F}}_{J}$ is contained in a $C^{1}(\Xi)$-ball of radius $C$ for some $C>0$. Consequently, using again Corollary 4.3.38 of \citet{gineMathematicalFoundationsInfiniteDimensional2015} we get
    \begin{align}
        N(\bar{\mathcal{F}}_J, \norm{}_{L^\infty(\Xi)}, \epsilon) \lesssim N(B(0,\norm{}_{C^{1}(\Xi)},C), \norm{}_{L^\infty(\Xi)}, \epsilon) \lesssim \exp{(C/\varepsilon)}.\label{eq:sub_entropy}
    \end{align}
    The claim follows from Lemma \ref{lem:UniformSubgaussianConcentration}, noting similar to \eqref{eq:Dudley_Integral} that the Dudley entropy integral satisfies $J_{\bar{\mathcal{F}}_{J}}\lesssim 2^{-J/2}$.
\item We use the notation from Lemma \ref{lem:ExpansionLaplaceTransform} and write $\sigma_1=d_1(\gamma,\bar \gamma)$, $\sigma_2=d_2(\gamma,\bar \gamma)$. Lemma \ref{lem:ExpansionLaplaceTransform} \eqref{num:Laplace_Metrics} and \eqref{eq:D_negligible} show for a random variable $0<C_\lambda<\infty$ satisfying $C_\lambda\rightarrow 1$ in $\prob[f_0]{}$-probability that
	\begin{align*}
		\EV[\lambda][D_\lambda]{e^{u (G_\lambda(\gamma)-G_\lambda(\bar\gamma))}\given X^\lambda} \leq C_\lambda e^{u^2\sigma_1^2/2+\abs{u}\sigma_2}.
	\end{align*}
	By the exponential Markov inequality, this gives for any $x>0$, $u\in\mathbb{R}$ the deviation inequality
	\begin{align*}
		\Pi^{D_\lambda}_{\lambda}\left(\abs{G_\lambda(\gamma)-G_\lambda(\bar\gamma)}\geq x\,\Big\vert\, X^\lambda\right) \leq 2C_{\lambda}e^{u^2\sigma^2_1/2+\abs{u}\sigma_2-ux}.
	\end{align*}
	Applying this to $u=x/(\sigma_1^2+\sigma_2x)$ implies 
	\begin{align*}
		\Pi^{D_\lambda}_{\lambda}\left(\abs{G_\lambda(\gamma)-G_\lambda(\bar\gamma)}\geq x\,\Big\vert\, X^\lambda\right) \leq 2eC_{\lambda}e^{-x^2/(2(\sigma_1^2+\sigma_2x))}, 
	\end{align*}
	which in turn gives for $z>0$
	\begin{align}
		\Pi^{D_\lambda}_{\lambda}\left(\abs{G_\lambda(\gamma)-G_\lambda(\bar\gamma)}\geq \sigma_1\sqrt{2z}+2\sigma_2z\,\Big\vert\, X^\lambda\right) \leq 2eC_{\lambda}e^{-z}.\label{eq:G_subgaussian}
	\end{align}	
	Up to modifying the metrics $d_1$, $d_2$ to absorb the factor $eC_\lambda$ into the upper bound, this means that the process $G_\lambda$ has (conditional on $X^\lambda$) mixed tails with respect to $d_1$, $d_2$ in the sense of Equation (3.8) of \citet{dirksen2015a}. As discussed before \eqref{eq:gamma0_bound}, $1/p_{f_0}\in C^1(\Xi)$, such that $1/p_{f_0}$ is a bounded multiplication operator on $H^1(\Xi)$ by Equation (1.87) of \citet{Yagi2010}. With this, if $\norm{\gamma}_{H^\rho(\Xi)}\leq 1$, then we have by the Sobolev embedding 
	\begin{align}
        d_1(\gamma,0) = 
		&\norm{P_M(\gamma-P_J\gamma)_0}_{L^\infty(\Xi)}
            \lesssim \norm{P_M(\gamma-P_J\gamma)_0}_{H^{1}(\Xi)}   leq \norm{(\gamma-P_J\gamma)_0}_{H^{1}(\Xi)}\nonumber\\
            &\lesssim \norm{\gamma-P_J\gamma}_{H^{1}(\Xi)} \lesssim 2^{-J(\rho-1)}\label{eq:d1}
	\end{align}
    and $d_2(\gamma,0)\lesssim \sqrt{\lambda}\epsilon_\lambda 2^{-M\rho}\norm{\gamma}_{H^\rho(\Xi)}\lesssim \lambda^{-(\rho-1/2)/(2\beta_0+1)}$. Together with the tail bound in \eqref{eq:G_subgaussian}, Lemma A.2 of \citet{dirksen2015a} implies  the moment bound	
	\begin{align*}
		\sup_{\gamma\in\bar{\mathcal{F}}_J}\EV[\lambda][D_\lambda]{\abs{G_\lambda(\gamma)}\given X^\lambda} \lesssim \sup_{\gamma\in\bar{\mathcal{F}}_J}(d_1(\gamma,0)+d_2(\gamma,0)) \lesssim 2^{-J(\rho-1)}+\lambda^{-(\rho-1/2)/(2\beta_0+1)}.
	\end{align*}
    We then obtain the claimed maximal inequality from Theorem 3.5 of \citet{dirksen2015a} by proving the following two entropy integral bounds:
	\begin{align}
    \begin{split}
		\int_0^{\infty}\sqrt{\log (2N(\bar{\mathcal{F}}_J, d_1, \epsilon))}\,\D\epsilon  &\lesssim 2^{-J(\rho-1)},\\
        \int_0^{\infty}\log (2N(\bar{\mathcal{F}}_J, d_2, \epsilon))\,\D\epsilon
			&\lesssim \lambda^{-(\rho-3/2)/(2\beta_0+1)}.
            \end{split}
            \label{eq:metric_entropy_integral_bounds}
	\end{align}	
	For this, let $\tilde{\mathcal{F}}_J=\{P_M\gamma_0:\gamma\in\bar{\mathcal{F}}_J\}$. We have $N(\bar{\mathcal{F}}_J,d_1,\varepsilon)=N(\tilde{\mathcal{F}}_J,\bar{C}\norm{}_{L^\infty(\Xi)},\varepsilon)$. Recalling \eqref{eq:d1}, we have $\norm{P_M(\gamma-P_J\gamma)_0}_{H^1(\Xi)}\lesssim 2^{-J(\rho-1)}$ if $\norm{\gamma}_{H^\rho(\Xi)}\leq 1$, such that $\tilde{\mathcal{F}}_J\subset B(0,\norm{}_{H^1(\Xi)},c2^{-J(\rho-1})$ for a constant $c=c(p_{f_0},\Xi)$ depending only on $p_{f_0}$ and $\Xi$. From Corollary 4.3.38 of \citet{gineMathematicalFoundationsInfiniteDimensional2015} we get $N(\bar{\mathcal{F}}_J,d_1,\varepsilon)\lesssim \exp(c2^{-J(\rho-1)}/\varepsilon)$, and \eqref{eq:diameter_bound} gives $    \sup_{\gamma,\bar\gamma\in\bar{\mathcal{F}}_J}d_1(\gamma,\bar\gamma)\lesssim 2^{-J(\rho-1)}$. The first integral bound in \eqref{eq:metric_entropy_integral_bounds} follows as in \eqref{num:Maximal2}. Regarding the second bound, note $\bar{\mathcal{F}}_J\subset B(0,\norm{}_{H^\rho(\Xi)},1)$ and $d_2(\gamma,\bar\gamma)\leq \sqrt{\lambda}\epsilon_\lambda\norm{\gamma-\bar\gamma}_{L^2(\Xi)}$ we have by Corollary 4.3.38 of \citet{gineMathematicalFoundationsInfiniteDimensional2015}
	\begin{align*}
		N(\bar{\mathcal{F}}_J, d_2, \epsilon) 
			&\leq N(\bar{\mathcal{F}}_J,\norm{}_{L^2(\Xi)},\varepsilon/(\sqrt{\lambda}\varepsilon_\lambda)) \lesssim \exp{((C\sqrt{\lambda}\varepsilon_\lambda/\varepsilon)^{1/\rho})}
	\end{align*}
    for a constant $C>0$. The second integral bound is obtained from
	\begin{align*}
		\int_0^{\infty}\log (2N(\bar{\mathcal{F}}_J, d_2, \epsilon))\,\D\epsilon 
			&\lesssim \int_0^{\bar c\sqrt{\lambda}\epsilon_\lambda 2^{-M\rho}}(\log (2) + \exp{((C\sqrt{\lambda}\varepsilon_\lambda/\varepsilon)^{1/\rho})}\,\D\epsilon\\
            &\lesssim  \sqrt{\lambda}\epsilon_\lambda 2^{-M\rho} + (\sqrt{\lambda}\varepsilon_\lambda)^{1/\rho}(\sqrt{\lambda}\epsilon_\lambda 2^{-M\rho})^{1-1/\rho}\\
            &\lesssim \lambda^{-(\rho-3/2)/(2\beta_0+1)}.
	\end{align*}
\item Let $\langle h\rangle_t=\lambda^{-1}\int_0^t\norm{h(X_s)}^2_{L^2(\Lambda)}\D s$ denote the quadratic variation of the continuous martingale $W_\lambda(h)\int_0^t h(X_s)\D W_s$, $t\geq 0$, for a square integrable function $h$. Then $\langle h\rangle_T\leq T\norm{h}_{\infty}^2$ and the Bernstein inequality for continuous martingales (e.g., Page 153 of \citet{RevuzYor1999}) implies for $x\geq 0$ 
	\begin{align*}
		\prob[f_0]{\abs{W_{\lambda}(\gamma)}\geq x} 
			&= \prob[f_0]{\abs{W_{\lambda}(\gamma)}\geq x, \langle \gamma\rangle_T\leq T\norm{\gamma}_{L^\infty(\Xi)}} \leq 2\exp\left(-\frac{x^2}{2T\norm{\gamma}_{L^\infty(\Xi)}^2}\right)
	\end{align*}
    for any $\gamma\in B(0,\norm{}_{H^\rho(\Xi)},1)$. In particular, $W_{\lambda}(\gamma)$ is sub-Gaussian with respect to the distance $d=T\norm{}_{\infty}^2$. Up to a constant this is the distance $d_1$ from \eqref{eq:Metrics_d1d2}, and the claim follows from Lemma \ref{lem:UniformSubgaussianConcentration} and the bound on the Dudley entropy integral relative to $d_1$ in \eqref{eq:metric_entropy_integral_bounds}.
\end{enumerate}
\end{proof}

\section{Probabilistic results}\label{sec:Appendix_Probability}

This section \ref{sec:Appendix_Probability} collects the general properties of the SPDE \eqref{eq:SPDE}, the required upper bounds for the derivative of the density $\partial_x p_{f_0,\lambda}(t,y,x)$ (Proposition \ref{prop:DensityBoundedDerivative}) and the convergence of $X_t(\lambda\bar{y})\xrightarrow{d}Z_t^{\bar{y}}(0)$ as $\lambda\to\infty$ in Proposition \ref{prop:Convergencetoglobal}.

\subsection{Properties of the semi-linear SPDE}\label{subsec:ProbabilisticProperties}

    As shown in Subsection \ref{subsec:Rescaling}, the process $Y_t(y)\coloneq X_t(\lambda^{-1}y)$ for $0\le t\le T$ and $y\in\Lambda$ satisfies the SPDE
    \begin{equation}
        \D Y_t(y) = \nu\Updelta Y_t(y)\,\D t+ f(Y_t(y))\,\D t + \nu^{1/4}\,\D W_t(y),\quad Y_0=\chi,\label{eq:Y_appendix1}
    \end{equation}
    on $\bar{\Lambda}$ with diffusivity level $\nu = \lambda^{-2}$ and noise level $\sigma = \lambda^{-1/2}=\nu^{1/4}$. The equivalence shows that the probabilistic results of \citet{gaudlitzNonparametricEstimationReaction2023} for the $(Y_t)_{t\in[0,T]}$ from \eqref{eq:Y_appendix1} also apply to $(X_t)_{t\in[0,T]}$ from \eqref{eq:SPDE} for any $\lambda\ge 1$. In the following, we collect the key analytical properties of the process $(X_t)_{t\in[0,T]}$.

		\begin{lemma}[Lemma E.1 of \citet{gaudlitzNonparametricEstimationReaction2023}]\label{lem:aux_WellPosedness}
            Let $\lambda\ge 1$ and $X_0\in C(\Lambda)$ be any deterministic initial condition for the SPDE \eqref{eq:SPDE}.
            \begin{enumerate}[(a)]
			\item\label{num:aux_wellposed_a} The semi-linear SPDE \eqref{eq:SPDE} has a unique mild solution $(X_t(y) \,\vert\, 0\le t\le T, y\in\Lambda)$ given by
            \begin{equation}
            \begin{split}
                X_t(y) =\int_\Lambda G_t(y,\eta)X_0(\eta)\,\D\eta+  \int_0^t \int_\Lambda G_{t-s}^\lambda (y,\eta)\mathcal{W}(\D \eta,\D s)+\int_0^t\int_\Lambda G_{t-s}^\lambda (y,\eta)f(X_s(\eta))\,\D \eta\D s.
                \end{split}
                \label{eq:RandomField1}
            \end{equation}
			\item\label{num:aux_wellposed_b} The Malliavin derivative $\mathcal{D}X_t(y)\in L^2([0,T],L^2(\Lambda))$ of $X_t(y)$ from \eqref{eq:RandomField1} exists for all times $0< t\le T$, locations $y\in\Lambda$, and satisfies 
            \begin{equation*}
                \mathcal{D}_{\tau,z}X_t(y) = G_{t-\tau}^\lambda(y,z) + \int_0^t\int_\Lambda G_{t-s}^\lambda(y,\eta)f'(X_s(\eta))\mathcal{D}_{\tau,z}X_s(\eta)\,\D \eta\D s
            \end{equation*}
            for all $0\le\tau<t$, and $z\in\Lambda$.
			\item\label{num:aux_wellposed_c} The process $(X_t(y) \,\vert\, 0\le t\le T, y\in\Lambda)$ from \eqref{eq:RandomField1} is also an analytically weak solution to \eqref{eq:SPDE} in the sense that
			\begin{equation*}
				 \iprod*{X_t}{\phi}_{L^2(\Lambda)}-\iprod*{X_0}{\phi}_{L^2(\Lambda)} = \int_0^t \left(\iprod*{X_s}{ \Updelta \phi}_{L^2(\Lambda)} + \iprod*{F(X_s)}{\phi}_{L^2(\Lambda)}\right)\,\D s+\int_0^t\iprod*{\phi}{\D W_s}_{L^2(\Lambda)}
			\end{equation*}
			for all $\phi\in \operatorname{dom}(\Updelta)$ and $0\le t\le T$.
			\end{enumerate}
		\end{lemma}

        \begin{lemma}[Lemma C.3 of \citet{gaudlitzNonparametricEstimationReaction2023}]\label{lem:UniformBoudnednessSolution}
		Let $\lambda\ge 1$, $p\ge 1$ and $X_0\in C(\Lambda)$ be any deterministic initial condition. Then there exists a constant $0<C<\infty$, depending only on $p$, $\norm{X_0}_{L^\infty(\Lambda)}$, $T$, $f(0)$ and $\Lip{f}$, such that for all domain sizes $\lambda\ge 1$, time points $0\le t\le T$ and locations $y\in\Lambda$ we have the bound
		\begin{equation*}
			\EV*[f]{\abs*{X_{t}(y)}^p}\le C<\infty.
		\end{equation*}
		\end{lemma}

        \begin{proposition}[Proposition 4.1 of \citet{gaudlitzNonparametricEstimationReaction2023}]\label{prop:DensityBounds}
			There exist constants $0<c_1\le C_1<\infty$, $0<c_2\le C_2<\infty$, depending only on $\Lip{f}$ and $T$, such that for all domain sizes $\lambda\ge 1$, locations $y\in\Lambda$ and time points $0<t\le T$ the Lebesgue-density $p_{f,\lambda}(t,y,\MTemptyplaceholder)$ of $X_{t}(y)$ started at the deterministic initial condition $X_0\in C(\Lambda)$ exists and satisfies the bound
			\begin{equation*}
				p_{f,\lambda}(t,y,x-\EV[f]{X_{t}(y)})\le C_1 t^{-1/4}\exp\left(-\frac{x^2}{2C_2t^{1/2}}\right),\quad x\in\mathbb{R}.
			\end{equation*}
			Moreover, there exists $0<t_1\le T$ such that
			\begin{equation*}
				p_{f,\lambda}(t,y,x-\EV[f]{X_{t}(y)})\ge c_1 t^{-1/4}\exp\left(-\frac{x^2}{2c_2t^{1/2}}\right),\quad x\in\mathbb{R},\quad 0<t\le t_1.
			\end{equation*}
		\end{proposition}

    		\begin{corollary}[Corollary 4.3 of \citet{gaudlitzNonparametricEstimationReaction2023}]\label{cor:DensityBounds}
            For any $\lambda\ge 1$ and deterministic initial condition $X_0\in C(\Lambda)$ we have the following.
			\begin{enumerate}[(a)]
				\item\label{num:cordensitybounds_a} There exists a constant $0<p_{\max}<\infty$, depending only on $\Lip{f}$ and $T$, such that for all locations $y\in\Lambda$ and time points $0<t \le T$ the Lebesgue-density $p_{f,\lambda}(t,y,\MTemptyplaceholder)$ of $X_{t}(y)$ started at $X_0$ exists and satisfies
			\begin{equation*}
				p_{f,\lambda}(t,y,x)\le p_{\max}t^{-1/4}<\infty,\quad x\in\mathbb{R}.
			\end{equation*}
			\item\label{num:cordensitybounds_b} There exist time points $0<t_0<t_1\le T$ and a constant $0<p_{\min}<\infty$, depending on $\norm{X_0}_{L^\infty(\Lambda)}$, $\Xi$, $t_0$, $t_1$, $\Lip{f}$ and $T$, such that
			\begin{equation*}
				p_{f,\lambda}(t,y,x) \ge p_{\min}>0,\quad x\in\Xi,
			\end{equation*}
			for all $y\in\Lambda$ and $t_0 \le t\le t_1$.
            \item\label{num:cordensitybounds_c} For every $\lambda\ge 1$, $0<t\le T$ and $y\in\Lambda$ the density $p_{f,\lambda}(t,y,\MTemptyplaceholder)$ has Gaussian tails.
			\end{enumerate}
		\end{corollary}
        \begin{proof}
            \eqref{num:cordensitybounds_a} and \eqref{num:cordensitybounds_b} follow directly from Proposition \ref{prop:DensityBounds}. The Gaussian tails in \eqref{num:cordensitybounds_c} follows by combining Proposition \ref{prop:DensityBounds} with the bound $\abs{\EV[f]{X_t(y)}}\le \EV[f]{\abs{X_t(y)}}\le C$ from Lemma \ref{lem:UniformBoudnednessSolution}.
        \end{proof}

\subsection{Upper bounds for the derivative of densities}\label{subsec:DensityDerivatives}

        This subsection contains the upper bound for the derivative of the density $\partial_x p_{f,\lambda}(t,y,x)$, which is required for Step 1 (Lemma \ref{lem:Bayes_BoundConditionalExpectation}) in the proof of the sub-Gaussian concentration in Theorem \ref{thm:SubgaussianConcentration}. We introduce $\mathfrak{H}\coloneq L^2([0,T],L^2(\Lambda))$ and continue with bounding the growth of the Green function $G_t^\lambda$.
        \begin{lemma}\label{lem:aux_GreenfunctionsGrowth}
            For any $\lambda\ge 1$ and $0<t\le T$ we have
            \begin{equation*}
                \norm{G_{t-\MTemptyplaceholder}^\lambda(y,\MTemptyplaceholder)}_{\mathfrak{H}}^2\le \bar{C} t^{1/2}
            \end{equation*}
            for some constant $0<\bar{C}<\infty$ depending on $T$.
        \end{lemma}
        \begin{proof}
            Define $\bar{y} \coloneq y/\lambda\in\bar{\Lambda}$. Then we have
            \begin{equation*}
                \norm{G_{t-\MTemptyplaceholder}^\lambda(y,\MTemptyplaceholder)}_{\mathfrak{H}}^2 = \int_0^t \int_\Lambda G_{t-s}^\lambda(\lambda\bar{y},\eta)^2\,\D \eta\D s = \lambda \int_0^t\int_{\bar{\Lambda}} G_{t-s}^\lambda(\lambda\bar{y},\lambda\bar{\eta})^2\,\D\bar{\eta}\D s.
            \end{equation*}
            Using the rescaling of the Green function $G_t^\lambda(\lambda\bar{y},\lambda{\bar{\eta}}) = \lambda^{-1}G_{\lambda^{-2}t}^1(\bar{y},\bar{\eta})$ for $\bar{y},\bar{\eta}\in\bar{\Lambda}$ from \eqref{eq:Scaling_Heat_Kernel}, we find
            \begin{equation*}
                \norm{G_{t-\MTemptyplaceholder}^\lambda(y,\MTemptyplaceholder)}_{\mathfrak{H}}^2 = \lambda^{-1}\int_0^t\int_{\bar{\Lambda}} G_{\lambda^{-2}(t-s)}^1(\bar{y},\bar{\eta})^2\,\D\bar{\eta}\D s.
            \end{equation*}
            Applying the bound $G_{t}^1(\bar{y},\bar{\eta})\le c (t^{-1/2}\vee 1)\exp(-(\bar{y}-\bar{\eta})^2/(\tilde{c}t))$ for some constants $0<c,\tilde{c}<\infty$ from Theorem 3.2.9 of \citet{daviesHeatKernelsSpectral1990}, we find
            \begin{equation*}
                \norm{G_{t-\MTemptyplaceholder}^\lambda(y,\MTemptyplaceholder)}_{\mathfrak{H}}^2\le c^2\tilde{c}^{1/2}\lambda^{-1}\int_0^t ((\lambda^{-2}(t-s))^{-1/2}\vee 1)^2 \lambda^{-1}(t-s)^{1/2}\,\D s\lesssim_T t^{1/2}.
            \end{equation*}\qedhere
        \end{proof}
    With this preparation, we can proceed to the proof of Proposition \ref{prop:DensityBoundedDerivative}. It relies on the explicit formula 
    \begin{equation*}
        \partial_x^n p_{f,\lambda}(t,y,x) = (-1)^n\EV[f]{\indicator(X_t(y)> x)H_{1^{\otimes (n+1)}}},\quad x\in\mathbb{R}, n\in\mathbb{N}_0,
    \end{equation*}
    which is based on the integration-by-parts formula of the Malliavin calculus and can be found in the proof of Theorem 2.1.4 of \citet{nualartMalliavinCalculusRelated2006}. With the Skorodhod integral $\delta$, the quantity $H_{1^{\otimes (n+1)}}$ is defined as
    \begin{equation*}
        H_{1^{\otimes 1}}\coloneq \delta \left(\frac{\mathcal{D} X_t(y)}{\norm{\mathcal{D} X_t(y)}_{\mathfrak{H}}^2}\right),\quad H_{1^{\otimes (n+1)}}\coloneq \delta\left(\frac{\mathcal{D}X_t(y)}{\norm{\mathcal{D}X_t(y)}_{\mathfrak{H}}^2} H_{1^{\otimes n}}\right),\quad n\in\mathbb{N}.
    \end{equation*}
    Subsequently, $\abs{\partial_x^n p_{f,\lambda}(t,y,x)}$ can be bounded by an successive application of the Hölder-inequality given in Proposition 1.5.6 of \citet{nualartMalliavinCalculusRelated2006} and upper and lower bounds on $\norm{\mathcal{D}X_t(y)}_{\mathfrak{H}}$. The latter bounds are based on the Green function bound of Lemma \ref{lem:aux_GreenfunctionsGrowth} and the Gronwall inequality.
	\begin{proof}[Proof of Proposition \ref{prop:DensityBoundedDerivative}]
		Fix a time point $0<t\le T$ and a location $y\in\Lambda$. Let $\gamma_{t,y}\coloneq \norm{\mathcal{D}X_t(y)}_\mathfrak{H}^2$ and define as in Equations (1.32) and (1.37) of \citet{nualartMalliavinCalculusRelated2006} the semi-norms
		\begin{align*}
			\norm{H}_{k,p}^p&\coloneq \EV[f]{\abs{H}^p}+\sum_{l=1}^{k} \EV[f]{\norm{\mathcal{D}^l H}_{\mathfrak{H}^{\otimes l}}^{p}},\quad k\in\mathbb{N},\quad p\ge 1,\\
            \norm{\tilde{H}}_{k,p}^p&\coloneq \EV[f]{\norm{\tilde{H}}_{\mathfrak{H}}^p}+\sum_{l=1}^{k} \EV[f]{\norm{\mathcal{D}^l \tilde{H}}_{\mathfrak{H}^{\otimes l}}^{p}},\quad k\in\mathbb{N},\quad p\ge 1,
		\end{align*}
		for any sufficiently Malliavin-differentiable random variable $H\in\mathbb{R}$ and $\tilde{H}\in\mathfrak{H}$. For $n\in\Set*{0,\dots, N}$, we use Theorem 2.1.4 and Proposition 2.1.4 of \citet{nualartMalliavinCalculusRelated2006}, to obtain
		\begin{equation*}
			\abs{\partial_x^n p_{f,\lambda}(t,y,x)} = \abs{\EV[f]{\indicator_{\{X_t(y)>x\}}H_{1^{\otimes (n+1)}}}}\le \EV[f]{\abs{H_{1^{\otimes (n+1)}}}^p}^{1/p}\le c_p \norm{\gamma_{t,y}^{-1}\mathcal{D}X_t(y)}_{n+1,q}^{n+1}
		\end{equation*}
		 for any $x\in\mathbb{R}$, where $p\ge 1$ is arbitrary and $q\ge 1$ and $0<c_p<\infty$ depend only on $p$. The application of Hölder-inequality for the norms $\norm{}_{r,s}$ (Proposition 1.5.6 of \citet{nualartMalliavinCalculusRelated2006}) yields
		\begin{equation}
			\abs{\partial_x^n p_{f,\lambda}(t,y,x)}\le \tilde{c}_p \norm{\gamma_{t,y}^{-1}}_{n+1,r}^{n+1}\norm{\mathcal{D}X_t(y)}_{n+1,s}^{n+1}\label{eq:aux_density_deriv}
		\end{equation}
		uniformly in $x\in\mathbb{R}$, where $1/q = 1/r+1/s$ and $0<\tilde{c}_p<\infty$ only depends on $p,r$ and $s$, compare page 102 of \citet{nualartMalliavinCalculusRelated2006}.\\
		\textbf{Step 1 (Controlling $\norm{\mathcal{D}X_t(y)}_{n,s}^n$).}
		We claim that for each $0<t\le T$ we can bound
		\begin{align}
			\begin{split}
				\sup_{\lambda\ge 1,y\in\Lambda}\norm{\mathcal{D} X_t(y)}_{\mathfrak{H}}& \le C_1t^{1/4},                                  \\
				\sup_{\lambda\ge 1, y\in\Lambda}\norm{\mathcal{D}^n X_t(y)}_{\mathfrak{H}^{\otimes n}} & \le C_n t^{1+n/4},\quad n=2,\dots, N+2,
			\end{split}\label{eq:Deriv_Momentbounds}
		\end{align}
		almost surely, where the constants $0<C_n<\infty$ depend on $n$, $K$ and $T$, but not on the domain size $\lambda\ge 1$, the time point $0<t\le T$ or the location $y\in\Lambda$. The deterministic bounds \eqref{eq:Deriv_Momentbounds} imply
		\begin{equation}
			\norm{\mathcal{D}X_t(y)}_{n,s}\lesssim _{n,T,K}t^{1/4},\label{eq:aux_DerivNorm}
		\end{equation}
		uniformly in $s\ge 1$, $0<t\le T$ and $y\in\Lambda$. We prove \eqref{eq:Deriv_Momentbounds} by induction over $n\in\Set*{1,\dots,N+2}$, similarly to Proposition 4.3 of \citet{ballyMalliavinCalculusWhite1998}.\\
		\textbf{Base case ($n=1$).} For $n=1$ and $0<t\le T$ we apply the expression \eqref{eq:malliavin_derivative} for the Malliavin derivative and Lemma \ref{lem:aux_GreenfunctionsGrowth} to obtain
		\begin{align*}
			\sup_{y\in\Lambda}\norm{\mathcal{D}X_t(y)}_{\mathfrak{H}} & \le \norm{G_{t-\MTemptyplaceholder}^\lambda(y,\MTemptyplaceholder)}_{\mathfrak{H}} + \norm{f'}_{L^\infty(\mathbb{R})}\int_0^t \int_{\Lambda} G_{t-s}^\lambda(y,\eta)\sup_{z\in\Lambda}\norm{\mathcal{D}X_s(z)}_{\mathfrak{H}}\,\D \eta\D s \\
            & \le \bar{C}t^{1/4} + \norm{f'}_{L^\infty(\mathbb{R})}  \int_0^t \sup_{z\in\Lambda}\norm{\mathcal{D}X_s(z)}_{\mathfrak{H}}\,\D s.
		\end{align*}
		Applying Gronwall's inequality yields the (deterministic) bound
		\begin{equation}
			\sup_{y\in\Lambda}\norm{\mathcal{D}X_t(y)}_{\mathfrak{H}} \le \bar{C}t^{1/4} e^{\norm{f'}_{L^\infty(\mathbb{R})}  t} =\colon C_1(t),\quad 0<t\le T,\label{eq:C_1}
		\end{equation}
		uniformly in $\lambda \ge 1$.\\
		\textbf{Intermezzo.} For the reader's convenience, we treat the case $n=2$ separately before continuing the induction. We have for all times $0< t\le T$ and locations $y\in\Lambda$ the equality
		\begin{equation*}
			\mathcal{D}^2X_t(y)=\begin{multlined}[t]
				\int_0^t \int_{\Lambda}  G_{t-s}^\lambda(y,\eta)f''(X_s(\eta))\mathcal{D}X_s(\eta)\otimes\mathcal{D} X_s(\eta)\,\D\eta \D s\\
				+ \int_0^t \int_{\Lambda} G_{t-s}^\lambda(y,\eta)f'(X_s(\eta))\mathcal{D}^2X_s(\eta)\,\D \eta\D s
			\end{multlined}
		\end{equation*}
		and thus
		\begin{align*}
			\sup_{y\in\Lambda}\norm{\mathcal{D}^2X_t(y)}_{\mathfrak{H}^{\otimes 2}} & \le  \begin{multlined}[t]
            \sup_{y\in\Lambda}\norm{f''}_{L^\infty(\mathbb{R})}\int_0^t \int_{\Lambda}  G_{t-s}^\lambda(y,\eta)\norm{\mathcal{D}X_s(\eta)}_{\mathfrak{H}}^2\,\D\eta \D s\\
            + \sup_{y\in\Lambda}\norm{f'}_{L^\infty(\mathbb{R})}\int_0^t \int_{\Lambda} G_{t-s}^\lambda(y,\eta)\norm{\mathcal{D}^2X_s(\eta)}_{\mathfrak{H}^{\otimes 2}}\,\D \eta\D s
            \end{multlined}\\
            & \le \begin{multlined}[t]
            \sup_{y\in\Lambda}\norm{f''}_{L^\infty(\mathbb{R})}\int_0^t\int_{\Lambda} G_{t-s}^\lambda(y,\eta)C_1(s)^2\,\D\eta\D s \\
            + \norm{f'}_{L^\infty(\mathbb{R})} \sup_{y\in\Lambda}\int_0^t\sup_{z\in\Lambda}\norm{\mathcal{D}^2X_s(z)}_{\mathfrak{H}^{\otimes 2}}\int_{\Lambda} G_{t-s}^\lambda(y,\eta)\,\D\eta\D s
		      \end{multlined} \\
              & \le \norm{f''}_{L^\infty(\mathbb{R})}  C_1(t)^2t + \norm{f'}_{L^\infty(\mathbb{R})} \int_0^t \sup_{z\in\Lambda}\norm{\mathcal{D}^2X_s(z)}_{\mathfrak{H}^{\otimes 2}}\,\D s
		\end{align*}
		for all $0<t\le T$. An application of Gronwall's inequality yields the bound
		\begin{equation}
			\sup_{y\in\Lambda}\norm{\mathcal{D}^2X_t(y)}_{\mathfrak{H}^{\otimes 2}} \le \norm{f''}_{L^\infty(\mathbb{R})} C_1(t)^2t e^{\norm{f'}_{L^\infty(\mathbb{R})} t}=\colon C_2(t),\quad 0<t\le T,\label{eq:Bayes_aux_C2}
		\end{equation}
		uniformly in $\lambda\ge 1$.\\
		\textbf{Induction step ($n-1\leadsto n$, $n\le N+2$).} For the induction step we use the chain rule for Malliavin derivatives in the form of (3.20) of \citep{SanzSole2005}. It implies
		\begin{equation*}
			\mathcal{D}^nX_t(y)=\begin{multlined}[t]
				\int_0^t \int_{\Lambda}  G_{t-s}^\lambda(y,\eta)\sum_{l=2}^{n}\sum_{P_l} c_l f^{(l)}(X_s(\eta))\bigotimes_{i=1}^l\mathcal{D}^{\abs{p_i}}X_s(\eta)\,\D\eta \D s\\
				+ \int_0^t \int_{\Lambda} G_{t-s}^\lambda(y,\eta)f'(X_s(\eta))\mathcal{D}^n X_s(\eta)\,\D \eta\D s,
			\end{multlined}
		\end{equation*}
		where $P_l$ denotes the set of partitions of $\Set{1,\dots,n}$ consisting of $l$ disjoint sets $p_1,\dots,p_l$ and $c_l>0$. In addition to $C_1(t)$ defined in \eqref{eq:C_1}, we define the quantities
		\begin{equation*}
			C_n(t)\coloneq \sum_{l=2}^{n}\sum_{P_l} c_l \norm{f^{(l)}}_{L^\infty(\mathbb{R})}\prod_{i=1}^lC_{\abs{p_i}}(t) t e^{\norm{f'}_{L^\infty(\mathbb{R})}  t},\quad n\in\mathbb{N}, n\ge 2,\quad 0<t\le T.
		\end{equation*}
		Note that $C_2(t)$ agrees with \eqref{eq:Bayes_aux_C2} and that for any $n\in\mathbb{N}$, $0<s\le t\le T$, the monotonicity $C_n(s)\le C_n(t)$ holds. Furthermore, for $n\ge 2$ and $0<t\le T$ holds $C_n(t)\lesssim_{K,T} t^{1+n/4}$. Then we find
		\begin{align*}
			\sup_{y\in\Lambda}\norm{\mathcal{D}^n X_t(y)}_{\mathfrak{H}^{\otimes n}} & \le
			\sup_{y\in\Lambda}\int_0^t \int_{\Lambda}  G_{t-s}^\lambda(y,\eta)\sum_{l=2}^{n}\sum_{P_l} c_l \norm{f^{(l)}}_{L^\infty(\mathbb{R})}\prod_{i=1}^l\norm{\mathcal{D}^{\abs{p_i}}X_s(\eta)}_{\mathfrak{H}^{\otimes \abs{p_i}}}\,\D\eta \D s\\
            & \quad + \sup_{y\in\Lambda}\norm{f'}_{L^\infty(\mathbb{R})}\int_0^t \int_{\Lambda} G_{t-s}^\lambda(y,\eta)\norm{\mathcal{D}^nX_s(\eta)}_{\mathfrak{H}^{\otimes n}}\,\D \eta\D s\\
            & \le
			\sup_{y\in\Lambda}\int_0^t \int_{\Lambda}  G_{t-s}^\lambda(y,\eta)\sum_{l=2}^{n}\sum_{P_l} c_l \norm{f^{(l)}}_{L^\infty(\mathbb{R})}\prod_{i=1}^lC_{\abs{p_i}}(s)\,\D\eta\D s\\
            & \quad+ \norm{f'}_{L^\infty(\mathbb{R})} \sup_{y\in\Lambda}\int_0^t\sup_{z\in\Lambda}\norm{\mathcal{D}^nX_s(z)}_{\mathfrak{H}^{\otimes n}}\int_{\Lambda} G_{t-s}^\lambda(y,\eta)\,\D\eta\D s
			\\
            & \le \sum_{l=2}^{n}\sum_{P_l} c_l \norm{f^{(l)}}_{L^\infty(\mathbb{R})} \prod_{i=1}^lC_{\abs{p_i}}(t)t+ \norm{f'}_{L^\infty(\mathbb{R})} \int_0^t \sup_{z\in\Lambda}\norm{\mathcal{D}^nX_s(z)}_{\mathfrak{H}^{\otimes n}}\,\D s.
		\end{align*}
		An application of Gronwall's inequality shows
		\begin{equation*}
			\sup_{y\in\Lambda}\norm{\mathcal{D}^nX_t(y)}_{\mathfrak{H}^{\otimes n}} \le \sum_{l=2}^{n}\sum_{P_l} c_l \norm{f^{(l)}}_{L^\infty(\mathbb{R})}\prod_{i=1}^lC_{\abs{p_i}}(t) t e^{\norm{f'}_{L^\infty(\mathbb{R})}  t}= C_n(t),\quad 0<t\le T,
		\end{equation*}
		uniformly in the domain size $\lambda\ge 1$ and $n\in\mathbb{N}$.\\
		\textbf{Step 2 (Controlling $\norm{\gamma_{t,y}^{-1}}_{n,r}^n$).}\\
		The upper bound for $n=1$ from \eqref{eq:Deriv_Momentbounds} and the lower bound from the proof of Lemma C.1 of \citet{gaudlitzNonparametricEstimationReaction2023} show
		\begin{equation*}
			\gamma_{t,y} = \norm{\mathcal{D}X_t(y)}_{\mathfrak{H}}^2\sim t^{1/2},
		\end{equation*}
		uniformly in $\lambda\ge 1$, $0< t\le T$ and $y\in\Lambda$. The recursion (2.24) of \citep{nualartMalliavinCalculusRelated2006} together with \eqref{eq:Deriv_Momentbounds} reveal that $\norm{\gamma_{t,y}^{-1}}_{n,r}^r$, $r\ge 1$, is dominated by its zeroth component when no derivative is taken, such that
		\begin{equation*}
			\norm{\gamma_{t,y}^{-1}}_{n,r}^r \lesssim_T \EV[f]{\abs{\gamma_{t,y}^{-1}}^r}\lesssim t^{-r/2},\quad 0<t\le T,\quad y\in\Lambda,
		\end{equation*}
		and thus
		\begin{equation}
			\norm{\gamma_{t,y}^{-1}}_{n,r}^{n}\lesssim_T t^{-n/2}\label{eq:aux_gamma_bounds}
		\end{equation}
		uniformly in $\lambda\ge 1$, $0< t\le T$ and $y\in\Lambda$.\\
		\textbf{Step 3 (Conclusion).}\\
		The claim follows from \eqref{eq:aux_density_deriv}, \eqref{eq:aux_DerivNorm} and \eqref{eq:aux_gamma_bounds}: There exists a constant $0<\tilde{C}<\infty$ depending on $T$ and $K$ such that for all $\lambda\ge 1$, $0<t\le T$, $y\in\Lambda$ and $x\in\mathbb{R}$ the bound
		\begin{equation*}
			\abs{\partial_x^n p_{f,\lambda}(t,y,x)} \le \tilde{C} t^{-(n+1)/2}t^{(n+1)/4} = \tilde{C}t^{-(n+1)/4}
		\end{equation*}
		holds.
	\end{proof}

	\subsection{Convergence to the global solution}\label{chap:Convergence}

	This section explores the behaviour of the semi-linear stochastic heat equation driven by space-time white noise on the growing spatial domain $\Lambda=\lambda\bar{\Lambda}$ with $\lambda\to\infty$. Without loss of generality we assume $0\in \bar{\Lambda}$ such that $\Lambda = \lambda\bar{\Lambda}\to \mathbb{R}$ as $\lambda\to\infty$. The goal is to show the weak convergence $X_t(\lambda \bar{y})\xrightarrow{d}Z_t^{\bar{y}}(0)$ as $\lambda\to\infty$ for every $0\le t\le T$ and $\bar{y}\in\bar{\Lambda}$ and proving Theorem \ref{thm:ErgodicProperties} \eqref{num:thm_Ergodicity1}. Denote by $\star$ the operation
	\begin{equation*}
		(u\star v)(y)\coloneq \int_{\mathbb{R}}u(y,\eta)v(\eta)\,\D\eta,\quad y\in\mathbb{R},
	\end{equation*}
	for any functions $u\colon \mathbb{R}^2\to\mathbb{R}$, $v\colon\mathbb{R}\to\mathbb{R}$ such that $u(y,\MTemptyplaceholder)v(\MTemptyplaceholder)\in L^1(\mathbb{R})$ for all $y\in\mathbb{R}$.
	 Let $\phi_t$ be the heat kernel on $\mathbb{R}$, that is
	\begin{equation*}
		\phi_t(y,\eta) = \frac{1}{\sqrt{4\pi t}}\exp\left(-\frac{(y-\eta)^2}{4t}\right),\quad t>0,\quad y,\eta\in\mathbb{R}.
	\end{equation*}
	Recall that the process $(Z_t^{\bar{y}}(y)\,\vert \, y\in\mathbb{R}, 0\le t\le T)$ solving the semi-linear heat equation \eqref{eq:Z} can be written as
	\begin{equation}
		Z_t^{\bar{y}}(y)= \int_0^t (\phi_{t-s}\star f(Z_s^{\bar{y}}))(y)\,\D s + \bar{Z}_t^{\bar{y}}(y),\quad  \bar{Z}_t^{\bar{y}}(y)\coloneq \chi(\bar{y})+\int_{\mathbb{R}}\phi_{t-s}(y,\eta)\mathcal{W}(\D\eta,\D s),\label{eq:ConvGlobal_Z}
	\end{equation}
	where $\mathcal{W}$ is understood as spatially homogeneous (space-time white) noise $\mathcal{W}$ on $\mathbb{R}$, compare \citep{Dalang2011}.
     Denote by $G_t^\lambda(y,\eta)$ for $t>0$ and $y,\eta\in\Lambda$ the Neumann heat kernel on $\Lambda$. Without further mentioning we will extend $G_t^\lambda$ by zero to a function on $\mathbb{R}\times\mathbb{R}$. Define the process
	\begin{align}
    \begin{split}
    \bar{X}_t(y)&\coloneq (G_{t}^\lambda \star X_0)(y) +  \int_{\mathbb{R}}G_{t-s}^\lambda(y,\eta)\mathcal{W}(\D\eta,\D s),\\
		X_t(y)&\coloneq \int_0^t (G_{t-s}^\lambda\star f(X_s))(y)\,\D s + \bar{X}_t(y),
    \end{split}\label{eq:ConvGlobal_X}
	\end{align}
	for $0\le t\le T, y\in \mathbb{R}$, driven by the same noise $\mathcal{W}$ as $Z_t^{\bar{y}}$. Note that the process $(X_t)_{t\in[0,T]}$ defined in \eqref{eq:ConvGlobal_X} and the process $(X_t)_{t\in[0,T]}$ defined in \eqref{eq:SPDE} have the same marginal laws for every $0\le t\le T$ and $y\in\Lambda$.
	 Since we are only aiming to prove convergence in distribution $X_t(\lambda\bar{y})\xrightarrow{d}Z_t^{\bar{y}}(0)$ as $\lambda\to\infty$ for every $0\le t\le T$ and $\bar{y}\in\bar{\Lambda}$, we keep the notation $X$ for both processes for notational simplicity.

	For a standard Brownian motion $(B_t)_{t\ge 0}$ define the stopping time
	\begin{equation*}
		\tau_\lambda\coloneq \inf\Set{t\ge 0\,\colon\,B_{2t}\not\in \lambda\bar{\Lambda}}.
	\end{equation*}
	For any location $y\in\Lambda$ and time $t\ge 0$, the probability of the event $\Set{\tau_\lambda\le t}$, given that $B_0=y$, is denoted by $\prob[y]{\tau_\lambda\le t}$. Lemma \ref{lem:Stopping_time} implies $\prob[\lambda\bar{y}]{\tau_\lambda\le t}\to 0$ as $\lambda\to \infty$ for every $\bar{y}\in\bar{\Lambda}$ and $t>0$. Recall that for $A\subset\mathbb{R}$ and $y\in\mathbb{R}$ we define their distance as $\operatorname{dist}(y,A)\coloneq \inf\Set{\abs{y-\eta}\given\eta\in A}$. The main result of this section is the following (pointwise) convergence of $X_t(\lambda \bar{y})$ to $Z_t^{\bar{y}}(\lambda \bar{y})$.

	\begin{proposition}\label{prop:Convergencetoglobal}
		Assume that $f\colon\mathbb{R}\to\mathbb{R}$ is Lipschitz-continuous and that $\chi\in H^{\tilde{s}}(\bar{\Lambda})$ for some $\tilde{s}>1/2$. For any $p\ge 1$ we have convergence
		\begin{equation*}
			\EV[f]{\abs{Z_t^{\bar{y}}(\lambda \bar{y}) - X_t(\lambda \bar{y})}^p}\to 0,\quad \bar{y}\in\bar{\Lambda},\quad  0\le t\le T,
		\end{equation*}
		as $\lambda\to \infty$ for the process $X$ defined in \eqref{eq:ConvGlobal_X}. In particular, the process $X$ defined in \eqref{eq:SPDE} satisfies $X_t(\lambda \bar{y})\xrightarrow{d}Z_t^{\bar{y}}(0)$ for all $0\le t\le T$ and $\bar{y}\in\bar{\Lambda}$. More precisely, for any exponent $p\ge 1$ and location $y\in\Lambda$ define
		\begin{equation*}
			K(y,\lambda,p)\coloneq \exp\left(-\frac{\operatorname{dist}(y,\partial\Lambda)^2}{8 c T p}\right)\begin{cases}(1 + \log(\lambda) + \operatorname{dist}(y,\partial\Lambda)^2),&p=2\\ 2,& p\neq 2,\end{cases}
		\end{equation*}
		where $c>1$ is arbitrary. Then, for all $0\le t\le T$, $y=\lambda\bar{y} \in \Lambda$ and $0<\epsilon<\tilde{s}-1/2$,
		\begin{align}
			\EV[f]{\abs{Z_t^{\bar{y}}(y)-X_t(y)}^p}\le C \left(\prob[y]{\tau_\lambda\le t}^{p/2}+ K(y,\lambda,p)+K(y,\lambda,2p)+\lambda^{-p\epsilon}\norm{\chi}_{H^{\tilde{s}}(\bar{\Lambda})}^p\right),\label{eq:ConvergenceGlobal_Claim}
		\end{align}
		where the constant $C$ depends only on $T$, $\Lip{f}$ and $p$.
	\end{proposition}
	Before proving Proposition \ref{prop:Convergencetoglobal}, a few remarks are in order.
	\begin{remark}[On $K(y,\lambda,p)$]~
		\begin{enumerate}[(a)]
		\item For any location $\bar{y}\in\bar{\Lambda}$, we have $K(\lambda \bar{y},\lambda,p)\sim \exp(-\lambda^2/(8cTP))\to 0$ as $\lambda\to \infty$. This convergence is not uniform in $\bar{y}\in\bar{\Lambda}$ due to the effect of the boundary conditions.
				\item By integration we find for all $1\le \lambda<\infty$ and $r>0$ the bound
				\begin{equation}
					\frac{1}{\lambda}\int_\Lambda K(y,\lambda,p)^r\,\D y\lesssim_{T,p} \lambda^{-1}\begin{cases}\log(\lambda)^r,&p=2,\\ 1,&p\neq 2.\end{cases}\label{eq:Greenfunction_IntegratedProb2}
				\end{equation}
		\end{enumerate}
	\end{remark}

	\begin{corollary}\label{cor:Convergence_polynomials}
		Assume that $f\colon\mathbb{R}\to\mathbb{R}$ is Lipschitz-continuous and that $\chi\in H^{\tilde{s}}(\bar{\Lambda})$ for some $\tilde{s}>3/2$. Let $(X_t)_{t\in[0,T]}$ be either the process defined in \eqref{eq:SPDE} or in \eqref{eq:ConvGlobal_X}. For any Lipschitz-continuous function $h\colon\mathbb{R}\to\mathbb{R}$ and $k\in\mathbb{N}$ we have
		\begin{equation*}
			\abs[\bigg]{\EV[f][][\Big]{\frac{1}{\lambda}\int_0^T\int_\Lambda h(X_t(y))^k\,\D y\D t}-\int_0^T\int_{\bar{\Lambda}}\EV[f][][\Big]{h(Z_t^{\bar{y}}(0))^k}\,\D\bar{y}\D t}\le C\Lip{h} \lambda^{-1},
		\end{equation*}
        where the constant $0<C<\infty$ depends on $T$, $\Lip{f}$, $k$ and $\norm{\chi}_{H^{\tilde{s}}(\bar{\Lambda})}$.
	\end{corollary}
	\begin{proof}
		Note that $Z_t^{\bar{y}}(0) \overset{d}{=} Z_t^{\bar{y}}(y)$ by the stationarity of the process $y\mapsto Z_t^{\bar{y}}(y)$. This implies
		\begin{align*}
			\abs[\bigg]{\EV[f][][\Big]{\frac{1}{\lambda}\int_0^T\int_\Lambda h(X_t(y))^k\,\D y\D t}-\int_0^T\int_{\bar{\Lambda}}\EV[f][][\Big]{h(Z_t^{\bar{y}}(0))^k}\,\D\bar{y}\D t}\hspace{-18em} &\\
			& =  \frac{1}{\lambda}\abs[\bigg]{\EV[f][][\Big]{\int_0^T\int_\Lambda h(X_t(y))^k\,\D y\D t}-\EV[f][][\Big]{\int_0^T\int_{\Lambda} h(Z_t^{y/\lambda}(y))^k\,\D y\D t}}\\
			& \le \frac{1}{\lambda}\int_0^T\int_\Lambda \EV[f][][\Big]{\abs[\Big]{\left(h(X_t(y))- h(Z_t^{y/\lambda}(y))\right)\Big(\sum_{i=1}^k h(X_t(y))^{k-i}h\Big(Z_t^{y/\lambda}(y)\Big)^{i-1}\Big)}}\,\D y\D t.
            \end{align*}
            Applying Hölder's inequality with $r>1$, $r\neq 2$ and $1/r+1/\tilde{r}=1$ we find
        \begin{align*}
			\abs[\bigg]{\EV[f][][\Big]{\frac{1}{\lambda}\int_0^T\int_\Lambda h(X_t(y))^k\,\D y\D t}-\int_0^T\int_{\bar{\Lambda}}\EV[f][][\Big]{h(Z_t^{\bar{y}}(0))^k}\,\D\bar{y}\D t}\hspace{-21em} &\\
            & \le \frac{1}{\lambda}\int_0^T\int_\Lambda \Big(\EV[f][][\Big]{\left(h(X_t(y))- h(Z_t^{y/\lambda}(y)\right)^r}^{1/r}\EV[f][][\Big]{\Big(\sum_{i=1}^k h(X_t(y))^{k-i}h(Z_t^{y/\lambda}(y))^{i-1}\Big)^{\tilde{r}}}^{1/{\tilde{r}}}\,\D y\D t.
		\end{align*}
		By Lemma \ref{lem:UniformBoudnednessSolution} and Theorem 25 of \citet{Dalang2011}, we can deduce for all $n\in\mathbb{N}$ the bounds
		\begin{equation*}
			\sup_{\lambda\ge 1 ,y\in \Lambda, 0\le t\le T}\EV[f]{\abs{X_t(y)}^n}\le c,\quad\sup_{y\in \mathbb{R}, 0\le t\le T,\bar{y}\in\bar{\Lambda}}\EV[f]{\abs{Z_t^{\bar{y}}(y)}^n}\le \tilde{c},
		\end{equation*}
        where the constants $0<c,\tilde{c}<\infty$ depend on $T$, $\Lip{f}$, $n$ and $\norm{\chi}_{L^\infty(\bar{\Lambda})}$. Using the Lipschitz-continuity of $h$, Proposition \ref{prop:Convergencetoglobal} with $p=r\neq 2$ and finally the bounds \eqref{eq:Greenfunction_IntegratedProb2} and \eqref{eq:Greenfunction_IntegratedProb1} we find
		\begin{align*}
			\abs[\bigg]{\EV[f][][\Big]{\frac{1}{\lambda}\int_0^T\int_\Lambda h(X_t(y))^k\,\D y\D t}-\int_0^T\int_{\bar{\Lambda}}\EV[f][][\Big]{h(Z_t^{\bar{y}}(0))^k}\,\D\bar{y}\D t}\hspace{-13em} &\\
			  & \lesssim_{T,\Lip{f},k,\norm{\chi}_{L^\infty(\bar{\Lambda})}}\Lip{h}\frac{1}{\lambda}\int_0^T\int_{\Lambda} \EV*[f]{\abs[\Big]{X_t(y)-Z_t^{y/\lambda}(y)}^r}^{1/r}\,\D y\D t \\
			& \lesssim_{T,\Lip{f},r}\Lip{h}\frac{1}{\lambda}\int_0^T\int_{\Lambda} \sqrt{\prob[y]{\tau_\lambda\le t}} + K(y,\lambda,r)^{1/r}+K(y,\lambda,2r)^{1/r}\\
            &\quad +\lambda^{-\epsilon}\norm{\chi}_{H^{\tilde{s}}(\bar{\Lambda})}\,\D y\D t \\
			  & \lesssim_{T}\Lip{h} \left(\lambda^{-1}+\lambda^{-1}+\lambda^{-1} + \lambda^{-\epsilon}\norm{\chi}_{H^{\tilde{s}}(\bar{\Lambda})}\right)\\
              & \lesssim_{\norm{\chi}_{H^{\tilde{s}}(\bar{\Lambda})}} \Lip{h} \lambda^{-1},
		\end{align*}
        where we chose $\epsilon\ge 1$ in the last line.
	\end{proof}

	\subsubsection{Proving Proposition \ref{prop:Convergencetoglobal}}

	The proof of the bound \eqref{eq:ConvergenceGlobal_Claim} is based on two ingredients. First, it requires the convergence of the stochastic integrals $\bar{X}$ to $\bar{Z}^{\bar{y}}$ in \eqref{eq:ConvGlobal_Z} and \eqref{eq:ConvGlobal_X}, which is proved using the convergence of the corresponding Green functions. The impact of the varying initial conditions can be controlled by understanding the action of the heat semi-group $t\mapsto G_t\star\MTemptyplaceholder$ for short times, as the subsequent Lemma shows. To this end, a convenient tool is the scaling property of the Neumann heat kernel $G_t^\lambda$ on $\Lambda$ in terms of the Neumann heat kernel $G_t^1$ on $\bar{\Lambda}$ such that
    \begin{equation}
        G_t^\lambda(y,\eta) = \lambda^{-1} G_{\lambda^{-2}t}^1(\lambda^{-1}y,\lambda^{-1}\eta),\quad t>0, y,\eta\in\Lambda.\label{eq:Scaling_Heat_Kernel}
    \end{equation}
    Second, we perform a Gronwall argument to extend the convergence to the semi-linear case.

    \begin{lemma}\label{lem:Impact_Initial_Condition}
    Assume that $\chi\in H^{\tilde{s}}(\bar{\Lambda})$ for some $\tilde{s}>1/2$ and let $0<\epsilon< \tilde{s}-1/2$. Then there exists an absolute constant $0<C<\infty$ such that
    \begin{equation*}
        \norm{X_0 - G_t^\lambda \star X_0}_{L^\infty(\Lambda)}=\norm{\chi-G_{\lambda^{-2}t}^1\star \chi}_{L^\infty(\bar{\Lambda})}\le C t^{\epsilon/2}\lambda^{-\epsilon}\norm{\chi}_{H^{\tilde{s}}(\bar{\Lambda})},\quad 0<t\le T.
    \end{equation*}
\end{lemma}
\begin{proof}
Using the scaling of the Green kernel from \eqref{eq:Scaling_Heat_Kernel}, we deduce for every $y\in\Lambda$
\begin{align*}
  (G_t^\lambda\star X_0)(y) &=\int_\Lambda G_t^\lambda (y,\eta)X_0(\eta)\,\D \eta =\lambda\int_{\bar{\Lambda}}G_t^\lambda (y,\lambda\bar{\eta})\chi(\bar{\eta})\,\D\bar{\eta}= \int_{\bar{\Lambda}}G_{\lambda^{-2}t}^1 (y/\lambda,\bar{\eta})\chi(\bar{\eta})\,\D\bar{\eta}\\
  &= (G_{\lambda^{-2}t}^1 \star \chi)(y/\lambda),
\end{align*}
which shows the claimed equality. Since $\tilde{s}-\epsilon>1/2$, the Sobolev embedding (Theorem 4.12 of \citet{fournierSobolevSpaces2003}) yields
\begin{equation*}
    \norm{\chi-G_{\lambda^{-2}t}^1\star \chi}_{L^\infty(\bar{\Lambda})}\lesssim \norm{\chi-G_{\lambda^{-2}t}^1\star \chi}_{H^{\tilde{s}-\epsilon}(\bar{\Lambda})}.
\end{equation*}
 By Theorem 16.6 of \citet{Yagi2010}, the interpolation norm $\norm{}_{\gamma}\coloneq \norm{(-\Updelta)^{\gamma}\MTemptyplaceholder}_{L^2(\bar{\Lambda})}$ of the Neumann Laplacian on $\bar{\Lambda}$ is equivalent to the Sobolev norm $\norm{}_{H^{2\gamma}(\bar{\Lambda})}$, $0\le\gamma\le 1$. Since $\chi\in \operatorname{dom}(-\Updelta)^{(\tilde{s}-\epsilon)/2}$ and powers of $(-\Updelta)$ commute with its semi-group $t\mapsto G_t^1\star \MTemptyplaceholder$, we find
\begin{equation*}
    \norm{\chi-G_{\lambda^{-2}t}^1\star \chi}_{L^\infty(\bar{\Lambda})}\lesssim \norm{(-\Updelta)^{(\tilde{s}-\epsilon)/2}\chi - G_{\lambda^{-2}t}^1\star (-\Updelta)^{(\tilde{s}-\epsilon)/2}\chi}_{L^2(\bar{\Lambda})}.
\end{equation*}
Since the Neumann Laplacian generates an analytic semi-group $t\mapsto G_t^1\star \MTemptyplaceholder$ on $L^2(\bar{\Lambda})$, we can apply Theorem 6.13 (d) of \citet{Pazy1983} to deduce
\begin{equation*}
    \norm{\chi-G_{\lambda^{-2}t}^1\star \chi}_{L^\infty(\bar{\Lambda})}\lesssim (\lambda^{-2}t)^{\epsilon/2}\norm{(-\Updelta)^{\tilde{s}/2}\chi}_{\epsilon/2}\lesssim \lambda^{-\epsilon}t^{\epsilon/2}\norm{\chi}_{H^{\tilde{s}}(\bar{\Lambda})}.
\end{equation*}\qedhere
\end{proof}

    \textbf{Step 1: Preliminaries on Green functions.}

		In the following, we will make frequent use of the fact that the Green functions $\phi_t$ and $G_t^\lambda$ can be represented as transition probabilities of Markov processes.
            Since the Neumann heat kernel $G_t^\lambda$ can be interpreted as the transition probabilities of a Brownian motion $(B_{2t}^\leftrightarrow)$, which is reflected when hitting the boundary of $\Lambda$ (see Section 2 of \citet{kendallCoupledBrownianMotions1989}), it is natural to expect that $G_t^{\lambda_1}(y,\eta)\le G_t^{\lambda_2}(y,\eta)$ whenever $\lambda_1\ge\lambda_2$ since then $\lambda_1 \bar{\Lambda}\supset \lambda_2\bar{\Lambda}$. In contrast to the Dirichlet heat kernel, however, the domain monotonicity property for Neumann boundary conditions is not true in general and counter examples are given by \citet{bassDomainMonotonicityNeumann1993}. \citet{kendallCoupledBrownianMotions1989} proves that
			      \begin{equation}
				      G_t^{\lambda_1}(y,\eta)\le G_t^{\lambda_2}(y,\eta),\quad 0< t\le T,\quad  y,\eta\in \lambda_2 \bar{\Lambda},\quad \lambda_1\ge \lambda_2,\label{eq:Greenfunction_Monotonicity_neumann}
			      \end{equation}
			      whenever there exists a ball $B$ between $\lambda_1 \bar{\Lambda}$ and $\lambda_2 \bar{\Lambda}$ such that $\lambda_2 \bar{\Lambda}\subset B\subset \lambda_1 \bar{\Lambda}$. Since $0\in\bar{\Lambda}$, we find $\Lambda\uparrow\mathbb{R}$ as $\lambda\to \infty$. Consequently, \eqref{eq:Greenfunction_Monotonicity_neumann} holds for $\lambda_1\gg \lambda_2$ sufficiently large.

	\begin{lemma}\label{lem:DomainDifferences}
		Let $1\le\lambda<\infty$, $0<t\le T$, $y\in\Lambda$. Then
		\begin{align}
			\int_{\mathbb{R}}\abs{\phi_{t}(y,\eta) - G_{t}^\lambda(y,\eta)}\,\D\eta & \le 2\prob[y]{\tau_\lambda\le t},\label{eq:domainDifferences_Neu_1}                      \\
			\int_{\mathbb{R}}(\phi_{t}(y,\eta) - G_{t}^\lambda(y,\eta))^2\,\D\eta   & \le C t^{-1/2}\prob[y]{\tau_\lambda\le t}, \label{eq:domainDifferences_Neu_2}
		\end{align}
		where the constant $0<C<\infty$ depends only on $T$.
	\end{lemma}
	\begin{proof}
        The domain monotonicity \eqref{eq:Greenfunction_Monotonicity_neumann} implies for any Borel set $A\subset \mathbb{R}$, location $y\in \Lambda$ and time $0< t\le T$ the bound
		\begin{align*}
			\int_{A}\abs{G_t^\lambda(y,\eta) - \phi_t(y,\eta)}\,\D\eta &= \int_{A\cap \Lambda} (G_t^\lambda(y,\eta)-\phi_t(y,\eta))\,\D\eta + \int_{A\cap \Lambda^c}\phi_t(y,\eta)\,\D\eta\\
			&= \int_{A\cap \Lambda} G_t^\lambda(y,\eta)\,\D\eta -\int_{A\cap \Lambda}\phi_t(y,\eta)\,\D\eta + \int_{A\cap \Lambda^c}\phi_t(y,\eta)\,\D\eta.
		\end{align*}
		Recalling that $\phi_t$ and $G_t^\lambda$ are the transition kernels of $B_t$ and $B_{2t}^\leftrightarrow$, respectively, we find
		\begin{align*}
			\int_{A}\abs{G_t^\lambda(y,\eta) - \phi_t(y,\eta)}\,\D\eta & = \prob[y]{B_{2t}^\leftrightarrow\in A\cap \Lambda} - \prob[y]{B_{2t}\in A\cap \Lambda} + \prob[y]{B_{2t}\in A\cap \Lambda^c} \\
			  & = \prob[y]{B_{2t}^\leftrightarrow\in A\cap \Lambda } - \prob[y]{B_{2t}\in A} + 2 \prob[y]{B_{2t}\in A\cap \Lambda^c}              \\
			  & \le \prob[y]{B_{2t}^\leftrightarrow\in A\cap \Lambda} - \prob[y]{B_{2t}\in A} + 2\prob[y]{\tau_\lambda\le t}.
		\end{align*}
		Choosing $A=\mathbb{R}$ yields \eqref{eq:domainDifferences_Neu_1}. Since $G_t^\lambda(y,\eta) =\lambda^{-1} G_{\lambda^{-2} t}^1(\lambda^{-1}y,\lambda^{-1}\eta)$ by \eqref{eq:Scaling_Heat_Kernel}, where the latter is the Green kernel on $\bar{\Lambda}$, we can apply Theorem 3.2.9 of \citet{daviesHeatKernelsSpectral1990} to obtain
		\begin{equation}
			G_t^\lambda(y,\eta) = \lambda^{-1} G_{\lambda^{-2} t}^1(\lambda^{-1}y,\lambda^{-1}\eta)\lesssim \lambda^{-1}((\lambda^{-2} t)^{-1/2}\vee 1)\lesssim_T t^{-1/2} ,\quad y,\eta\in\Lambda.
		\end{equation}
		This implies the bound
		\begin{align*}
			\int_{\mathbb{R}}(G_t^\lambda(y,\eta) - \phi_t(y,\eta))^2\,\D\eta & \le \int_{\mathbb{R}} \max(G_t^\lambda(y,\eta), \phi_t(y,\eta))\abs{G_t^\lambda(y,\eta)-\phi_t(y,\eta)}\,\D\eta \\
			& \lesssim_{T}  t^{-1/2} \int_{\mathbb{R}}\abs{G_t^\lambda(y,\eta) - \phi_t(y,\eta)}\,\D\eta                          \\
			 & \lesssim t^{-1/2}\prob[y]{\tau_\lambda\le T}.
		\end{align*}\qedhere
	\end{proof}
	The following Lemma quantifies the convergence $\prob[\lambda\bar{y}]{\tau_\lambda\le t}\to 0$ as $\lambda\to\infty$ for any fixed $\bar{y}\in\bar{\Lambda}$ and follows by classical arguments.
	\begin{lemma}\label{lem:Stopping_time}
		For all $y\in\Lambda$ and $t> 0$ we have
		\begin{equation}
			\prob[y]{\tau_\lambda \le t}\lesssim 1\wedge \frac{\sqrt{t}}{\operatorname{dist}(y,\partial\Lambda)}\exp\left(-\frac{\operatorname{dist}(y,\partial\Lambda)^2}{4t}\right)\label{eq:Greenfunction_StoppingTimeEstimate}
		\end{equation}
		and
		\begin{equation}
			\frac{1}{\lambda}\int_\Lambda \sqrt{\prob[y]{\tau_\lambda\le T}}\,\D y \lesssim_T \lambda^{-1}.\label{eq:Greenfunction_IntegratedProb1}
		\end{equation}
	\end{lemma}
	\begin{proof}
		Using the estimate (2.8.3') of \citet{karatzasBrownianMotionStochastic1998} we can bound
		\begin{align*}
			\prob[y]{\tau_\lambda \le t} & =\prob[y]{\Exists 0\le s\le t\given  B_{2s}\in\partial\Lambda}\le \prob[y]{\max_{s\in[0,t]}\abs{B_{2s}-y}\ge \operatorname{dist}(y,\partial\Lambda)}        \\
			   & = \prob[0]{\max_{s\in[0,t]}\abs{B_{2s}}\ge \operatorname{dist}(y,\partial \Lambda)} = \prob[0]{\max_{s\in[0,t]}\abs{B_s}\ge \operatorname{dist}(y,\partial \Lambda)/\sqrt{2}} \\
	           & \le \sqrt{\frac{t}{\pi}}\frac{4}{\operatorname{dist}(y,\partial \Lambda)}\exp\left(-\frac{\operatorname{dist}(y,\partial \Lambda)^2}{4t}\right)                               \\
			     & \lesssim \frac{\sqrt{t}}{\operatorname{dist}(y,\partial\Lambda)}\exp\left(-\frac{\operatorname{dist}(y,\partial\Lambda)^2}{4t}\right).
		\end{align*}
		The bound \eqref{eq:Greenfunction_IntegratedProb1} follows from taking the square root and integrating over $y\in\Lambda$.
	\end{proof}

    \textbf{Step 2: A Gronwall argument.}\\
    The subsequent lemma follows from the semi-group property of the heat kernel and is the basis for the upcoming Gronwall argument.
	\begin{lemma}\label{lem:Convolution}
		Let $n\in\mathbb{N}_0$ and let $g\colon [0,T]\times\mathbb{R}\to\mathbb{R}_{\ge 0}$ be measurable. Then
		\begin{equation*}
			\int_0^t \bigg(G_{t-s}^\lambda\star \int_0^s (s-r)^n G_{s-r}^\lambda\star g(r,\MTemptyplaceholder)\,\D r\bigg)(y)\,\D s=(n+1)^{-1}\int_0^t (t-r)^{n+1} (G_{t-r}^\lambda \star g(r,\MTemptyplaceholder))(y)\,\D r
		\end{equation*}
		for all $0\le t\le T$ and $y\in\mathbb{R}$.
	\end{lemma}
	\begin{proof}
		Using $G_{t_1}\star G_{t_2}=G_{t_1+t_2}$ for any $0< t_1,t_2<\infty$ we compute
		\begin{align*}
			\int_0^t \left(G_{t-s}^\lambda\star \int_0^s (s-r)^n G_{s-r}^\lambda\star g(r,\MTemptyplaceholder)\,\D r\right)(y)\,\D s & = \int_0^t \int_0^s (s-r)^n (G_{t-s}^\lambda\star (G_{s-r}^\lambda \star g(r,\MTemptyplaceholder)))(y)\,\D r\D s \\
			 & = \int_0^t\int_0^s (s-r)^n (G_{t-r}^\lambda \star g(r,\MTemptyplaceholder))(y)\,\D r\D s                     \\
	       & =\int_0^t (G_{t-r}^\lambda \star g(r,\MTemptyplaceholder))(y) \int_r^t (s-r)^n\,\D s\D r                     \\
	       & =(n+1)^{-1}\int_0^t (t-r)^{n+1} (G_{t-r}^\lambda \star g(r,\MTemptyplaceholder))(y)\,\D r.
		\end{align*}\qedhere
	\end{proof}
	We continue with a Gronwall argument based on Lemma \ref{lem:Convolution} and represent $Z_t(y)-X_t(y)$ in terms of differences of the Green functions $G_t^\lambda$ and $\phi_t$ in Lemma \ref{lem:Expansion/Gronwall}. Subsequently, Lemma \ref{lem:Aux_Convergence} bounds the terms individually using Lemmas \ref{lem:DomainDifferences} and \ref{lem:Stopping_time}.

	\begin{lemma}\label{lem:Expansion/Gronwall}
		Assume that $f\colon\mathbb{R}\to\mathbb{R}$ is Lipschitz-continuous. Let $y=\lambda\bar{y}\in\Lambda$ and $0\le t\le T$. We can bound
		\begin{align*}
			\abs{Z_t^{\bar{y}}(y)-X_t(y)}^p&\lesssim_{p,T,\Lip{f}}\abs{\bar{Z}^{\bar{y}}_t(y)-\bar{X}_t(y)}^p\\
            & \quad  +\prob[y]{\tau_\lambda\le t}^{p-1}\int_0^t(\abs{\phi_{t-s}-G_{t-s}^\lambda}\star\abs{f(Z_s^{\bar{y}})}^p)(y)\,\D s\\
            & \quad +\int_0^t (G_{t-s}^\lambda \star\abs{\bar{Z}_s^{\bar{y}}-\bar{X}_s}^p)(y)\,\D s\\
            & \quad +\int_0^t \left(G_{t-s}^\lambda \star\left(\prob[\MTemptyplaceholder]{\tau_\lambda\le t}^{p-1}\int_0^s \abs{\phi_{s-r}-G_{s-r}^\lambda}\star \abs{f(Z_r^{\bar{y}})}^p\,\D r\right)\right)(y)\,\D s,
		\end{align*}
		where the hidden constant in $\lesssim$ depends only on $p$, $T$, $\Lip{f}$ and $\norm{\chi}_{L^\infty(\bar{\Lambda})}$.
	\end{lemma}

	\begin{proof}
		\textbf{Step 1 ($p=1$ using a Gronwall expansion).} First expanding $X_t-Z_t^{\bar{y}}$ using \eqref{eq:ConvGlobal_Z} and \eqref{eq:ConvGlobal_X} into their linear and nonlinear parts, then using the triangle inequality for $\abs{\MTemptyplaceholder}$ and finally using the Lipschitz-continuity of $f$, we find
		\begin{align}
			\abs{Z_t^{\bar{y}}(y) -X_t(y)}\hspace{-0em} & \hspace{0em}\le \abs{\bar{Z}_t^{\bar{y}}(y)-\bar{X}_t(y)} + \abs[\bigg]{\int_0^t (\phi_{t-s}\star f(Z_s^{\bar{y}}))(y) - (G_{t-s}^\lambda\star f(X_s))(y)\,\D s}\nonumber \\
            & \le \abs[bigg]{\bar{Z}_t^{\bar{y}}(y)-\bar{X}_t(y)} + \abs[\bigg]{\int_0^t ((\phi_{t-s}-G_{t-s}^\lambda)\star f(Z_s^{\bar{y}}))(y)\,\D s} \nonumber  \\
            & \quad +\abs[\bigg]{\int_0^t (G_{t-s}^\lambda\star (f(X_s) - f(Z_s^{\bar{y}})))(y)\,\D s}\nonumber \\
            & \le\abs{\bar{Z}_t^{\bar{y}}(y)-\bar{X}_t(y)} + \int_0^t (\abs{\phi_{t-s}-G_{t-s}^\lambda}\star \abs{f(Z_s^{\bar{y}})})(y)\,\D s \nonumber\\
            & \quad +\Lip{f} \int_0^t (G_{t-s}^\lambda\star \abs{X_s - Z_s^{\bar{y}}})(y)\,\D s.\label{eq:Conv_aux_lastSummand}
		\end{align}
		Analogously expanding \eqref{eq:Conv_aux_lastSummand} yields the expression
		\begin{align*}
			\int_0^t (G_{t-s}^\lambda\star \abs{X_s - Z_s^{\bar{y}}})(y)\,\D s \hspace{-0em} & \hspace{0em}\le  \int_0^t (G_{t-s}^\lambda\star \abs{\bar{Z}_s^{\bar{y}}-\bar{X}_s})(y)\,\D s\\
            & \quad + \int_0^t \left(G_{t-s}^\lambda\star \int_0^s \abs{\phi_{s-r}-G_{s-r}^\lambda}\star \abs{f(Z_r^{\bar{y}})}\,\D r\right)(y)\,\D s        \\
            & \quad + \Lip{f}\int_0^t \left(G_{t-s}^\lambda\star \int_0^s G_{s-r}^\lambda\star \abs{X_r - Z_r^{\bar{y}}}\,\D r\right)(y)\,\D s.
		\end{align*}
		Lemma \ref{lem:Convolution} with $n=0$ and $g(r,\eta)=\abs{X_t(\eta) - Z_t^{\bar{y}}(\eta)}$ implies
		\begin{equation*}
			\int_0^t \bigg(G_{t-s}^\lambda\star \int_0^s G_{s-r}^\lambda\star \abs{X_r- Z_r^{\bar{y}}}\,\D r\bigg)(y)\,\D s = \int_0^t (t-s) (G_{t-s}^\lambda\star \abs{X_s - Z_s^{\bar{y}}})(y)\,\D s.
		\end{equation*}
		Recursively, we obtain for any $N\in\mathbb{N}$ the bound
		\begin{align*}
			\abs{Z_t^{\bar{y}}(y)-X_t(y)}\hspace{-3em} & \hspace{3em}\le \abs{\bar{Z}_t^{\bar{y}}(y)-\bar{X}_t(y)} +\int_0^t(\abs{\phi_{t-s}-G_{t-s}^\lambda}\star\abs{f(Z_s^{\bar{y}})})(y)\,\D s\\
            & \quad +\sum_{n=0}^N\frac{\Lip{f}^{n+1}}{n!}\int_0^t (t-s)^{n}(G_{t-s}^\lambda \star\abs{\bar{Z}_s^{\bar{y}}-\bar{X}_s})(y)\,\D s\\
            & \quad +\sum_{n=0}^N\frac{\Lip{f}^{n+1}}{n!}\int_0^t (t-s)^{n}\bigg(G_{t-s}^\lambda \star\int_0^s \abs{\phi_{s-r}-G_{s-r}^\lambda}\star \abs{f(Z_r^{\bar{y}})}\,\D r\bigg)(y)\,\D s \\
            & \quad +R_N,
		\end{align*}
		with \begin{align*}
			R_N & = \frac{\Lip{f}^{N+2}}{(N+1)!}\int_0^t (t-s)^{N+1}(G_{t-s}^\lambda\star \abs{X_s - Z_s^{\bar{y}}})(y)\,\D s\\
            & \le \sup_{\eta\in\mathbb{R},s\in [0,T]} (\abs{X_s(\eta)} + \abs{Z_s^{\bar{y}}(\eta)})\frac{\Lip{f}^{N+2}}{(N+1)!}\int_0^t (t-s)^{N+1}(G_{t-s}^\lambda\star 1)(y)\,\D s \\
            & \lesssim \frac{\Lip{f}^{N+2}}{(N+1)!}\int_0^t (t-s)^{N+1}\,\D s\lesssim  \frac{\Lip{f}^{N+2}}{(N+2)!}\xrightarrow{N\to\infty} 0,
		\end{align*}
		uniformly in $1\le\lambda<\infty$ by Lemma \ref{lem:UniformBoudnednessSolution} and Theorem 4.3 of \citet{Dalang2011}. This implies the representation
		\begin{align*}
			\abs{Z_t^{\bar{y}}(y)-X_t(y)}&\le \abs{\bar{Z}_t^{\bar{y}}(y)-\bar{X}_t(y)} +\int_0^t(\abs{\phi_{t-s}-G_{t-s}^\lambda}\star\abs{f(Z_s^{\bar{y}})})(y)\,\D s\\
            & \quad +\sum_{n\in\mathbb{N}_0}\frac{\Lip{f}^{n+1}}{n!}\int_0^t (t-s)^{n}(G_{t-s}^\lambda \star\abs{\bar{Z}_s^{\bar{y}}-\bar{X}_s})(y)\,\D s \\
            & \quad +\sum_{n\in\mathbb{N}_0}\frac{\Lip{f}^{n+1}}{n!}\int_0^t (t-s)^{n}\bigg(G_{t-s}^\lambda \star\int_0^s \abs{\phi_{s-r}-G_{s-r}^\lambda}\star \abs{f(Z_r^{\bar{y}})}\,\D r\bigg)(y)\,\D s.
		\end{align*}
		Estimating $t-s\le t$ and using the exponential series, we obtain
		\begin{align*}
			\abs{Z_t^{\bar{y}}(y)-X_t(y)} & \le \abs{\bar{Z}_t^{\bar{y}}(y)-\bar{X}_t(y)} +\int_0^t(\abs{\phi_{t-s}-G_{t-s}^\lambda}\star\abs{f(Z_s^{\bar{y}})})(y)\,\D s\\
            & \quad +\Lip{f} e^{\Lip{f} t}\int_0^t (G_{t-s}^\lambda \star\abs{\bar{Z}_s^{\bar{y}}-\bar{X}_s})(y)\,\D s\\
            & \quad +\Lip{f} e^{\Lip{f} t} \int_0^t \bigg(G_{t-s}^\lambda \star\int_0^s (\phi_{s-r}-G_{s-r}^\lambda)\star \abs{f(Z_r^{\bar{y}})}\,\D r\bigg)(y)\,\D s.
		\end{align*}
		\textbf{Step 2 ($p\ge 1$ using Jensen's inequality).} Note that $G_{t-s}^\lambda(y,\eta)\D\eta$ and $\abs{\phi_{t-s}(y,\eta)-G_{t-s}^\lambda(y,\eta)}\D \eta$ are finite measures (with total mass $= 1$ and $\le2\prob[y]{\tau_\lambda\le t-s}$, respectively) for all $0\le s\le t\le T$ and $y\in\mathbb{R}$ by Lemma \ref{lem:DomainDifferences}. We can use Jensen's inequality to infer for any $p\ge1$ the bound
		\begin{align*}
			\abs{Z_t^{\bar{y}}(y)-X_t(y)}^p&\lesssim_{p,T} \abs{\bar{Z}_t^{\bar{y}}(y)-\bar{X}_t(y)}^p\\
            & \quad +\prob[y]{\tau_\lambda\le t}^{p-1}\int_0^t(\abs{\phi_{t-s}-G_{t-s}^\lambda}\star\abs{f(Z_s^{\bar{y}})}^p)(y)\,\D s\\
            & \quad +\Lip{f}^pe^{p\Lip{f} t}\bigg[\int_0^t (G_{t-s}^\lambda \star\abs{\bar{Z}_s^{\bar{y}}-\bar{X}_s}^p)(y)\,\D s\\
            & \quad +\int_0^t \bigg(G_{t-s}^\lambda \star\bigg(\prob[\MTemptyplaceholder]{\tau_\lambda\le t}^{p-1}\int_0^s \abs{\phi_{s-r}-G_{s-r}^\lambda}\star \abs{f(Z_r^{\bar{y}})}^p\,\D r\bigg)\bigg)(y)\,\D s\bigg] \\
            & \lesssim_{T,\Lip{f}}\abs{\bar{Z}_t^{\bar{y}}(y)-\bar{X}_t(y)}^p +\prob[y]{\tau_\lambda\le t}^{p-1}\int_0^t(\abs{\phi_{t-s}-G_{t-s}^\lambda}\star\abs{f(Z_s^{\bar{y}})}^p)(y)\,\D s\\
            & \quad +\int_0^t (G_{t-s}^\lambda \star\abs{\bar{Z}_s^{\bar{y}}-\bar{X}_s}^p)(y)\,\D s\\
            & \quad +\int_0^t \left(G_{t-s}^\lambda \star\left(\prob[\MTemptyplaceholder]{\tau_\lambda\le t}^{p-1}\int_0^s \abs{\phi_{s-r}-G_{s-r}^\lambda}\star \abs{f(Z_r^{\bar{y}})}^p\,\D r\right)\right)(y)\,\D s.
		\end{align*}\qedhere
	\end{proof}
	\begin{lemma}\label{lem:Aux_Convergence}
		Assume that $f\colon\mathbb{R}\to\mathbb{R}$ is globally Lipschitz-continuous and that $\chi\in H^{\tilde{s}}(\bar{\Lambda})$ for some $\tilde{s}>1/2$. Let $0<\epsilon<\tilde{s}-1/2$, $p\ge 1$ and $y=\lambda\bar{y}\in\Lambda$. We have the following bounds, where the hidden constant in $\lesssim$ does not depend on the time $0\le t\le T$, location $y\in\Lambda$ or domain size $1\le\lambda<\infty$.
		\begin{enumerate}[(a)]
			\item\label{num:Aux_convergence_a} Linear part I:
			      \begin{align*}
				      \EV[f]{\abs{\bar{Z}_t^{\bar{y}}(y)-\bar{X}_t(y)}^p} 
                      & \lesssim_{T,p} t^{p/4}\prob[y]{\tau_\lambda\le t}^{p/2}+t^{p\epsilon/2}\lambda^{-p\epsilon}\norm{\chi}_{H^{\tilde{s}}(\bar{\Lambda})}^p.
			      \end{align*}
			\item\label{num:Aux_convergence_b} Linear part II:
			      \begin{equation*}
				      \EV[f][][\bigg]{\int_0^t (G_{t-s}^\lambda\star \abs{\bar{Z}_s^{\bar{y}}-\bar{X}_s}^p)(y)\D s}
                       \lesssim_{T,p} K(y,\lambda,p) + t^{p\epsilon/2+1}\lambda^{-p\epsilon}\norm{\chi}_{H^{\tilde{s}}(\bar{\Lambda})}^p.
			      \end{equation*}
			\item Remainder I:
			      \begin{equation*}
				      \prob[y]{\tau_\lambda\le t}^{p-1}\EV[f][][\bigg]{\int_0^t \abs{\phi_{t-s}-G_{t-s}^\lambda}\star \abs{f(Z_s^{\bar{y}})}^p)(y)\,\D s}\lesssim_{T,\Lip{f},p} t\prob[y]{\tau_\lambda\le t}^p.
			      \end{equation*}
			\item  Remainder II:
			      \begin{align*}
				      \EV[f][][\bigg]{\int_0^t \left(G_{t-s}^\lambda \star\left(\prob[\MTemptyplaceholder]{\tau_\lambda\le t}^{p-1}\int_0^s \abs{\phi_{s-r}-G_{s-r}^\lambda}\star \abs{f(Z_r^{\bar{y}})}^p\,\D r\right)\right)(y)\,\D s}\hspace{-10em} &\\
                      & \lesssim_{T,\Lip{f},p}K(y,\lambda,2p).
			      \end{align*}
		\end{enumerate}
	\end{lemma}
	\begin{proof}~
		\begin{enumerate}[(a)]
			\item Using the triangle inequality, Gaussian hypercontractivity and Itô's isometry we can write
            \begin{align*}
                \EV[f]{\abs{\bar{Z}_t^{\bar{y}}(y)-\bar{X}_t(y)}^p}&\lesssim_p  \EV[f][][\bigg]{\abs[\bigg]{\int_0^t\int_{\mathbb{R}} (\phi_{t-s}(y,\eta)-G_{t-s}^\lambda(y,\eta)) \mathcal{W}(\D\eta,\D s)}^p}\\
                &\quad+ \abs{X_0(y) - (G_t^\lambda \star X_0)(y)}^p\\
                &\lesssim_p \Big(\int_0^t \int_{\mathbb{R}}(\phi_{t-s}(y,\eta)-G_{t-s}^\lambda(y,\eta))^2\,\D\eta\D s\Big)^{p/2}\\
                &\quad+ \abs{X_0(y) - (G_t^\lambda \star X_0)(y)}^p.
            \end{align*}
            The claim follows Lemmas \ref{lem:DomainDifferences} and \ref{lem:Impact_Initial_Condition}.
			\item Using the triangle inequality, Gaussian hypercontractivity and Itô's isometry we find 
            \begin{align*}
                \EV[f][][\bigg]{\int_0^t (G_{t-s}^\lambda\star \abs{\bar{Z}_s^{\bar{y}}-\bar{X}_s}^p)(y)\,\D s}\hspace{-4em} &\\
                      &\lesssim_p \int_0^t \int_{\Lambda}G_{t-s}^\lambda(y,z)\Big(\int_0^s \int_{\mathbb{R}}(\phi_{s-r}(z,\eta)-G_{s-r}^\lambda(z,\eta))^2\,\D\eta\D r\Big)^{p/2}\D z\D s\\
                      &\quad+ \int_0^t (G_{t-s}^\lambda \star \abs{X_0(y)-G_s^\lambda\star X_0}^p)(y)\,\D s.
            \end{align*}
            By Lemma \ref{lem:Impact_Initial_Condition} the second summand is bounded by $C t^{p\epsilon/2+1}\lambda^{-p\epsilon}\norm{\chi}_{H^{\tilde{s}}(\bar{\Lambda})}^p$ for some constant $0<C<\infty$ depending on $p$. We proceed by bounding the first summand. 
            Take some $h=h(\lambda)>0$ such that $h<\operatorname{diam}(\bar{\Lambda})/2$ and set $\mathcal{C}_h = (B_h(0)+\partial \bar{\Lambda})\cap \bar{\Lambda}$. Using first Itô's isometry, then the Property \eqref{num:Aux_convergence_a} and finally the bound \eqref{eq:Greenfunction_StoppingTimeEstimate}, we compute
			      \begin{align}
				      \int_0^t \int_{\Lambda}G_{t-s}^\lambda(y,z)\Big(\int_0^s \int_{\mathbb{R}}(\phi_{s-r}(z,\eta)-G_{s-r}^\lambda(z,\eta))^2\,\D\eta\D r\Big)^{p/2}\D z\D s\hspace{-25em}&\nonumber\\
                      & \lesssim_{T}\int_0^t  \int_{\Lambda}G_{t-s}^\lambda(y,z)\prob[z]{\tau_\lambda\le T}^{p/2}\,\D z\D s\label{eq:Convergence_Global_aux1}\\
                      & =\lambda\int_0^t  \int_{\bar{\Lambda}}G_{t-s}^\lambda(y,\lambda z)\prob[\lambda z]{\tau_\lambda\le T}^{p/2}\,\D z\D s\nonumber\\
                      & \lesssim_T \lambda\int_0^t  \int_{\bar{\Lambda}\setminus \mathcal{C}_h}G_{t-s}^\lambda(y,\lambda z)\frac{\exp\left(-\frac{p\lambda^2\operatorname{dist}(z,\partial\bar{\Lambda})^2}{8 T}\right)}{\lambda^{p/2}\operatorname{dist}(z,\partial\bar{\Lambda})^{p/2}}\,\D z\D s+ +\lambda\int_0^t  \int_{ \mathcal{C}_h}G_{t-s}^\lambda(y,\lambda z)\,\D z\D s\nonumber  \\
                      & \lesssim_T \lambda \int_0^t  \int_{\bar{\Lambda}\setminus \mathcal{C}_h}G_{t-s}^\lambda(y,\lambda z)\frac{\exp\left(-\frac{p\lambda^2\operatorname{dist}(z,\partial\bar{\Lambda})^2}{8 T}\right)}{\lambda^{p/2}\operatorname{dist}(z,\partial\bar{\Lambda})^{p/2}}\,\D z\D s+\lambda h\nonumber \\
                      & =\colon I + \lambda h\nonumber.
			      \end{align}
			      Recall the scaling of the heat kernel from \eqref{eq:Scaling_Heat_Kernel}. Applying the bound $G_t^1\lesssim_{T} \phi_{ct}$, for $0\le t\le T$ and any constant $c>1$ from Theorem 3.2.9 of \citet{daviesHeatKernelsSpectral1990}, we obtain
			      \begin{align*}
				      I & =\int_0^t  \int_{\bar{\Lambda}\setminus \mathcal{C}_h}G_{\lambda^{-2}(t-s)}^1(\lambda^{-1}y,z)\frac{\exp\left(-\frac{p\lambda^2\operatorname{dist}(z,\partial \bar{\Lambda})^2}{8 T}\right)}{\lambda^{p/2}\operatorname{dist}(z,\partial \bar{\Lambda})^{p/2}}\,\D z\D s                                                                      \\
				        & \lesssim_{T}\lambda\int_0^t  (t-s)^{-1/2}\int_{\bar{\Lambda}\setminus \mathcal{C}_h}\frac{\exp\left(-\frac{\lambda^2(\lambda^{-1}y-z)^2}{4c (t-s)}\right)\exp\left(-\frac{p\lambda^2\operatorname{dist}(z,\partial \bar{\Lambda})^2}{8 T}\right)}{\lambda^{p/2}\operatorname{dist}(z,\partial \bar{\Lambda})^{p/2}}\,\D z\D s \\
				        & \le \lambda^{1-p/2}\int_0^t  (t-s)^{-1/2}\int_{\bar{\Lambda}\setminus \mathcal{C}_h}\frac{\exp\left(-\frac{\lambda^2(\lambda^{-1}y-z)^2}{8c T}\right)\exp\left(-\frac{\lambda^2\operatorname{dist}(z,\partial \bar{\Lambda})^2}{8c T}\right)}{\operatorname{dist}(z,\partial \bar{\Lambda})^{p/2}}\,\D z\D s.                          
			      \end{align*}
			      Since\footnote{Consider two cases: If $a$ and $b$ project to the same boundary point $p$, use triangle inequality. Otherwise, with projection points $p_{a},p_b\in\partial\bar{\Lambda}$, we find $\abs{p_{a}-a}=\operatorname{dist}(a,\partial\bar{\Lambda}) \le \abs{p_b-a}=\abs{p_b-b}+\abs{b-a}$.}
			      \begin{equation*}
				      \operatorname{dist}(a,\partial \bar{\Lambda})\le \abs{a-b} + \operatorname{dist}(b,\partial \bar{\Lambda}),\quad a,b\in\bar{\Lambda},
			      \end{equation*}
			      we find
			      \begin{align*}
				      I & \lesssim \lambda^{1-p/2}\exp\left(-\frac{\lambda^2\operatorname{dist}(\lambda^{-1}y,\partial\bar{\Lambda})^2}{16c T}\right)\int_0^t (t-s)^{-1/2}\D s \int_{\bar{\Lambda}\setminus \mathcal{C}_h}\operatorname{dist}(z,\partial \bar{\Lambda})^{-p/2}\,\D z \\
				        & \lesssim_{T,p} \lambda^{1-p/2}\exp\left(-\frac{\lambda^2\operatorname{dist}(\lambda^{-1}y,\partial \bar{\Lambda})^2}{16c T}\right) \begin{cases}\log(h^{-1}),&p=2,\\ h^{-p/2+1},&p\neq 2.\end{cases}
			      \end{align*}
			      Define $\zeta\coloneq \exp\left(-\frac{\lambda^2\operatorname{dist}(\lambda^{-1}y,\partial \bar{\Lambda})^2}{16c T}\right)$, then the choice $h=\zeta^{2/p}/\lambda$ ensures that
			      \begin{align*}
				      \int_0^t \int_{\Lambda}G_{t-s}^\lambda(y,z)\Big(\int_0^s \int_{\mathbb{R}}(\phi_{s-r}(z,\eta)-G_{s-r}^\lambda(z,\eta))^2\,\D\eta\D r\Big)^{p/2}\D z\D s\hspace{-24em} &\\
                      & \lesssim_{T,p} \lambda h + \lambda^{1-p/2}\zeta \begin{cases}\log(h^{-1}),&p=2,\\h^{1-p/2},&p\neq 2,\end{cases}\\
                      & \lesssim_{T,p} \zeta^{2/p}\begin{cases}(1 + \log(\lambda) + \operatorname{dist}(\lambda^{-1}y,\partial\bar{\Lambda})^2),&p=2\\ 2,& p\neq 2,\end{cases} \\
                      & =K(y,\lambda,p).
			      \end{align*}
			\item Using Lemma \ref{lem:DomainDifferences} in the second estimate we compute
			      \begin{align*}
				      \prob[y]{\tau_\lambda\le t}^{p-1}\EV[f][][\bigg]{\int_0^t (\abs{\phi_{t-s}-G_{t-s}^\lambda}\star \abs{f(Z_s^{\bar{y}})}^p)(y)\,\D s}\hspace{-17em} &\\
                      & \lesssim \prob[y]{\tau_\lambda\le t}^{p-1}\sup_{\eta\in\mathbb{R}, s\in[0,T]}\EV[f]{\abs{f(Z_s^{\bar{y}}(\eta))}^p}\int_0^t\int_{\mathbb{R}}\abs{\phi_{t-s}(y,\eta)-G_{t-s}^\lambda(y,\eta)}\,\D\eta\D s \\
                      & \le 2\prob[y]{\tau_\lambda\le t}^{p-1}\sup_{\eta\in\mathbb{R}, s\in[0,T]}\EV[f]{\abs{f(Z_s^{\bar{y}}(\eta))}^p}\int_0^t\prob[y]{\tau_\lambda\le t-s}\,\D s\\
                      & \lesssim\sup_{\eta\in\mathbb{R}, s\in[0,T]}\EV[f]{\abs{f(Z_s^{\bar{y}}(\eta))}^p}t\prob[y]{\tau_\lambda\le t}^p.
			      \end{align*}
			      The claim follows by noting $\sup_{\eta\in\mathbb{R}, s\in[0,T]}\EV[f]{\abs{f(Z_s^{\bar{y}}(\eta))}^p}\le \tilde{C}$ for some constant $0<\tilde{C}<\infty$ depending on $T$, $\Lip{f}$, $p$ and $\norm{\chi}_{L^\infty(\bar{\Lambda})}$ by Theorem 4.3 of \citet{Dalang2011}.
			\item Using Lemma \ref{lem:DomainDifferences} and Theorem 4.3 of \citet{Dalang2011} in the second estimate we compute
			      \begin{align*}
				      \EV[f][][\bigg]{\int_0^t \left(G_{t-s}^\lambda \star\left(\prob[\MTemptyplaceholder]{\tau_\lambda\le t}^{p-1}\int_0^s \abs{\phi_{s-r}-G_{s-r}^\lambda}\star \abs{f(Z_r^{\bar{y}})}^p\,\D r\right)\right)(y)\,\D s}\hspace{-22em} & \\
                      & \le \sup_{\eta\in\mathbb{R}, s\in[0,T]}\EV[f]{\abs{f(Z_s^{\bar{y}}(\eta))}^p}\\
                      & \quad \int_0^t \int_{\Lambda} G_{t-s}^\lambda(y,z)\prob[z]{\tau_\lambda\le t}^{p-1} \int_0^s \int_{\mathbb{R}}\abs{\phi_{s-r}(z,\eta)-G_{s-r}^\lambda(z,\eta)}\,\D\eta \D r\D z\D s \\
                      & \lesssim_{T,\Lip{f},p,\norm{\chi}_{L^\infty(\bar{\Lambda})}}\int_0^t \int_{\Lambda} G_{t-s}^\lambda(y,z)\prob[z]{\tau_\lambda\le s}^p\, \D z\D s.
			      \end{align*}
			      Proceeding exactly as in the proof of Property \eqref{num:Aux_convergence_b} with $2p$ instead of $p$ (compare \eqref{eq:Convergence_Global_aux1}), we find
			      \begin{equation*}
				      \EV[f][][\bigg]{\int_0^t \left(G_{t-s}^\lambda \star\left(\prob[\MTemptyplaceholder]{\tau_\lambda\le t}^{p-1}\int_0^s \abs{\phi_{s-r}-G_{s-r}^\lambda}\star \abs{f(Z_r^{\bar{y}})}^p\,\D r\right)\right)(y)\,\D s} \lesssim_{T,\Lip{f},p} K(y,\lambda,2p).
			      \end{equation*}
		\end{enumerate}
	\end{proof}

	\subsubsection{Proof of Proposition \ref{prop:Convergencetoglobal}}\label{subsec:Proof_Convergence_global}

	\begin{proof}[Proof of Proposition \ref{prop:Convergencetoglobal}]
		Using first Lemma \ref{lem:Expansion/Gronwall} and then Lemma \ref{lem:Aux_Convergence} implies
		\begin{align*}
			\EV[f]{\abs{Z_t^{\bar{y}}(y)-X_t(y)}^p}\hspace{-6em} & \hspace{6em}\lesssim_{p,T,\Lip{f}}\EV[f]{\abs{\bar{Z}_t^{\bar{y}}(y)-\bar{X}_t(y)}^p}\\
            &\quad +\EV[f][][\bigg]{\int_0^t (G_{t-s}^\lambda \star\abs{\bar{Z}_s^{\bar{y}}-\bar{X}_s}^p)(y)\,\D s}\\
            & \quad + \prob[y]{\tau_\lambda\le t}^{p-1}\EV[f][][\bigg]{\int_0^t(\abs{\phi_{t-s}-G_{t-s}^\lambda}\star\abs{f(Z_s^{\bar{y}})}^p)(y)\,\D s}\\
            & \quad+\EV[f][][\bigg]{\int_0^t \left(G_{t-s}^\lambda \star\left(\prob[\MTemptyplaceholder]{\tau_\lambda\le t}^{p-1}\int_0^s \abs{\phi_{s-r}-G_{s-r}^\lambda}\star \abs{f(Z_r^{\bar{y}})}^p\,\D r\right)\right)(y)\,\D s} \\
            & \lesssim_{T,\Lip{f},p} t^{p/4}\prob[y]{\tau_\lambda\le t}^{p/2}+t^{p\epsilon/2}\lambda^{-p\epsilon}\norm{\chi}_{H^{\tilde{s}}(\bar{\Lambda})}^p+t \prob[y]{\tau_\lambda\le t}^p\\
            &\quad +t^{p\epsilon/2+1}\lambda^{-p\epsilon}\norm{\chi}_{H^{\tilde{s}}(\bar{\Lambda})}^p+K(y,\lambda,p)+K(y,\lambda,2p)\\
            & \lesssim_{T} \prob[y]{\tau_\lambda\le t}^{p/2}+ K(y,\lambda,p)+K(y,\lambda,2p)+\lambda^{-p\epsilon}\norm{\chi}_{H^{\tilde{s}}(\bar{\Lambda})}^p,
		\end{align*}
        which shows \eqref{eq:ConvergenceGlobal_Claim}. Given the bound \eqref{eq:ConvergenceGlobal_Claim}, the claimed convergence of $X_t(\lambda\bar{y})$ to $Z_t^{\bar{y}}(0)$ follows from $Z_t^{\bar{y}}(y)\overset{d}{=}Z_t^{\bar{y}}(0)$ for all $y\in\mathbb{R}$ and $0\le t\le T$ by Lemma 7.1 of \citet{chenSpatialErgodicitySPDEs2021}.
	\end{proof}

\printbibliography

\end{document}